\documentclass[a4paper]{amsart}

\newtheorem{theorem}{Theorem}[section]
\newtheorem{lemma}[theorem]{Lemma}
\newtheorem{corollary}[theorem]{Corollary}               

\theoremstyle{definition}
\newtheorem{definition}[theorem]{Definition}
\newtheorem{example}[theorem]{Example}
\newtheorem{assumption}[theorem]{Assumption}

\theoremstyle{remark}
\newtheorem{remark}[theorem]{Remark}

\numberwithin{equation}{section}
\usepackage{graphicx}              
\usepackage{amsmath}               
\usepackage{amsfonts}              
\usepackage{amsthm}
\usepackage{amsmath}
\usepackage{stmaryrd} 
\usepackage{chemarr, graphicx,  subfigure} 
\usepackage[usenames, dvipsnames]{color}
\usepackage{pict2e}
\usepackage{picture}

\usepackage{comment}
\usepackage{grffile}
\usepackage{xcolor}

\usepackage[normalem]{ulem}
\newcommand{\blue}[1]{\textcolor{blue}{#1}}

\newcommand{\Rev}[1]{\textcolor{black}{#1}}

\newcommand{\Revthree}[1]{\textcolor{black}{#1}}

\newcommand{\vertiii}[1]{{\|\kern-0.25ex | #1 
		| \kern-0.25ex \|}}

\newcommand{\ddd}{\text{\rm D}}
\newcommand{\dn}{\text{\rm N}}
\newcommand{\dint}{\text{\rm int}}
\newcommand{\vecb}{{\bf b}}
\newcommand{\vecn}{{\bf n}}
\newcommand{\mean}[1]{\{\kern-1.1mm\{#1\}\kern-1.1mm\}}                  
\newcommand{\jump}[1]{[\![#1]\!]}                        
\newcommand{\uu}[1]{\mathbf{#1}}
\newcommand{\ltwo}[2]{\|{#1}\|_{#2}}
\newcommand{\linf}[2]{\|{#1}\|_{L_{\infty}({#2})}}

\newcommand{\ud}{\,\mathrm{d}}
\newcommand{\ndg}[1]{| \kern -.25mm \|{#1}| \kern -.25mm \|}
\newcommand{\nsdg}[1]{| \kern -.25mm \|{#1}| \kern -.25mm \|_{\rm s}}

\newcommand{\vx}{{\tiny\textbullet } }
\newcommand{\ujump}[1]{\lfloor #1\rfloor}   

\newcommand{\vecp}{{\bf p}}

\newcommand{\fes}{S_\mathcal{T}^{\vecp} }

\newcommand{\bgrad}{\nabla_{\mathcal{T}}}

\newcommand{\diam}{\operatorname{diam}}

\newcommand{\coveringmesh}{\mathcal{T}^{\sharp}}
\newcommand{\mesh}{\mathcal{T}}

\newcommand{\norm}[2]{\|{#1}\|_{{#2}}}
\newcommand{\ncdg}[1]{| \kern -.25mm \|{#1}| \kern -.25mm \|_{\rm DG}}

\newcommand{\m}{{\bf m}_i}

\title[$hp$-DG methods on essentially arbitrarily-shaped elements]{$\boldsymbol{\lowercase{hp}}$-Version discontinuous Galerkin methods \\ on essentially arbitrarily-shaped elements}

\author{Andrea Cangiani} \address{
	SISSA, via Bonomea 265, 34136 Trieste, Italy} \email{Andrea.Cangiani@sissa.it}
\author{Zhaonan Dong} \address{
	1) Inria, 2 rue Simone Iff, 75589 Paris, France, 
	and 2) CERMICS, Ecole des Ponts, 77455 Marne-la-Vall\'{e}e 2, France }  \email{Zhaonan.Dong@inria.fr}
\author{Emmanuil H.~Georgoulis} \address{
	1) School of Mathematics and Actuarial Science,
	University of Leicester,
	University Road,
	Leicester, LE1 7RH,
	United Kingdom, 
	2) Department of Mathematics, School of Applied Mathematical and Physical Sciences, National Technical University of Athens, Zografou 15780, Greece, and
	3) IACM-FORTH, Crete, Greece} \email{Emmanuil.Georgoulis@le.ac.uk}

\begin{document}
	
	\begin{abstract} 
		We extend the applicability of the popular interior-penalty discontinuous Galerkin (dG) method discretizing advection-diffusion-reaction problems to meshes comprising extremely general, essentially arbitrarily-shaped element shapes. In particular, our analysis allows for \emph{curved} element shapes, without the use of  {non-linear} elemental maps. The feasibility of the method relies on the definition of a suitable choice of the discontinuity penalization, which turns out to be  {explicitly dependent} on the particular element shape,  {but essentially independent on small shape variations}.  {This is achieved upon proving extensions of classical trace and Markov-type inverse estimates to arbitrary element shapes. A further new $H^1-L_2$-type inverse estimate on essentially arbitrary element shapes enables the proof of inf-sup stability of the method in a streamline-diffusion-like norm. These inverse estimates may be of independent interest.} A priori error bounds for the resulting method are given under very mild structural assumptions restricting the magnitude of the local curvature of element boundaries. Numerical experiments are also presented, indicating the practicality of the proposed approach. 
	\end{abstract}

	\maketitle
	
	\section{Introduction}

	Recent years have witnessed a coordinated effort to generalize mesh concepts in the context of Galerkin/finite element methods.  {A} key argument has been that more general-shaped elements/cells can potentially lead to computational complexity reduction. This effort has given rise to a number of recent approaches: mimetic finite difference methods \cite{MimeticBook2014}, virtual element methods \cite{VEM6,brenner_vem}, various discontinuous Galerkin approaches such as interior penalty \cite{book}, hybridized DG \cite{HDG2009} and the related \Rev{hybrid} high-order methods \cite{DIPIETRO20151}. Earlier approaches involving non-polynomial approximation spaces, such as polygonal and other generalized finite element methods \cite{Tabarraei:Sukumar:2004,Fries:Belytschko:2009}, have also been developed and used by the engineering community. All the above numerical frameworks allow for polygonal/polyhedral element shapes (henceforth, collectively termed as \emph{polytopic}) of varying levels of generality.

	{Simultaneously, various classes of fitted and unfitted grid methods for interface or transmission problems exploit generalized concepts of mesh elements in an effort to provide accurate representations of internal interfaces. 
		Several unfitted finite element methods have been proposed in recent years: unfitted finite element methods \cite{Barrett87}, \Rev{immersed} finite element methods \cite{MR3903560,MR2377272}, virtual element methods \cite{MR3606231}, unfitted penalty methods \cite{Engwer,Massjung,MR3668542,WX19,CGS18}, see also  \cite{Johansson2013} for unfitted discretization of the boundary, cutCell/cutFEM \cite{MR2899249,MR3416285}, and unfitted hybrid high-order methods \cite{MR3809538}, to name just the few closer to the developments we shall be concerned with below. A central idea in the majority of these methods is the weak imposition of interface conditions in conjunction with some form of penalization, see, e.g., \cite{HansboHansbo}, an idea going back to~\cite{Babuska70}. These approaches are often combined with level set concepts~\cite{MR1939127} to describe the interfaces accurately. Nonetheless, in their practical implementation, the interface is typically represented via piecewise smooth polynomial approximations to the level sets.
	}

	The interior penalty discontinuous Galerkin (IP-dG) approach appears to allow for  {extreme generality with regard to element shapes/geometries. Indeed, in contrast to aforementioned families of general mesh methods, IP-dG can handle arbitrary number of faces per element with solid theoretical backing involving provable stability and convergence results; see \cite{cangiani2017hp,book} for details. This property becomes extremely relevant upon realising that IP-dG (as well as other classical dG methods, such as LDG) associate local numerical degrees of freedom to the elements \emph{only}, and not to other geometrical entities such as faces or vertices. As such, the nature and dimension of the local discretization space is independent of the number of vertices/faces. The latter observation implies also naturally a form of complexity reduction: classical total degree (`$P-$type') local polynomial spaces in physical coordinates are admissible on box-type or highly complex element shapes \cite{cangiani2013hp,cangiani2015hp,book}. 
		We refer to our monograph \cite{book} for details on the admissible polygonal/polyhedral element shapes for which the IP-dG method is, provably, both stable and convergent. The mild element shape assumptions in \cite{book} are such to ensure the validity of crucial generalizations of standard approximation results, such as inverse estimates, best approximation estimates, and extension theorems. Thus, the developments presented for IP-dG in \cite{book} can be potentially ported also to other classical dG approaches within the unified framework of~\cite{ABCM}; we also note the recent static condensation approach presented in \cite{Lozinski} in this context.}
	
	{The question, therefore, of further extending rigorously the applicability of $hp$-version IP-dG methods to meshes consisting of \emph{curved} polygonal/polyhedral elements arises naturally. Indeed, such a development is expected to provide multifaceted advantages compared to current approaches, including, but not restricted to, the treatment of curved interfaces as done, e.g., in \cite{Massjung,MR3668542,WX19,MR3809538,CGS18}. For instance, allowing for extremely general curved elements enables the \emph{exact} representation of curved computational domains, e.g., arising directly from Computer Aided Design programs. Allowing also for arbitrary local polynomial degrees, provides the possibility of achieving required accuracy via local (polynomial) basis enrichment ($p$-version Galerkin approaches) without increased mesh-granularity. If, nonetheless, local mesh refinement is also required/desired, IP-dG methods can be immediately applied on refined curved elements without local re-parametrizations of the local Galerkin spaces. This is in contrast to the need to perform costly re-parametrisations upon mesh refinement in other approaches, e.g., Isogeometric Analysis \cite{hennig2017adaptive} or, indeed, even to keep track of the domain-approximation variational crimes of standard finite element discretizations. Exact geometry representation can also be highly relevant in representing locally discontinuous/sharply changing PDE coefficients, e.g., in permeability pressure computations in porous media, coefficients defined via level-sets of smooth functions, or shape/topology optimization applications.}
	
	{
		Furthermore, exact geometry representation is relevant in the $p$-version Galerkin context: to achieve spectral/exponential convergence for smooth PDE problems posed on general curved domains, we are required to use isoparametrically mapped elements. This is both cumbersome to implement and  costly as the polynomial degree increases \cite{murti1986numerical,murti1988numerical}. A successful alternative to isoparametric maps is the use of non-linear maps on element patches \cite{schwab,melenk_book} to represent domain geometry. Nevertheless, if the elemental maps are not \emph{a priori} provided, it is difficult to construct them in practice, especially in three dimensions. }
	
	{Finally, curved element capabilities should ideally be developed \emph{in conjunction} with the already developed highly general polytopic mesh IP-dG methods, allowing for instance elements with \emph{arbitrary} number of faces. This is particularly pertinent in the contexts of adaptivity and multilevel solvers, which benefit from element agglomeration \cite{hp_multigrid_polytopes_2017,Antonietti_2019} to achieve coarser representations. With regard to adaptivity, mesh coarsening is essential in keeping the computation sizes at bay, at least in the case of evolution problems. The extreme coarsening capabilities via element agglomeration, therefore, have the potential in retaining structure, e.g., possible coefficient heterogeneities at the discrete level for instance.} 
	
	{ It is, therefore, desirable to design and analyze IP-dG and related methods posed on meshes comprising of elements with \emph{arbitrary} number of \emph{curved} faces, under as mild geometric assumptions as possible. To address this central, in our view, question, this work aims at rigorously extending the applicability of IP-dG methods on meshes comprising of essentially arbitrarily \emph{curve-shaped} polytopic elements with arbitrary number of faces per element; this includes, in particular, curved elements \emph{not} exactly representable by (iso-)parametric polynomial element mappings.}

	{The theoretical developments presented below regarding stability and a-priori error analysis of IP-dG methods hinge on new, to the best of our knowledge, extensions of known inverse and trace inequalities. More specifically, we extend the $hp$-version trace inverse estimate presented in \cite{CGS18}, allowing for more general curved element shapes; see also \cite{MR3668542} for an earlier, related result.}  {Trace inverse estimates are crucial in the proof of stability of IP-dG methods and, simultaneously, determine the so-called discontinuity-penalization parameter for a given mesh. This is crucial on meshes of such generality: insufficient penalization results \Rev{in} loss of stability, while excessive penalization typically results \Rev{in} accuracy loss. Also, we prove new $hp$-version $L_\infty-L_2$ and $H^1-L_2$ inverse inequalities on extremely general curved domains. Particular care has been given so that these new inverse estimates are `shape-robust', in the sense that there is no hidden dependence of the element shape in the constants. We believe that these extensions of known inverse estimates to be of independent interest, due to their frequent use in the analysis of finite element methods. }
	
	{The new inverse estimates are combined with ideas from the analysis of polytopic dG methods \cite{book}, resulting in significant generalization of the results presented therein. More specifically, by relaxing certain earlier coverability assumptions, (postulating the ability to cover tightly general-shaped elements by unions of simplices of similar size, cf. \cite[Definition 10]{book}) as well as by proving a new stability result for norms of polynomials under domain perturbations (Lemma \ref{new_chop}), we prove stability and a new $hp$-version a priori error analysis for the IP-dG method on essentially arbitrary element shapes. The a priori error analysis follows closely the proof from \cite{cangiani2015hp}: upon establishing an inf-sup stability result of the method in a streamline-diffusion-like norm, standard Strang-type arguments with $hp$-best approximation results lead to an error bound. The inf-sup result justifies also the good stability properties of the method in convection-dominated problems.  The theoretical tools presented may also be of interest in Nitsche-type formulations of unfitted grid interface methods.} To  {emphasize} the mesh-generality of the proposed approach, we shall refer to the framework presented below as \emph{discontinuous Galerkin method on essentially arbitrarily-shaped elements (dG-EASE)}.

	The remainder of this work is organised as follows. Upon describing the advection-diffusion-reaction model problem in Section \ref{Problem}, we introduce the $hp$-version interior penalty discontinuous Galerkin method in Section \ref{sec:dg}.  We prove new inverse estimates in Section \ref{sec:approx}, along with the necessary $hp$-approximation results. In Section \ref{sec:aprior}, we present the stability  and a-priori error analysis. Finally, the performance of the dG methods is assessed in practice through a series of numerical experiments presented in Section \ref{numerical example}.

	\section{Model problem} \label{Problem}
	To highlight the versatility of dG-EASE, we consider the class of second--order partial differential equations with 
	nonnegative characteristic form over an open bounded Lipschitz domain $\Omega$ in $\mathbb{R}^d$, $d\geq 1$, with boundary $\partial\Omega$. This class includes general advection-diffusion-reaction problems possibly of changing type, see, e.g.,~\cite{book}. The model problem reads:
	find $u\in\mathcal{V}$ such that
	\begin{equation}\label{pde}
		\begin{aligned}
			-\nabla \cdot (a \nabla u)
			+ \nabla \cdot(\vecb u) + cu =& f ~~\mbox{ in } \Omega,
		\end{aligned}
	\end{equation}
	for some suitable solution space $\mathcal{V}$, and $a = \left\{ a_{ij} \right\}_{i,j=1}^d$, symmetric with
	$a_{ij} \in L_{\infty}(\Omega)$, so that at each $\uu{x}$ in $\bar\Omega$, we have
	\begin{equation}\label{eq:anonneg}
		\begin{aligned}
			\sum_{i,j=1}^d a_{ij}(\uu{x})\xi_i \xi_j \geq 0, \qquad \text{for any}\quad {\bf\xi} = (\xi_1,\dots,\xi_d)^T\in\mathbb{R}^d;
		\end{aligned}
	\end{equation}
	also $\vecb =(b_1,\ldots,b_d)^T \in \left[ W^{1}_{\infty}(\Omega) \right]^d$,
	$c \in L_{\infty}(\Omega)$ and $f\in L_2(\Omega)$.
	
	To supplement \eqref{pde} with suitable boundary conditions,
	following \cite{or73}, we first subdivide the boundary $\partial\Omega$ into 
	$
	\partial_0\Omega = \Big\{ \uu{x} \in \partial\Omega : \sum_{i,j=1}^d a_{ij}(\uu{x}) n_i n_j > 0
	\Big\}$, 
	and $\partial\Omega\backslash \partial_0\Omega$ with $\vecn = (n_1,\ldots,n_d)^T$ denoting the unit
	outward normal vector to $\partial\Omega$. Loosely speaking, we may think of $\partial_0\Omega$
	as being the `elliptic' portion of the boundary $\partial\Omega$.
	We further split the `hyperbolic' portion of the boundary $\partial\Omega\backslash\partial_0\Omega$,
	into inflow and outflow boundaries $\partial_-\Omega$ and $\partial_+\Omega$, 
	respectively, by
	\begin{equation*}
		\partial_-\Omega = \left\{  \uu{x} \in \partial\Omega\backslash\partial_0\Omega : \vecb(\uu{x}) 
		\cdot \vecn (\uu{x})< 0 \right\} ,\   
		\partial_+\Omega =\left\{ \uu{x} \in \partial\Omega\backslash\partial_0\Omega : \vecb(\uu{x}) 
		\cdot \vecn (\uu{x}) \geq 0 \right\}. 
	\end{equation*}
	If $\partial_0\Omega$ is nonempty, we shall further divide it into two disjoint subsets 
	$\partial\Omega_{\ddd}$ and $\partial\Omega_{\dn}$, with 
	$\partial\Omega_{\ddd}$ nonempty and relatively open in $\partial\Omega$. It is evident
	from these definitions that 
	$\partial\Omega=\partial\Omega_{\ddd}\cup\partial\Omega_{\dn}\cup\partial_-\Omega\cup\partial_+\Omega$. 
	
	It is physically reasonable to assume that $\vecb \cdot \vecn \geq 0$ on
	$\partial\Omega_{\dn}$, whenever $\partial\Omega_{\dn}$ is nonempty; then, we impose
	the boundary conditions:
	\begin{equation}
		\begin{aligned}
			u = g_{\ddd}^{} ~~~\text{ on }\partial\Omega_{\ddd}\cup\partial_-\Omega, \qquad
			{\bf n} \cdot (a\nabla u) =g_{\dn} ~~~\text{ on }\partial\Omega_{\dn}; \label{pde_bcs}\end{aligned}
	\end{equation}
	For an extension, allowing also for  $\vecb \cdot \vecn < 0$ 
	on $\partial\Omega_{\dn}$, we refer to~\cite{CGJ}. Additionally, we assume that there exists a positive constant $\gamma_{0}$ such that
	\begin{equation}\label{assumption-cb}
		c_0(\uu{x}):=\Big(c(\uu{x})+\frac{1}{2}\nabla\cdot \vecb(\uu{x})\Big)^{1/2}\geq \gamma_0 \quad \text{ a.e. } x\in\Omega.
	\end{equation} 
	For a proof of the well--posedness of \eqref{pde}, \eqref{pde_bcs}, subject to \eqref{assumption-cb}, we refer to \cite{or73,houston2001stabilised}.

	\section{Discontinuous Galerkin method}\label{sec:dg}
	We shall now define the interior penalty discontinuous Galerkin (dG) method posed on essentially arbitrarily-shaped elements. A key attribute of the method is the use of physical frame basis functions, i.e., the elemental bases consist of polynomials on the elements themselves, rather than mapped from a reference element. The implementation challenges arising from this non-standard choice with regard to construction of the resulting linear system will be discussed below.
	
	\subsection{The mesh} 
	Let $\mathcal{T}=\{ K\} $ be a subdivision of $ \Omega$ into non-overlapping subsets (elements) $ K\in\mathcal{T}$ with, possibly curved, Lipschitz boundaries and let $h_K:=\diam(K)$.  The mesh skeleton $ \Gamma:=\cup_{K\in\mathcal{T}}\partial K $ is subdivided into the internal part $\Gamma_{\dint}:=\Gamma\backslash  \partial\Omega $ and boundary part $\partial\Omega$. We further explicitly assume that the $(d-1)$-dimensional Hausdorff measure of $\Gamma$ is globally finite, thereby, not allowing for fractal-shaped elements.
	
	We note immediately that we allow mesh elements $K\in\mathcal{T}$ which are essentially arbitrarily-shaped and with very general interfaces with neighbouring elements. For instance, two elements may  {interface at a collection of $(d-1)$-dimensional (possibly curved) \emph{faces},}
	as those shown in Figure~\ref{fig:threeelements}.
	The precise assumptions on the admissible element shapes are given  {in Section \ref{sec:approx}} below.
	
	\begin{figure}[h]
		\centering
		\includegraphics[height=3cm,width=8cm]{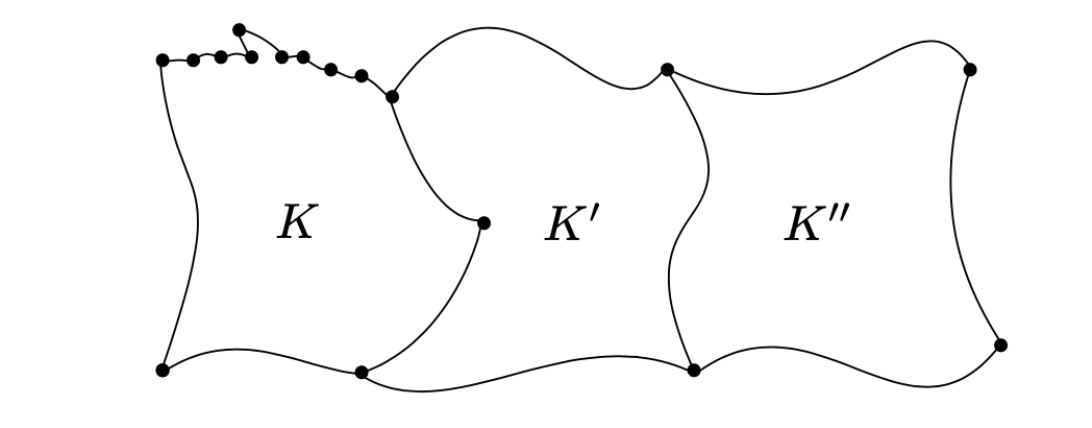}	
			\caption{Curved elements $K,K',K''$ for $d=2$ with possibly many curved faces; \vx denotes a vertex.}
		\label{fig:threeelements}
	\end{figure}

	\subsection{Discontinuous Galerkin method}
	
	We define the $hp$-version discontinuous finite element space $ \fes$, subordinate to the mesh $\mathcal{T}=\{K\}$ and a polynomial degree vector $\vecp:=\{p_K\}$, possibly different for each element $K$, by
	\begin{equation}
		\fes=\{  v\in L^{2}(\Omega):v\vert_K\in\mathcal{P}_{p_K}(K),\ K\in\mathcal{T}\}.
	\end{equation} 
	
	
	For any elemental face $F\subset \Gamma_\text{int}$, let $K$ and $K'$ be the two elements such that $F\subset\partial K\cap\partial K'$. The outward unit normal vectors on $F $ of $ \partial K $ and $ \partial K' $ are denoted by $ \mathbf{n}_{K} $ and $ \mathbf{n}_{K'} $, respectively. For a function $ v:\Omega\rightarrow \mathbb{R} $ that may be discontinuous across $ \Gamma $, we define the jump  $\llbracket v \rrbracket $ and  the average $ \lbrace v\rbrace $ of $ v $  across $F$ by
	\begin{equation}
		\label{equ3}
		\llbracket v \rrbracket = v\vert_{K} \mathbf{n}_{K}+v\vert_{K'} \mathbf{n}_{K'}, \quad\lbrace v\rbrace=\frac{1}{2}\left( v\vert_{K}+v\vert_{K'}\right).        
	\end{equation}
	Similarly, for a vector valued function $\mathbf{w}$, piecewise smooth on $\mathcal{T}$, we define 
	\begin{equation*}
		\llbracket \mathbf{w} \rrbracket = \mathbf{w}\vert_{K}\cdot \mathbf{n}_{K}+\mathbf{w}\vert_{K'}\cdot \mathbf{n}_{K'}, \quad \lbrace \mathbf{w}\rbrace=\frac{1}{2}\left( \mathbf{w}\vert_{K}+\mathbf{w}\vert_{K'}\right).
	\end{equation*} When $ F\subset \partial\Omega$, we set $ \lbrace v\rbrace=v$, $\llbracket v \rrbracket =v\mathbf{n} $ and $\llbracket \mathbf{w} \rrbracket=\mathbf{w}\cdot\mathbf{n}$ with $\mathbf{n}$ denoting the outward unit normal to the boundary $\partial\Omega$.

	For any element $K\in \mathcal{T}$,  we define  inflow and outflow parts of $\partial K$ by
	\begin{eqnarray*}
		\partial_{-}K=\{\uu{x}\in \partial K:\  \bold{b}(\uu{x})\cdot{\bold{n}_K(\uu{x})}<0\} ,  \quad \partial_{+}K=\{\uu{x}\in \partial K: \bold{b}(\uu{x})\cdot{\bold{n}_K(\uu{x})} \geq 0\},
	\end{eqnarray*}
	respectively, with $\bold{n}_K(\uu{x})$ denoting the unit outward normal vector to $\partial K$ 
	at $\uu{x} \in \partial K$. 
	Further, we define the \emph{upwind jump} of the (scalar-valued) function $v$ across the inflow boundary  $\partial_-K $ of $K\in\mathcal{T}$ by
	\begin{equation*}
		\ujump{v}(\uu{x}):=\lim_{\epsilon\to 0^+}\Big(v(\uu{x}+\vecb(\uu{x})\epsilon)-v(\uu{x}-\vecb(\uu{x})\epsilon)\Big),\quad\text{when } \uu{x}\in\partial_-K \backslash \partial\Omega.
	\end{equation*}

	Finally, we define the broken gradient $\bgrad v$ of a function $v\in L_2(\Omega)$ with $v|_K\in H^1(K)$, for all $K\in\mathcal{T}$,  element-wise by $\big(\bgrad v\big)|_K := \nabla (v|_K)$.

	The \emph{discontinuous Galerkin method on essentially arbitrarily-shaped elements} (dG-EASE for short) reads: find $u_h\in \fes$ such that
	\begin{equation}\label{adv_galerkin_dg}
		B(u_h,v_h)=\ell(v_h) \qquad\text{for all }\quad v_h\in \fes,
	\end{equation}
	with the bilinear form $B(\cdot, \cdot ):\fes\times \fes\to \mathbb{R}$ defined as \begin{equation*}\label{adv-dg-bilinear}
		B(u,v) := B_{\rm ar}(u,v)+ B_{\rm d}(u,v) ,
	\end{equation*}
	where $B_{\rm ar}(\cdot, \cdot )$ accounts for the advection and reaction terms:
	\begin{equation}\label{advection_bilinear}
		\begin{aligned}
			B_{\rm ar}(u,v) :=& \int_\Omega \big(  \bgrad(\vecb \cdot  u) +cu  \big)v \ud \uu{x}  -\sum_{K\in\mathcal{T}}
			\int_{\partial_-K \backslash \partial\Omega}(\bold{b} \cdot  \bold{n})\ujump{u}v\ud s\\
			&-\sum_{K\in\mathcal{T}}\int_{\partial_-K \cap(\partial\Omega_{\rm D} \cup \partial_-\Omega)}(\bold{b} \cdot \bold{n})uv\ud s,
		\end{aligned}
	\end{equation}
	and $B_{\rm d}(\cdot, \cdot )$ corresponds to the diffusion part: 
	\begin{equation}\label{diffusion_bilinear}
		\begin{aligned}
			B_{\rm d}(u,v) :=&  \int_\Omega  a\bgrad u \cdot \bgrad v \ud \uu{x} +\int_{\Gamma_\text{int}\cup \partial\Omega_\text{D}} \!\!\! \sigma \jump{u} \cdot \jump{v} \ud{s} \\ 
			&- \int_{\Gamma_\text{int}\cup \partial\Omega_\text{D}} \Big( \mean{a \nabla u }\cdot \jump{v} + \mean{a \nabla v }\cdot \jump{u} \Big) \ud s,
		\end{aligned}
	\end{equation}
	while the linear functional $\ell:\fes\to\mathbb{R}$ is defined by
	\begin{equation} \label{adv-dg-linear}
		\begin{aligned}
			\ell(v)
			:=& \int_{\Omega} f  v \ud \uu{x}
			-\sum_{K\in\mathcal{T}}\int_{\partial_-K \cap(\partial\Omega_{\rm D} \cup \partial_-\Omega)}\!\!\!(\bold{b} \cdot \bold{n}) g_{\rm D}^{}v\ud s  \\
			&-\int_{\partial\Omega_{\rm D}} g_{\rm D}^{} \big( (a \nabla v ) \cdot \bold{n} -\sigma v \big)\ud s  +\int_{\partial\Omega_{\rm N}} g_{\rm N}^{}v \ud s.
		\end{aligned}
	\end{equation}
	The nonnegative function  $\sigma \in L_\infty(\Gamma_{\rm int} \cup \partial\Omega_{\rm D})$ appearing in~\eqref{diffusion_bilinear} and~\eqref{adv-dg-linear} is the \emph{discontinuity-penalization  {function}}, whose precise definition, which depends on the diffusion tensor $a$ and the discretization parameters,  will be given below. We note that a `good' choice of discontinuity penalization is instrumental for the stability of the method, while simultaneously not affecting the approximation properties in the general mesh setting considered herein. 
	
	For simplicity of presentation, we shall assume that the entries of the diffusion tensor $a$ are element-wise constants on each element $K\in \mathcal{T}$, i.e.,
	\begin{equation}\label{assumption_a}
		a\in [S^{\bold{0}}_\mathcal{T}]^{d\times d}_{\rm sym},
	\end{equation}
	Our results can \Rev{ be applied} to the case of general $a\in [H^{1/2}(\Omega)]^{d\times d}_{\rm sym}$
	by slightly modifying the bilinear form above as proposed originally in \cite{georgoulis2006note} and extended to polytopic meshes in \cite{book}. In the following, $\sqrt{a}$ denotes the (positive semidefinite) square-root of the symmetric matrix $a$; further, $\bar{a}_K:=|\sqrt{a}|_2^2 |_K$, where $|\cdot |_2$ denotes the matrix-$2$--norm. 
	{Also, in the interest of accessibility, we shall not consider problems with high contrast diffusion tensors, with the usual weighted averaging modification of the method ~\cite{MR2257119,MR2491426,MR2383212}; the extension to that setting is completely analogous to the analysis presented below.}
	{
		\begin{remark}
			The parameter $\sigma$ is typically selected to be face-wise constant in the definition and implementation of IP-dG methods. To ensure that only physically correct penalization takes place, $\sigma$ is chosen below to be proportional to the quantity ${\bf n}^Ta{\bf n}$; see \cite{georgoulis2006note} for details. As such, $\sigma$ will vary along a curved element face even for element-wise constant diffusion $a$, thereby justifying the terminology ``penalization function'' as opposed to the standard terminology ``penalization function'' from the literature. Further, the theory presented below \Rev{can also be extended with minor modifications to} curved faces $F$, such that ${\bf n}^Ta{\bf n}>0$ only on a strict subset of that face and ${\bf n}^Ta{\bf n}=0$ on the remaining part.  That way one can reduce or even remove unphysical penalization on the hypersurfaces where the PDE may change type.
		\end{remark}
	}

	\section{Inverse and approximation estimates}\label{sec:approx}
	
	A key challenge in the error analysis presented below is the availability of inverse estimation and approximation results with uniform {/explicit} constants with respect to the shape of the elements in a given mesh.  
	
	A trace type inverse estimate for elements with one curved face has been recently proven in \cite{CGS18}  {under a shape-regularity assumption; see \eqref{stronger_assumption} below. Results in this direction have also appeared under various geometric assumptions in \cite{WX19,Massjung, MR3809538}, among others. Here, we extend these results by proving trace-inverse estimates for elements that are locally star-shaped, Lipschitz domains (see Assumption \ref{assumption_mesh} below). Moreover, given the importance of trace-inverse estimates for the stability of interior penalty dG methods, the new estimate constant is expressed via explicit and practically verifiable, geometric quantities (Lemma \ref{curved_inv_est} below).}  
	
	{In the same vein, we also extend the classical (Markov-type) $H^1-L_2$ inverse estimate to elements with piecewise $C^1$, locally star-shaped boundaries (see Assumptions \ref{assumption_mesh} and \ref{ass:c1} below). The proof builds upon and extends on earlier ideas from \cite{Kroo1} and \cite{book}. Here, we are particularly concerned with explicit quantification of the respective constant for a given element geometry. We note that $H^1-L_2$ inverse estimates are also relevant in the determination of penalty parameters in IP-dG methods for biharmonic operators \cite{DGpolyBiharmonic}.}
	
	{Also, we revisit a key stability argument that enabled the use of `degenerate' polytopic element shapes, i.e., ones containing very small/degenerating faces/edges compared to the element diameter, first proposed in \cite{cangiani2013hp}; see also \cite{cangiani2015hp,book} for improvements. This result is crucial in offering a practical choice of the discontinuity-penalization parameter for general polytopic meshes. The stability argument was based on two ingredients: 1) control of integral norms of polynomials with respect to domain perturbations using \cite[Lemma 3.7]{G08}, and 2) an $L_\infty-L_2$ inverse estimate. To retain this capability in the current setting, we prove an extension of \cite[Lemma 3.7]{G08} (see also \cite[Lemma 6]{Kroo1} for a related result) for generalized/curved prismatic elements; see Lemma \ref{new_chop} below. Moreover, we also prove an extension of the classical $L_\infty-L_2$ inverse estimate for generalized/curved prismatic elements. The latter two new estimates, in conjunction with a revised concept of \emph{coverability} (compared to \cite{cangiani2013hp,book}) are enough to provide extensions to previously known stability results for IP-dG within the present level of mesh generality.
	}

		\begin{figure}[h]
		\centering
		\includegraphics[height=3.5cm,width=8cm]{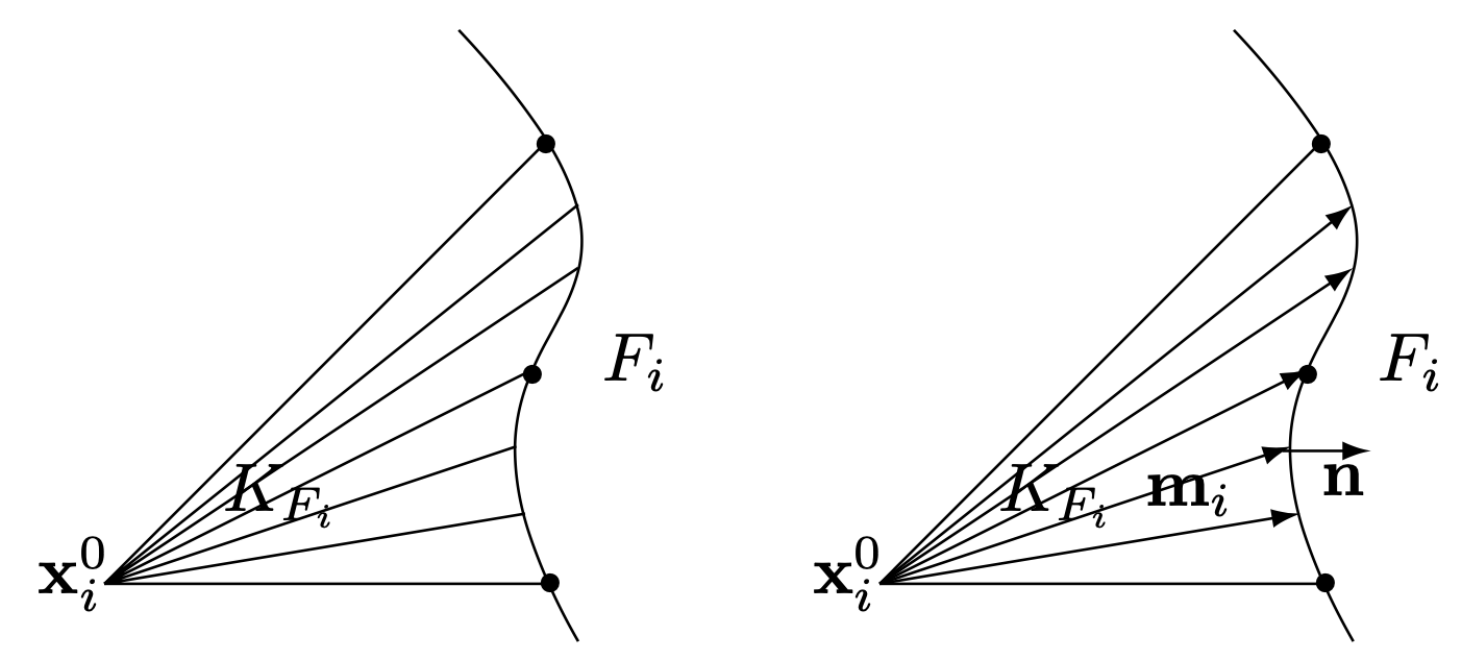}	
			\caption{Elements $K\in\mathcal{T}^{tr}$ are assumed to satisfy Assumption \ref{assumption_mesh} (a) (left) and (b) (right); \vx denotes a vertex.}
		\label{fig:curved_star-shaped}
	\end{figure}

%
%
%
%
%
%
%
%
%
%
%
%
%

				\begin{figure}[t]
			\centering
			\includegraphics[height=3.5cm,width=8cm]{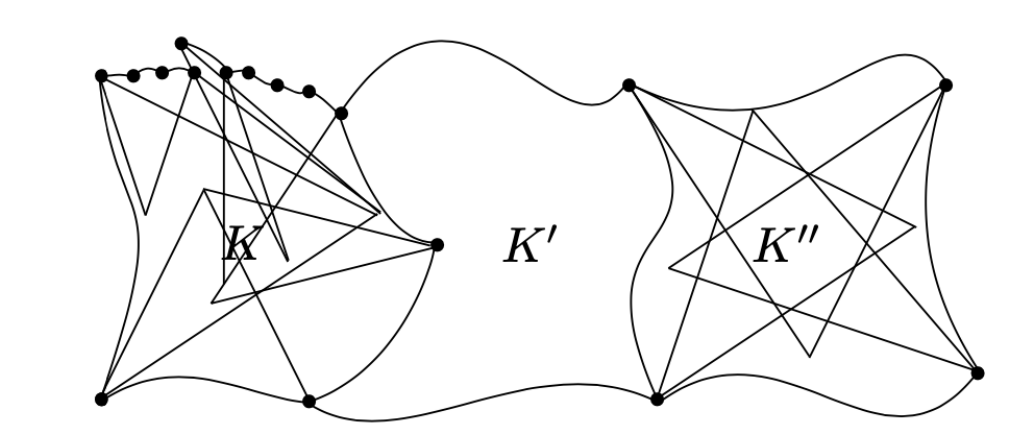}	
		\caption{Curved elements $K,\ K''$  with, respectively, $8$ and $4$ sub-elements satisfying Assumption \ref{assumption_mesh}.}
	\label{fig:curved_star-shaped_many_KFI}
		\end{figure}
	\begin{assumption}\label{assumption_mesh}
		For each element $K\in \mathcal{T}$, we assume that  {$K$ is a Lipschitz domain, and that we can subdivide} $\partial K$ into
		mutually exclusive subsets
		$\{F_i\}_{i=1}^{n_{K}^{}}$ satisfying the following property: there exist respective sub-elements $K_{F_i} {\equiv K_{F_i}({\bf x}^0_i)}\subset K$ with $d$ planar faces meeting at one vertex ${\bf x}^0_i\in K$,  {with $F_i\subset \partial K_{F_i}$}, such that, for $i=1,\dots,n_K^{}$, 
		\begin{itemize}
			\item[(a)]   $K_{F_i}$ is star-shaped with respect to ${\bf x}^0_i$. 
			We refer to Figure \ref{fig:curved_star-shaped}(left) for an illustration for $d=2$; 
			
			\item[(b)] $\m({\bf x})\cdot{\bf n}({\bf x})>0$ for $\m({\bf x}):={\bf x} - {\bf x}_i^0$,  ${\bf x}\in K_{F_i}$, and ${\bf n}({\bf x})$ the respective unit outward normal vector to $F_i$ at ${\bf x}\in F_i$. (We refer to Figure \ref{fig:curved_star-shaped}(right) for an illustration for $d=2$.)

		\end{itemize}
	\end{assumption}
	
	\begin{remark}\label{remark:comments}
		Some remarks on the above (very mild) mesh assumption are in order:
		\begin{itemize}
			\item[(i)] The sub-domains  $\{F_i\}_{i=1}^{n_{K}^{}}$ are \emph{not} required to coincide with the faces of the element $K$: each $F_i$ may be part of a face or may include one or more faces of $K$.  {Also, there is \emph{no} requirement for $\{n_K\}_{K\in\mathcal{T}}$ to remain uniformly bounded across the mesh.} 
			
			\item[(ii)]  We can make Assumption \ref{assumption_mesh}(b) stronger by further postulating that: it is possible to fix the point ${\bf x}^0_i$ such that there exists a global constant $c_{sh}>0$, such that 
			\begin{equation}\label{stronger_assumption}
				\m({\bf x})\cdot{\bf n}({\bf x})\ge c_{sh} h_{K_{F_i}}
				;
			\end{equation} this is the case, of course, for straight-faced polytopic elements, cf., \cite{CGS18,WX19}. Note that \eqref{stronger_assumption} does \emph{not} imply shape-regularity of the $K_{F_i}$'s; in particular $K_{F_i}$'s with `small' $F_i$ compared to the remaining (straight) faces of $K_{F_i}$ are acceptable. Such anisotropic sub-elements $K_{F_i}$'s may be necessary to ensure that each $K_{F_i}$ remains star-shaped when an element boundary's curvature is locally large; see, e.g.,  {$K_{F_i}$ in Figure \ref{fig:curved_star-shaped} and a collection of both `shape-regular' and `anisotropic' $K_{F_i}$'s in Figure \ref{fig:curved_star-shaped_many_KFI}}.
			
			\item[(iii)]  On certain geometrically extreme cases, satisfying Assumption \ref{assumption_mesh} may require a small number of refinements of the elements $K\in\mathcal{T}$ of a given initial mesh. 
			{\item[(iv)] $F_i$ is not required to be connected. However, by splitting $F_i$ to its connected subsets, re-indexing the $F_i$'s to correspond to unique $K_{F_i}$, we can allocate one $K_{F_i}$ to each $F_i$; we shall take the latter point of view in what follows to avoid further notational complexity.}
			\qed
		\end{itemize}
	\end{remark}
	{\begin{assumption}\label{ass:c1}
			We assume that the boundary $\partial K$ of each element $K\in\mathcal{T}$ is the union of a finite (yet, arbitrarily large!) number of closed $C^1$ surfaces.
		\end{assumption}
		Assumption \ref{assumption_mesh} is sufficient for the proof of the trace estimates presented below. Requiring both Assumptions \ref{assumption_mesh} and \ref {ass:c1} is sufficient for the validity of the $H^1-L_2$ inverse estimate presented below. 
	}
	
	\subsection{ {Basic trace estimates}}
	{We now discuss the new trace-inverse estimate and a version of the standard Sobolev trace estimate for Lipschitz domains satisfying Assumption \ref{assumption_mesh}.}

	\begin{lemma}\label{curved_inv_est}
		Let element $K\in\mathcal{T}$  {be a Lipschitz domain satisfying} Assumption \ref{assumption_mesh}. Then, for each $F_i\subset \partial K$, $i=1,\dots, n_K$, and for each $v\in\mathcal{P}_{p}(K)$, we have the inverse estimate:
		\begin{equation}\label{eq:inv_gen}
			\|v\|_{F_i}^2\le\frac{(p+1)(p+d)}{\displaystyle \min_{\mathbf{x}\in F_i}(\m\cdot\mathbf{n})}  \|v\|_{K_{F_i}}^2.
		\end{equation}
	\end{lemma}
	\begin{proof}
		We partition $F_i$ into $r$ $(d-1)$-dimensional curved simplices denoted by $F_i^j$, $j=1,\dots,r$, which are subordinate to the vertices possibly contained in $F_i$; $r$ is large enough to accommodate this requirement. Further, we construct a partition of $K_{F_i}$ into (curved) sub-elements $K_i^j$, by considering the simplices with 
		one (curved) face $F_i^j$ and the remaining vertex being $\uu{x}_i^0$; this is possible due to the star-shapedness of $K_{F_i}$ with respect to $\uu{x}^0_i$ as per Assumption \ref{assumption_mesh}(a). We refer to Figure \ref{fig:curved_inv_est} for an illustration when $d=2$. Notice that each $F_i^j$ may include at most one constituent curved face of $F_i$, or part thereof.

		\begin{figure}[h]
			\centering
			\includegraphics[height=3cm,width=10cm]{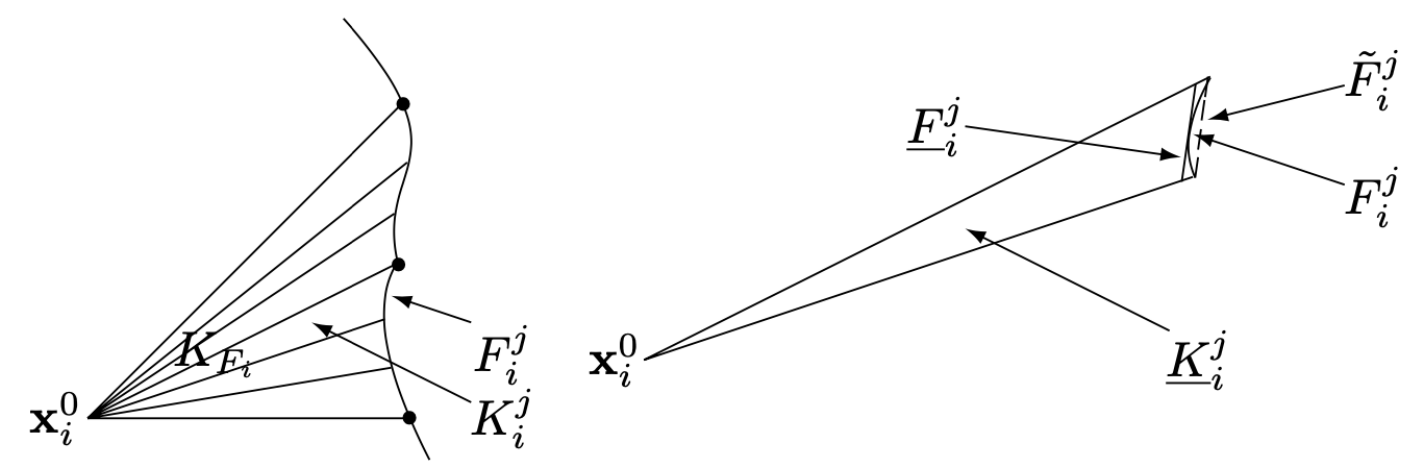}	
		
			\caption{Partitioned curved sub-element $K_{F_i}\subset K\in\mathcal{T}$; \vx denotes a vertex of $K$ (left); detail with $\underline{K}_i^j$ and related faces $\tilde{F}_i^j$ and $\underline{F}_i^j$ (right).}
		\label{fig:curved_inv_est}
		\end{figure}

		Let now $\tilde{F}_i^j$ denote the straight/planar related face defined by the $d-1$ vertices of $F_i^j$. Let also $\underline{K}_i^j$ be the largest straight-faced simplex contained in $K_i^j$ with face $\underline{F}_i^j$ parallel to $\tilde{F}_i^j$ and the remaining faces being subsets of the straight faces of $K_i^j$. The Divergence Theorem implies
		\[
		\begin{aligned}
			\int_{K_i^j\backslash  \underline{K}_i^j} \nabla\cdot (v^2\m)\ud \mathbf{x}=&\int_{\partial(K_i^j\backslash  \underline{K}_i^j)} v^2\m\cdot \mathbf{{n}}_{\partial (K_i^j\backslash  \underline{K}_i^j)} \ud s\\
			=&\int_{F_i^j} v^2\m\cdot\mathbf{{n}}\ud s + \int_{\underline{F}_i^j} v^2\m\cdot\mathbf{{n}}_{\underline{F}_i^j}\ud s,
		\end{aligned}
		\]
		with $\mathbf{{n}}_\omega$  denoting the outward normal vector of a domain  $\omega\subset\mathbb{R}^d$ and $\m$ as in Assumption \ref{assumption_mesh}(b), upon observing that $\m\cdot\mathbf{{n}}_{\partial (K_i^j\backslash  \underline{K}_i^j)}=0$ on $\partial (K_i^j\backslash  \underline{K}_i^j)\backslash (F_i^j\cup \underline{F}_i^j)$. 
		Now, denoting by $|\cdot|_2$ the Euclidean distance in $\mathbb{R}^d$, the product rule and elementary estimates imply
		\[
		\int_{K_i^j\backslash  \underline{K}_i^j} \nabla\cdot (v^2\m)\ud \mathbf{x} \le
		\Big(2 \max_{K_i^j} |\m|_2\|v\nabla v\|_{L_{\infty}(K_i^j\backslash  \underline{K}_i^j)}+d\|v\|^2_{L_{\infty}(K_i^j\backslash  \underline{K}_i^j)}\Big)|K_i^j\backslash  \underline{K}_i^j|,
		\]
		noting that $\nabla\cdot \m = d$. The right-hand side of the above inequality converges to zero as $|K_i^j\backslash  \underline{K}_i^j|\to 0$, which,  {in turn}, is achieved as $r\to \infty$. Thus,  Assumption \ref{assumption_mesh}(b) gives
		\[
		\min_{\mathbf{x}\in F_i^j}(\m\cdot\mathbf{n})\|v\|_{F_i^j}^2\le \Big|\int_{F_i^j} v^2\m\cdot\mathbf{n}\ud s\Big| \le  \Big|\int_{\underline{F}_i^j} v^2\m\cdot \mathbf{{n}}_{\underline{F}_i^j}\ud s\Big| + \epsilon,
		\]
		for some $\epsilon = { \mathcal{O}(|K_i^j\backslash  \underline{K}_i^j|)}$
		as $r\rightarrow \infty$.  
		Each of the finite $F_i$'s  is, in turn, image of a finite number of Lipschitz functions locally. \Rev
		{Let $L$ be the Lipschitz constant of a parametrisation of $F_i^j$ with respect to $\underline{F}_i^j$, giving $|F_i^j|\le L |\underline{F}_i^j|$. At the same time, we have $|K_i^j\setminus \underline{K}_i^j| \leq Lh_{\underline{F}_i^j}|\underline{F}_i^j|$, as the maximum Euclidean distance between $\underline{F}_i^j$ and $\tilde{F}_i^j$ is bounded from above by $Lh_{\underline{F}_i^j}$. 
			Hence the area  $|K_i^j\setminus \underline{K}_i^j|$ converges to zero faster than $|F_i^j|$ by an order of $h_{\underline{F}_i^j}$.
		}
		
		At the same time, a standard trace-inverse estimate on simplices, \cite{warburton2003constants}, yields
		\[
		\Big|\int_{\underline{F}_i^j} v^2\m\cdot \mathbf{{n}}_{\underline{F}_i^j}\ud s\Big| \le \max_{\mathbf{x}\in\underline{F}_i^j}(\m\cdot\mathbf{n}_{\underline{F}_i^j}) \frac{(p+1)(p+d) |\underline{F}_i^j|}{d|\underline{K}_i^j|}  \|v\|_{\underline{K}_i^j}^2.
		\]
		Combining the above, we have that, for any $\delta>0$, there exists an $r$ large enough such that
		\[
		\begin{aligned}
			\|v\|_{F_i^j}^2\le &\     \frac{\max_{\mathbf{x}\in\underline{F}_i^j}(\m\cdot\mathbf{n}_{\underline{F}_i^j})}{\min_{\mathbf{x}\in F_i^j}(\m\cdot\mathbf{n})} \frac{(p+1)(p+d) |\underline{ {F}}_i^j|}{d|\underline{K}_i^j|}  \|v\|_{\underline{K}_i^j}^2+ \frac{\epsilon}{{\min_{\mathbf{x}\in F_i^j}(\m\cdot\mathbf{n})} }\\
			\le&\  (1+\delta)\frac{(p+1)(p+d) |\underline{F}_i^j|}{d|\underline{K}_i^j|}  \|v\|_{\underline{K}_i^j}^2\le     (1+\delta) \frac{(p+1)(p+d) |F_i^j|}{d|\underline{K}_i^j|}  \|v\|_{K_i^j}^2,
		\end{aligned}
		\]
		as the first ratio on the first estimate tends to $1$ as $r\to \infty$ 
		.
		In the last inequality we used the bound $|\underline{F}_i^j|\le |F_i^j|$ and that $\underline{K}_i^j\subset K_i^j$. Another application of the Divergence Theorem and elementary calculations give
		\[
		d| K_i^j|=\int_{ K_i^j} \nabla\cdot \m\ud \mathbf{x}=\int_{F_i^j} \m\cdot\mathbf{n}\ud s \ge \min_{\mathbf{x}\in F_i^j}(\m\cdot\mathbf{n})|F_i^j|,
		\]
		or 
		\[
		d| \underline{K}_i^j| +d|K_i^j\backslash  \underline{K}_i^j|\ge \min_{\mathbf{x}\in F_i^j}(\m\cdot\mathbf{n})|F_i^j|,
		\]
		or
		\[
		\frac{|F_i^j|}{| \underline{K}_i^j|}\le \frac{d}{\displaystyle\min_{\mathbf{x}\in F_i^j}(\m\cdot\mathbf{n})}\Big(1+\frac{|K_i^j\backslash  \underline{K}_i^j|}{| \underline{K}_i^j|}\Big)\le  \frac{(1+\delta)d}{\displaystyle \min_{\mathbf{x}\in F_i^j}(\m\cdot\mathbf{n})},
		\]
		for any $\delta>0$ when $r$ is sufficiently large. Combining the above, we deduce
		\[
		\|v\|_{F_i}^2=	\sum_{j=1}^r	\|v\|_{F_i^j}^2\le\sum_{j=1}^r\frac{(1+\delta)^2(p+1)(p+d)}{\displaystyle\min_{\mathbf{x}\in F_i^j}(\m\cdot\mathbf{n})}  	\|v\|_{K_i^j}^2\le \frac{(1+\delta)^2(p+1)(p+d)}{\displaystyle\min_{\mathbf{x}\in F_{K_i}}(\m\cdot\mathbf{n})}  \|v\|_{K_i}^2.
		\]
		Taking, finally, $r\to \infty$, allows for $\delta\to 0$ and the result \eqref{eq:inv_gen} follows. 
	\end{proof}
	
	\begin{remark}
		It is important to stress that the right-hand side of \eqref{eq:inv_gen} is a function of $\uu{x}^0_i$ defining $K_{F_i}$. Since the closure of the original (curved) element $K$ is compact in $\mathbb{R}^d$, it is possible to minimise the right-hand side of \eqref{eq:inv_gen} by selecting an `optimal' $\uu{x}_i^0$.   
		{Moreover, upon making the stronger assumption \eqref{stronger_assumption}, we arrive at the familiar trace inverse estimate for star-shaped, shape-regular elements with piecewise smooth boundaries:
			$
			\|v\|_{\partial K}^2\le Cp^2h_K^{-1}  \|v\|_{K}^2.
			$}
	\end{remark}\vspace{-.3cm}
	{
		\begin{example}\label{inv_circle}
			Let $K=B(0,R)\subset \mathbb{R}^d$ be the ball of radius $R$ centred at the origin. Then, selecting $F_1=\partial K=:S(0,r)$, we have $\|v\|_{S(0,R)}^2 \le  (p+1)(p+d)R^{-1}\|v\|_{B(0,R)}^2$.
		\end{example}
	}
	
	{Within this} geometric setting, we can specify the constants of the classical trace inequality for $H^1$-functions. The result below is a mild extension of \cite[Lemma 4.1]{CGS18}, (cf. also \cite{WX19}) following closely the classical proof from \cite{agmon}.
	
	\begin{lemma}\label{lem:trace}
		Let $K\in\mathcal{T}$  {be a Lipschitz domain satisfying} Assumption \ref{assumption_mesh}. Then, for all $\zeta>0$, we have the estimate
		\begin{equation}\label{eq:trace}
			\|v\|_{F_i}^2 \le \frac{d+\zeta}{\displaystyle \min_{\uu{x}\in F_i}(\m\cdot\mathbf{n})}\|v\|_{K_{F_i}^{}}^2 + \frac{\displaystyle \max_{\uu{x}\in F_i}|\m|_2^2}{\displaystyle \zeta\min_{\uu{x}\in F_i}(\m\cdot\mathbf{n})}
			\|\nabla v\|_{K_{F_i}^{}}^2,
		\end{equation}
		for all $v\in H^1(\Omega)$ and $i=1,\dots, n_{K}^{}$.
	\end{lemma}
	\begin{proof}
		The Divergence Theorem and the fact that $\m\cdot\mathbf{{n}}=0$ on $\partial K_{F_i}\setminus F_i$ imply
		\[
		\int_{F_i} v^2\m\cdot\mathbf{n}\ud s =\int_{K_{F_i}}  \nabla\cdot (v^2\m)\ud \mathbf{x}
		\le   d\|v\|_{K_{F_i}^{}}^2+
		2\max_{\uu{x}\in F_i}|\m|_2\|v\|_{K_{F_i}^{}}\|\nabla v\|_{K_{F_i}^{}}, 
		\]
		from which the result already follows. 
	\end{proof}
	\begin{remark}
		Summing over $i=1,\dots, n_K^{}$, assuming \eqref{stronger_assumption} and that $h_{K_{F_i}}\sim h_K$, \eqref{eq:trace} gives the classical trace estimate
		$
		\|v\|_{\partial K}^2 \le C\big( h_K^{-1}\|v\|_{K}^2 + h_K
		\|\nabla v\|_{K}^2\big).
		$
	\end{remark}

	\subsection{ {Basic $H^1-L_2$ inverse estimate} } 
	{$H^1-L_2$ inverse estimates for polynomials on $d$-dimensional simplicial or box-like domains are proven via directional arguments, if explicit dependence on the polynomial degree is desired, see, e.g., \cite{schwab}. Generalizations of these estimates on convex domains use an analogous method of proof \cite{Kroo1}. Here, in the same spirit, we extend further the domain generality in $H^1-L_2$ inverse estimates, by also employing directional arguments on curved prismatic subdomains; the general case then follows by covering general Lipschitz domains by these curved prisms.}
	\begin{figure}
		\centering
		\setlength{\unitlength}{1.3cm}
		\linethickness{0.1mm}
		\begin{picture}(0,2)(1,0)
			\put(0,0){\line(1,0){1.5}}
			\put(0,0){\line(0,1){1.5}}
			\put(1.5,0){\line(0,1){1.8}}
			\cbezier(0,1.5)(.2, 2.5)(.8, 1.6)(1, 1.45)
			\put(.98,1.47){\circle*{.1}}
			\cbezier(1, 1.45)(1.2, 1.5)(1.4, 1.6)(1.5, 1.8)

			\put(0,1.5){\circle*{.1}}
			\put(1.5,1.8){\circle*{.1}}
			\put(0,0){\circle*{.1}}
			\put(1.5,0){\circle*{.1}}
			\put(.7,1.9){$\hat{F}$}
			
			\put(.7,.1){$\hat{F}^0$}
			
			{\color{red} \put(0,1.42){\line(1,0){1.5}}}
			{\color{red} \put(-.28,0){\vector(0,-1){.04}}}
			{\color{red} \put(-.37,0){\vector(0,1){1.4}}}
			{\color{red} \put(-.67,0.5){$\hat{\rho}$}}
			
		\end{picture}
		\caption{A reference generalized prism $\hat{K}$ for $d=2$.}
		\label{fig:ref_gen_one}
	\end{figure}
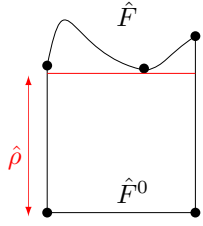
	{
		\begin{definition}\label{def:gen_prism} Let $\hat{F}^0:=[0,1]^{d-1}\Rev{\subset\mathbb{R}^d}$ and $\phi:\hat{F}^0\to\mathbb{R}$ a Lipschitz continuous scalar function. A \emph{reference generalized prism} is a domain $\hat{K}\equiv\hat{K}_{\phi}\subset \mathbb{R}^d$ given by
			\[
			\hat{K}\equiv\hat{K}_{\phi}:=\{ {\bf x}\in \mathbb{R}^d: 0\le x_i\le 1, i=1,\dots,d-1, 0\le x_d\le \phi(x_1,\dots,x_{d-1})\},
			\]
			with the properties: 1) $[0,1]^d\subset \hat{K}$, and 
			2) the straight line connecting any pair $({\bf x},{\bf y}) \in \hat F^0 \times \hat F$ lies fully in $\hat K$.
			\[
			\hat{F}:=\{ {\bf x}\in \mathbb{R}^d: 0\le x_i\le 1, i=1,\dots,d-1,  x_d= \phi(x_1,\dots,x_{d-1})\}.
			\] 
			Also, we set
			$
			\hat{\rho}:=\Rev{\sup} \{\rho \ge 1:\hat{F}^0\times[0,\rho]\subset \hat{K}\}
			$ and $\hat{r}:=\lfloor \max_{{\bf x}\in\hat{F}^0}\phi({\bf x})\rfloor +1$.
			\qed
		\end{definition}
		We refer to Figure \ref{fig:ref_gen_one} for an illustration.
	}
	{
		\begin{remark}\label{rem:contraction}
			A sufficient but, crucially, \emph{not} necessary condition for $\hat{K}\equiv\hat{K}_\phi$ to be a reference generalized prism is that $\phi$ is a contraction. Since, however, $\hat{K}_\phi$ will be used in conjunction with affine maps below, it will become possible to consider $\phi$ with Lipschitz constants greater than one. 
		\end{remark}
	}
	
	{
		\begin{remark}\label{rho}
			The \Revthree{`height' $\hat{r}$} is a measure of \Revthree{anisotropy of the reference generalized prism. Note that we} can take $\hat{\rho}=1$ without essential loss of generality. Indeed, if $\hat{\rho}>1$, the change of variables $x_d\to x_d/\hat{\rho}$ implies a modification of the Lipschitz function $\phi$, reducing its Lipschitz constant. Star-shapedness with respect to $\hat{F}^0$ is also ensured (cf., Remark \ref{rem:contraction}).  
		\end{remark}
	}
	
	\Rev{In light of the above remark,} we  consider the case $\hat{\rho}=1$ only, in what follows.
	
	{
		\begin{lemma}\label{basic_H1_L2} Let $v\in\mathcal{P}_p(\hat{K})$,  $p\in\mathbb{N}$, with $\hat{K}\subset\mathbb{R}^d$ a reference generalized prism. Then, we have the inverse estimate
			\begin{equation}\label{v_breaking_final}
				\|\nabla v\|_{\hat{K}}^2
				\le C_{\rm inv}^Bp^4\|v\|_{\hat{K}}^2,
			\end{equation}
			with $\displaystyle  C_{\rm inv}^B\equiv C_{\rm inv}^B(d,\hat{r}):=
			\Revthree{288(d-1)\hat{r}^2}
			+12d$.
		\end{lemma}
	}
	
	\begin{figure}[t]\vspace{.8cm}
		\centering
		\setlength{\unitlength}{1.2cm}
		\linethickness{0.15mm}
		\begin{picture}(0,0)(4,0)
			\put(0,0){\line(1,0){1.5}}
			\put(0,0){\line(0,1){2.61}}
			\put(1.5,0){\line(0,1){2.61}}
			

			\put(.58,1.24){$R_{\Revthree{-}}^{x_1}$}

			\put(.38,0){\line(-2,7){.38}}
			\put(1.5,0){\line(-2,7){.74}}
			
			{\color{red} \put(0,1.32){\line(1,0){1.5}}}
			{\color{red} \put(-.08,2.61){\line(1,0){1.5}}}
			
		\end{picture}
		\begin{picture}(0,2)(2,0)
			\put(0,0){\line(1,0){1.5}}
			\put(0,0){\line(0,1){2.61}}
			\put(1.5,0){\line(0,1){2.61}}

			{\color{red} \put(0,1.32){\line(1,0){1.5}}}
			
			{\color{red} \put(-.07,2.61){\line(1,0){1.5}}}

			
			\put(.57,1.24){$R_{\Revthree{+}}^{x_1}$}
			
			\put(1.01,0){\line(2,7){.38}}
			\put(-.08,0){\line(2,7){.74}}

		\end{picture}	
		\begin{picture}(0,2)(0,0)
			\put(0,0){\line(1,0){1.5}}
			\put(0,0){\line(0,1){1.5}}
			\put(1.5,0){\line(0,1){1.8}}
			\cbezier(0,1.5)(.2, 2.5)(.8, 1.6)(1, 1.45)
			\put(.98,1.47){\circle*{.1}}
			\cbezier(1, 1.45)(1.2, 1.5)(1.4, 1.6)(1.5, 1.8)
			
			\put(0,1.5){\circle*{.1}}
			\put(1.5,1.8){\circle*{.1}}
			\put(0,0){\circle*{.1}}
			\put(1.5,0){\circle*{.1}}
			\put(.38,0){\line(-2,7){.38}}
			\put(1.5,0){\line(-2,7){.42}}
			{\color{red} \put(0,1.32){\line(1,0){1.5}}}
			
			\put(.55,.6){$\hat{K}_{\Revthree{-}}^{x_1}$}
		\end{picture}
		\begin{picture}(0,2)(-2,0)
			\put(0,0){\line(1,0){1.5}}
			\put(0,0){\line(0,1){1.5}}
			\put(1.5,0){\line(0,1){1.8}}
			\cbezier(0,1.5)(.2, 2.5)(.8, 1.6)(1, 1.45)
			\put(.98,1.47){\circle*{.1}}
			\cbezier(1, 1.45)(1.2, 1.5)(1.4, 1.6)(1.5, 1.8)

			\put(0,1.5){\circle*{.1}}
			\put(1.5,1.8){\circle*{.1}}
			\put(0,0){\circle*{.1}}
			\put(1.5,0){\circle*{.1}}
			\put(0,0){\line(2,7){.535}}
			\put(1.1,0){\line(2,7){.38}}
			{\color{red} \put(0,1.32){\line(1,0){1.5}}}
			
			\put(.5,.6){$\hat{K}_{\Revthree{+}}^{x_1}$}
		\end{picture}
		
		\caption{ {$R=[0,1]\times [0,\hat{r}]$ and $\hat{K}$, with $\hat{r}=2$, and their respective truncated prisms $R_{\Revthree{\pm}}^{x_1}$ and $\hat{K}_{\Revthree{\pm}}^{x_1}$.}}
		\label{fig:ref_gen_rec_new}
	\end{figure}
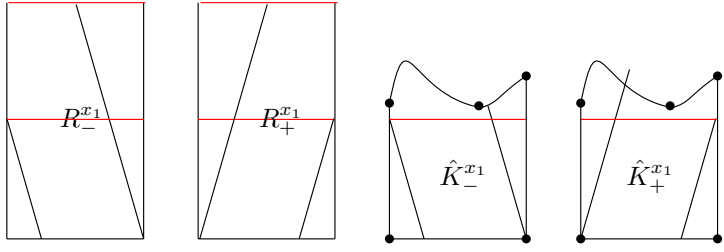

	\begin{proof}  {We begin by introducing some notation. Let $F$ be a  hyperplanar region in $\mathbb{R}^d$ and let vector ${\bf v}\in\mathbb{R}^d$. We define a \emph{zone} $Z(F,{\bf v})\subset\mathbb{R}^{d}$, to be the geometric locus given by
			$
			Z(F,{\bf v}):=\{ {\bf z} +\beta {\bf v} : {\bf z}\in F,\beta\in\mathbb{R} \}.
			$
			(Thus, for instance, the domain $[0,1]\times \mathbb{R}=Z([0,1],(0,\alpha))$ for any $\alpha\in \mathbb{R}$.)
			Using this notation, we now construct \Revthree{$2(d-1)$} suitable subsets \Revthree{$\{\hat{K}_{\pm}^{x_j}\}_{j=1}^{d-1}$},  so that  the union of $\hat{K}_0:=[0,1]^d$ together with \Revthree{$\{\hat{K}_{\pm}^{x_j}\}_{j=1}^{d-1}$} cover $\hat{K}$.}
		
		{		We first present the construction for $d=2$ for accessibility. Set  $\hat{F}_{0,-}^{x_1}:=\Revthree{[(2\hat{r})^{-1},1]}$ and $\hat{F}_{0,+}^{x_1}:=\Revthree{[0,1-(2\hat{r})^{-1}]}$. Then, elementary geometric arguments reveal that the rectangle $R:=[0,1]\times[0,\hat{r}]$ can be covered by the union of $R_0:=[0,1]^2$ and the $\Revthree{2}$ truncated prisms defined as:
			\[
			R_{\Revthree{\pm}}^{x_1}:=R\cap Z\big(\hat{F}_{0,\pm}^{x_1}, {\bf v}_{\Revthree{\pm}}^{x_1}\big), \text{ with } {\bf v}_{\Revthree{\pm}}^{x_1}:=\frac{(\pm 1/2,\Revthree{\hat{r}})^T}{|(1/2,\Revthree{\hat{r}})|},
			\]
			with $|\cdot|$ denoting the standard Euclidean distance; 
			we refer to Figure \ref{fig:ref_gen_rec_new} for an illustration with $\hat{r}=2$. }
		
		{		Correspondingly, for $d=3$, $R:=[0,1]^2\times[0,\hat{r}]$ can be covered by $R_0:=[0,1]^3$ together with $\Revthree{2}$ `$x_1$-direction tilted', truncated prisms:
			\[
			R^{x_1}_{\Revthree{\pm}}:=R\cap Z\big(\hat{F}_{0,\pm}^{x_1}, {\bf v}_{\Revthree{\pm}}^{x_1}\big), \text{ with } {\bf v}_{\Revthree{\pm}}^{x_1}:=\frac{(\pm1/2,0,\Revthree{\hat{r}})^T}{|(1/2,0,\Revthree{\hat{r}})|},
			\]
			with $\hat{F}_{0,-}^{x_1}:=\Revthree{[(2\hat{r})^{-1},1]}\times [0,1]$ and $\hat{F}_{0,+}^{x_1}:=\Revthree{[0,1-(2\hat{r})^{-1}]}\times [0,1]$ the respective prism bases. At the same time, $R$ can be also covered by $R_0$ together with the $\Revthree{2}$ `$x_2$-direction tilted', truncated prisms:
			\[
			R_{\Revthree{\pm}}^{x_2}:=R\cap Z\big( \hat{F}_{0,\pm}^{x_2}, {\bf v}_{\Revthree{\pm}}^{x_2}\big), \text{ with } {\bf v}_{\Revthree{\pm}}^{x_2}:=\frac{(0,\pm1/2, \Revthree{\hat{r}})}{|(0,1/2,\Revthree{\hat{r}})^T|},
			\]
			with $\hat{F}_{0,-}^{x_2}:= [0,1]\times\Revthree{[(2\hat{r})^{-1},1]}$ and $\hat{F}_{0,+}^{x_2}:=[0,1]\times \Revthree{[0,1-(2\hat{r})^{-1}]}$ the respective prism \emph{bases}. (We note the 'overloading' of notation with respect to dimension.) The construction for $d\ge 4$ follows in a completely analogous fashion by considering $\Revthree{2}$ `$x_j$-direction tilted', truncated prisms for each $j=1,\dots,d-1$.}
		
		{	Since $\hat{K}\subset [0,1]^{d-1}\times [0,\hat{r}]$, we consider the sets $\hat{K}_0=R_0$ together with
			\[
			\hat{K}_{\Revthree{\pm}}^{x_j}:=\hat{K}\cap R_{\Revthree{\pm}}^{x_j},
			\]
			for each fixed $j=1,\dots,d\Revthree{-1}$;
			see Figure \ref{fig:ref_gen_rec_new} for an illustration for $d=2$ and $\hat{r}=2$.  }
		
		{
			First, we observe the estimates
			\begin{equation}\label{der_est_squared}
				\begin{aligned}
					|v_{x_j}|^2\le&\   8\Revthree{\hat{r}}^2 |v_{x_d}|^2+2(4\Revthree{\hat{r}}^2+1) |{\bf v}_{\Revthree{\pm}}^{x_j}\cdot\nabla v|^2,
				\end{aligned}
			\end{equation}
			for $j=1,\dots,d-1$.
			Using the latter, we have, respectively,
			\begin{equation}\label{v_breaking}
				\begin{aligned}
					& \|\nabla v\|_{\Revthree{\hat{K}}}^2 -\|v_{x_d}\|_{\Revthree{\hat{K}}}^2-\sum_{j=1}^{d-1}\|v_{x_j}\|_{\hat{K}_0}^2
					\le
					\sum_{j=1}^{d-1}\left(\|v_{x_j}\|_{\hat{K}_{\Revthree{+}}^{x_j}}^2+\|v_{x_j}\|_{\hat{K}_{\Revthree{-}}^{x_j}}^2\right) \\
					\le &\ 
					\sum_{j=1}^{d-1} \Big( 8\Revthree{\hat{r}}^2 \|v_{x_d}\|_{\hat{K}}^2
					+2(4\Revthree{\hat{r}}^2+1) 
					\big(
					\|{\bf v}_{\Revthree{+}}^{x_j}\cdot\nabla v\|_{\hat{K}_{\Revthree{+}}^{x_j}}^2
					+
					\|{\bf v}_{\Revthree{-}}^{x_j}\cdot\nabla v\|_{\hat{K}_{\Revthree{-}}^{x_j}}^2
					\big)
					\Big).
				\end{aligned}
			\end{equation}
			We now estimate each term on the right-hand side of \eqref{v_breaking}. For ${\rm e}_d:=(0,\dots,0,1)^{\rm T}$, 
			let $\ell_{\bf x}:=\hat{K}\cap \{{\bf x}+\alpha {\rm e}_d:\alpha\in\mathbb{R}\}$, i.e., the vertical line contained in $\hat{K}$ and passing through a point ${\bf x}$. Then, we have
			\begin{equation}\label{v_y}
				\| v_{x_d}\|_{\hat{K}}^2= \int_{\hat{F}^0}\int_{\ell_{\bf x}} v_{x_d}^2 \ud x_d \ud {\bf x}
				\le
				12 p^4\| v\|_{\hat{K}}^2,
			\end{equation}
			from  Fubini's Theorem and an \emph{one-dimensional} inverse estimate, see, e.g., \cite[Theorem 3.91]{schwab}. 
			We set  $\ell_{{\bf x},\Revthree{\pm}}^{x_j}:=\hat{K}\cap \{{\bf x}+\alpha {\bf v}_{\Revthree{\pm}}^{x_j}:\alpha\in\mathbb{R}\}$. Then,
			\begin{equation*}\label{v_tilt}
				\|{\bf v}_{\Revthree{\pm}}^{x_j} \cdot\nabla v\|_{\hat{K}_{\Revthree{\pm}}^{x_j}}^2= 
				\int_{\hat{F}_{0,\Revthree{\pm}}^{x_j}}\int_{\ell_{{\bf x},\Revthree{\pm}}^{x_j}} ({\bf v}_{\Revthree{\pm}}^{x_j} \cdot\nabla v)^2 \ud \zeta \ud {\bf x}
				\le \frac{48\Revthree{\hat{r}}^2 p^4}{4\Revthree{\hat{r}}^2+1}\| v\|_{\hat{K}}^2,
			\end{equation*}
			for $j=1,\dots, d$, upon noticing that the length of $\ell_{{\bf x},\Revthree{\pm}}^{x_j}$ is bounded from below by $\sqrt{\Revthree{\hat{r}}^2+1/4}/\Revthree{\hat{r}}$ (the length of the portion of $\ell_{{\bf x},\Revthree{\pm}}^{x_j}$ contained in $R_0$). Thus, \eqref{v_breaking} implies
			\begin{equation*}
				\begin{aligned}
					&\|\nabla v\|_{\hat{K}}^2 -\|v_{x_d}\|_{\hat{K}}^2-\sum_{j=1}^{d-1}\|v_{x_j}\|_{\hat{K}_0}^2
					\le  
					\Revthree{288p^4(d-1)\hat{r}^2}
					\|v\|_{\hat{K}}^2.
				\end{aligned}
			\end{equation*}
			The result already follows by combining the last estimate with \eqref{v_y} and the corresponding inverse estimates for  $\|v_{x_j}\|_{[0,1]^d}^2$.
		}
	\end{proof}	
	
	{
		Notice that \eqref{v_breaking_final} retrieves the known constant for $d=1$.  
		If $K$ is cuspoidal, (i.e., not Lipschitz) \eqref{v_breaking_final} does \emph{not} hold in general; we refer to \cite{Kroo2} for a counterexample.
		\begin{remark}
			Careful inspection of the above proof shows that, in fact, we have proven the sharper inverse estimate $\|\nabla v\|_{\hat{K}}^2
			\le C_{\rm inv}^Bp^4\hat{\rho}^{-2}\|v\|_{\hat{K}}^2$. In view of Remark \ref{rho}, however, a linear scaling results into a modified $\hat{K}$ (and, possibly modified $\hat{r}$) for which \eqref{v_breaking_final} is sharp.
		\end{remark}
	}
	\vspace{-.2cm}
	\subsection{ {Stability with respect to domain perturbation}}
	{
		We now prove a stability result with respect to domain perturbation in the spirit of \cite[Lemma 3.7]{G08} (see also \cite[Lemma 6]{Kroo1}). 
	}
	{
		\begin{lemma}\label{new_chop}
			Let $\hat{K}$ a reference generalized prism and consider its subset
			$
			\hat{K}_\epsilon:=\hat{K}\cap (\hat{K}-\epsilon {\rm e}_d);
			$
			here $A+z:=\{x+z, x\in A\}$, for $A\subset \mathbb{R}^d$ and $z\in \mathbb{R}^d$. Then, for all $v\in \mathcal{P}_p(\hat{K})$,  $p\in\mathbb{N}$, and for any $0<\epsilon\le (8p)^{-2}$, we have the estimate
			\begin{equation}
				\frac{1}{2}\| v\|_{\hat{K}}^2
				\le  \| v\|_{\hat{K}_\epsilon}^2.
			\end{equation}
		\end{lemma}
		\begin{proof}
			Set $\ell_{{\bf x},\epsilon}:=\hat{K}\cap (\ell_{{\bf x}}-\epsilon {\rm e}_d) $. Then, we have, respectively,
			\begin{equation}\label{strip_small_gen}
				\begin{aligned}
					\|v\|_{\hat{K}\backslash \hat{K}_\epsilon}^2 
					=& \int_{\hat{F}^0}\int_{\ell_{{\bf x}}\backslash \ell_{{\bf x},\epsilon}} v^2 \ud x_d\,\ud {\bf x} 
					\le 
					\epsilon\int_{\hat{F}^0}\|v\|_{L_\infty(\ell_{{\bf x}}\backslash \ell_{{\bf x},\epsilon})} ^2\ud {\bf x}\\
					\le&\ 
					\epsilon\int_{\hat{F}^0}\|v\|_{L_\infty(\ell_{{\bf x}})}^2\ud {\bf x} \le
					32	\epsilon p^2\|v\|_{\hat{K}}^2, 
				\end{aligned}
			\end{equation}
			by Markov's inequality (see, e.g., \cite[Theorem 3.92]{schwab}) since the length of $\ell_{{\bf x}}$ is bounded from below by one. Selecting now $0<\epsilon\le (8p)^{-2}$, the result follows, by simply observing that $\|v\|_{\hat{K}}^2 -	\|v\|_{\hat{K}_\epsilon}^2=
			\|v\|_{\hat{K}\backslash \hat{K}_\epsilon}^2 $.
		\end{proof}
	} 
	\begin{example}\label{ex:osc}
		Consider $\hat{K}$ with $\phi({\bf x}) = 1+{\bf x}\cdot ({\bf 1}-{\bf x})+(16p)^{-2}\sin ( \alpha\pi{\bf 1} \cdot{\bf x})$, for some $\alpha\ge 1$, with ${\bf 1}:=(1,1,\dots,1)^T\in\mathbb{R}^{d-1}$. This is chosen so that 
		$
		\phi({\bf x})\in [1+{\bf x}\cdot ({\bf 1}-{\bf x})-\epsilon,1+ {\bf x}\cdot ({\bf 1}-{\bf x})+\epsilon],
		$
		for $\epsilon$ within the range required for the statement of Lemma \ref{new_chop} to hold. For sufficiently large $\alpha$, $\hat{K}$ is \emph{not} star-shaped with respect to $\hat{F}^0=[0,1]^{d-1}$.  
		Nevertheless, $\phi$ is sufficiently approximated by $\phi_2({\bf x}) := 1+{\bf x}\cdot ({\bf 1}-{\bf x})$, which is star-shaped with respect to $\hat{F}^0=[0,1]^{d-1}$. Thus, for $v\in\mathcal{P}_p(\hat{K})$, from Lemma \ref{basic_H1_L2}, we have 
		\[
		\|\nabla v\|_{\hat{K}}^2\le 2\|\nabla v\|_{\hat{K}_{\phi_2}}^2\le 4C_{\rm inv}^B(d,2)p^4\|v\|_{\hat{K}}^2.
		\]
		%
	\end{example}
	\vspace{-.2cm}
	\subsection{ {$L_\infty-L_2$ inverse estimate}}
	We continue by proving an $L_\infty-L_2$ inverse estimate for reference generalized prisms.
	\Rev{\vspace{-.1cm}
		\begin{lemma}
			\label{basic_Linfty_L2} 
			Let $\hat{K}\equiv\hat{K}_\phi\subset\mathbb{R}^d$ a reference generalized prism. Then, the inverse estimate 
			\begin{equation}\label{est_Linfty_L2}
				\|v\|_{L_\infty(\hat{K})}^2\le (32p^{2})^d\|v\|_{\hat{K}}^2,
			\end{equation}
			holds for all $v\in\mathcal{P}_p(\hat{K})$, $p\in\mathbb{N}$.
		\end{lemma}
	}
	\vspace{-.3cm}
	\Rev{	\begin{proof}
			{Let $\mathbf{x}_{\max}\in \hat{K}$ such that $	\|v\|_{L_\infty(\hat{K})}=|v(\mathbf{x}_{\max})|$. Then,  either 
				$\mathbf{x}_{\max}\in \hat{K}_0=[0,1]^d\subset \hat{K}$ or  	$\mathbf{x}_{\max}\in \hat{K}\backslash \hat{K}_0$. Let now $\hat{K}(\mathbf{x}_{\max})$ be the pyramid with vertex $\mathbf{x}_{\max}$ and base $\hat{F}^0=[0,1]^{d-1}$. If $\mathbf{x}_{\max}\in \hat{K}_0$, then, we have
				$
				\|v\|_{L_\infty(\hat{K})}^2\le 32^dp^{2d} 	\|v\|_{\hat{K}}^2$ (see, e.g., \cite[eq.~(3.6.4)]{schwab},) whereas if $\mathbf{x}_{\max}\in \hat{K}\backslash  \hat{K}_0$, we have, respectively,
				$
				|v(\mathbf{x}_{\max})|^2=			\|v\|_{L_\infty(\hat{K})}^2= 	\|v\|_{L_\infty(\hat{K}(\mathbf{x}_{\max}))}^2.
				$ 
				Since $\mathbf{x}_{\max}$ is a vertex, we employ an one-dimensional trace inverse estimate \cite{warburton2003constants} iteratively with respect to dimension, to deduce	$|v(\mathbf{x}_{\max})|\le \prod_{j=1}^d (p+j)\|v\|_{\hat{K}(\mathbf{x}_{\max})}$. Here we have used the fact that the dimensions of $\hat{K}(\mathbf{x}_{\max})$ are grater than $1$. Combining the two cases and taking the maximum constant, the result already follows.
			}
		\end{proof}
	}\vspace{-.3cm}
	{
		\begin{remark}		
			An inspection of the proof of Lemma \ref{basic_Linfty_L2} show that the only geometric assumption needed is that the curved face $\hat{F}$ is given as the graph of a Lipschitz function $\phi$. So Lemma \ref{basic_Linfty_L2} holds without assuming property 2) in Definition \ref{def:gen_prism}.
		\end{remark}\vspace{-.1cm}

		\subsection{ {Inverse estimates on general domains}}
		{We now extend the above inverse estimates to general curved polytopic elements $K\in\mathcal{T}$. To that end, we shall relax the concept of $p$-coverability of polytopic elements introduced in \cite{cangiani2013hp}, (see also \cite{cangiani2015hp,book}) from simplicial coverings of general-shaped elements $K\in\mathcal{T}$, to coverings involving affinely mapped generalized prisms. }
		{
			\begin{definition}\label{def_poly_assumption}
				An element $K\in\mathcal{T}$ is said to be {\em $p$-coverable} with respect to $p\in\mathbb{N}$, if there exists a set of $m_{K}\in \mathbb{N}$ generalized prisms $\hat{K}_j$ and corresponding affine maps $\Phi_j$, such that the mapped generalized prisms $\overline{K}_j:=\Phi_j(\hat{K}_j)$, $j=1,\dots, m_{K}$, form a, possibly overlapping, covering of $K$ with the additional properties
				\begin{equation}
					\Rev{{\rm dist}(\partial \overline K_j, K)} \le \mathfrak{h}_{\overline{K}_j}(8p)^{-2}
					\label{T_tilde_condition}
				\end{equation}
				and
				\begin{equation}
					|\overline{K}_j|\ge c_{as} |K|, \label{T_tilde_condition2}
				\end{equation}
				for all $j=1,\dots, m_{K}$, where $\mathfrak{h}_{\overline{K}_j}:=\sup_{{\bf x}\in\hat{F}^0 }|\Phi_j(\ell_{{\bf x},j})|$ and $c_{as}$ is a
				positive constant, independent of $K$ and of $\mathcal{T}$, \Rev{with ${\rm dist}(\partial \overline{K}_j,K):=\sup_{{\bf x}\in \partial \overline{K}}\inf_{{\bf y}\in K}|{\bf x}-{\bf y}|$  the one-sided Hausdorff
					distance of $\partial \overline{K}_j$ from $K$}, and $\ell_{{\bf x},j}:=\hat{K}_j\cap \{{\bf x}+\alpha {\rm e}_d:\alpha\in\mathbb{R}\}$.
			\end{definition}
		}

		The motivation for the above definition is the stability result for polynomials with respect to domain perturbation  {given in Lemma \ref{new_chop} above}. If $K$ is $p$-coverable,  {\eqref{T_tilde_condition} implies that} there exists a covering  {of affinely mapped generalized prisms} $\overline{K}_j$ and respective sub-prisms  $\underline{K}_j {:=\overline{K}_{j,\epsilon}}$, $0< \epsilon\le \mathfrak{h}_{\overline{K}_j}(8p)^{-2}$, $j=1,\dots,m_K$, such that $\underline{K}_j \subset K$. Then, we have, for any $v\in\mathcal{P}_p(K)$,
		\begin{equation}\label{subel_pcover}
			\frac{1}{2}	\ltwo{v}{\overline{K}_j}^{ {2}}\le  \ltwo{v}{\underline{K}_j}^{ {2}}\le\ltwo{v}{K}^{ {2}}.
		\end{equation}

		{We now show that \eqref{T_tilde_condition} is implied by Assumptions \ref{assumption_mesh} and \ref{ass:c1}. Therefore, $p$-coverability for an element satisfying Assumptions \ref{assumption_mesh} and \ref{ass:c1} is ensured under the validity of \eqref{T_tilde_condition2} only.}
		{
			\begin{lemma}\label{lem:tight_covering}
				Let $K\in\mathcal{T}$ satisfying Assumptions \ref{assumption_mesh} and \ref{ass:c1}. Then, there exists a set of $m_K\in\mathbb{N}$ generalized prisms $\hat{K}_j$ and corresponding affine maps $\Phi_j$, such that the mapped domains $\overline{K}_j:=\Phi_j(\hat{K}_j)$, $j=1,\dots,m_K$, form a cover of $K$ with the property
				\begin{equation}\label{tight_cover}
					\Rev{{\rm dist}(\partial \overline K_j, K)} \le \mathfrak{h}_{\overline{K}_j}(8p)^{-2}
				\end{equation}
				for all $j=1,\dots, m_{K}$, with the notation of Definition \ref{def_poly_assumption}.
			\end{lemma}
		}
		\begin{proof}
			{
				From Assumption \ref{ass:c1}, $\partial K$ is comprised of a finite number of closed (co-dimen\-sion one) $C^1$ surfaces $(\partial K)_j$, $j=1,\dots, \tilde{z}_K$, for some $\tilde{z}_K\in\mathbb{N}$. By possibly further subdividing the $(\partial K)_j$'s into subsets, say, $(\partial K)_j$, $j=1,\dots, z_K$, Assumption \ref{assumption_mesh}, ensures that for each of $(\partial K)_j$ there exist a point ${\bf x}^0_j\in K$ such that the curved simplex $K_{(\partial K)_j}\equiv K_{(\partial K)_j}({\bf x}^0_j)$ is star-shaped with respect to ${\bf x}^0_j$ and that $g_j({\bf x}):=({\bf x}-{\bf x}^0_j)\cdot {\bf n}({\bf x})>0$ for any ${\bf x}\in (\partial K)_j$. (More than one $(\partial K)_j$ are allowed to share the same ${\bf x}^0_j$.) Since  $(\partial K)_j$ is $C^1$, $g_j$ is continuous on $(\partial K)_j$ and, thus, there exists a positive number $\delta_j$, such that $g_j({\bf x})\ge \delta_j$.
			}	
			
			{Now, for any ${\bf \tilde{x}}^0_j\in\mathbb{R}^d$, with $ |{\bf x}^0_j-{\bf \tilde{x}}^0_j|<\delta_j$, we have 
				\[
				|g_j({\bf x})-({\bf x}-{\bf \tilde{x}}^0_j)\cdot {\bf n}({\bf x})|\le |{\bf x}^0_j-{\bf \tilde{x}}^0_j|<\delta_j,
				\]
				and, therefore, 
				$\tilde{g}_j({\bf x}):=({\bf x}-{\bf \tilde{x}}^0_j)\cdot {\bf n}({\bf x})>0$ for any ${\bf x}\in (\partial K)_j$. 
				Hence, $(\partial K)_j$ is star-shaped with respect to $B({\bf x}^0_j,\delta_j)$ \Rev{in $K$; that is any line connecting any point of $(\partial K)_j$ with a point of $B({\bf x}^0_j,\delta_j)$ lies wholly in $K$}. This implies that $K_{(\partial K)_j}$ is star-shaped \Rev{in $K$} with respect to any \Rev{subset of $B({\bf x}^0_j,\delta_j)$ and, in particular, with respect to any} $(d-1)$-hypercube passing through ${\bf x}^0_j$ and contained in $B({\bf x}^0_j,\delta_j)$. In general, however, $B({\bf x}^0_j,\delta_j)\not\subset K$, but we have $\Rev{{\rm dist}( \partial B({\bf x}^0_j,\delta_j),K)}  \le |{\bf x}^0_j-{\bf \tilde{x}}^0_j|<\delta_j$.  For the boundary pieces $(\partial K)_j$ with $B({\bf x}^0_j,\delta_j)\subset K$, we fix $\delta_j$ to its largest possible value ensuring $B({\bf x}^0_j,\delta_j)\subset K$. If, however, $B({\bf x}^0_j,\delta_j)\not\subset K$, we select $\delta_j$ small enough, so that $\Rev{{\rm dist} (\partial B({\bf x}^0_j,\delta_j),K)}\le\mathfrak{h}_{\overline{K}_j}(8p)^{-2}$. }
			
			{On the other hand, the $C^1$ smoothness of $(\partial K)_j$ ensures that there exists a finite tessellation comprising of diagonally scaled and rotated $(d-1)$-hypercubes approximating $(\partial K)_j$ to a desired accuracy, say $\delta_j/2$.  Consider now the truncated prisms intersecting $(\partial K)_j$ defined uniquely by the $2^{d-1}$ vertices of each element of the tessellation and the $2^{d-1}$ vertices of a second $(d-1)$-hypercubical base contained in $B({\bf x}^0_j,\delta_j)$ and passing through ${\bf x}^0_j$. The union of the latter generalized prisms covers $K_{(\partial K)_j}$ within a distance $\delta_j$.  Considering the corresponding construction for all $j$, we conclude the construction of a finite cover of $K$ by affinely mapped generalized prisms such that \eqref{tight_cover} holds.
			}
		\end{proof}\vspace{-.4cm}
		\begin{remark}
			The purpose of the construction in the proof of Lemma \ref{lem:tight_covering} is to assert the existence of at least \emph{one} covering with the required properties, and not to construct the `optimal' one.   
		\end{remark}

		\begin{lemma}\label{curved_inv_est_el}
			Let  $K\in\mathcal{T}$  {Lipschitz satisfying} Assumption \ref{assumption_mesh}. Then, for each $v\in\mathcal{P}_{p}(K)$, we have the inverse inequality
			\begin{equation}\label{eq:inv_gen_el}
				\|v\|_{F_i}^2\le \Rev{\mathcal{C}_{\rm INV}(p,K,F_i)} \frac{(p+1)(p+d)|F_i|}{|K|} \|v\|_{K}^2,
			\end{equation}
			with
			\begin{equation}
				\Rev{\mathcal{C}_{\rm INV}(p,K,F_i)
					: = \left\{
					\begin{array}{ll}
						\min \big\{\mathcal{C}_{\rm reg}(K,F_i) ,2c_{as}^{-1} 32^d p^{2(d-1)} \big\},  & \text{$K$ $p$-coverable,} 
						\vspace{2mm}	\\ 
						\mathcal{C}_{\rm reg}(K,F_i), &  \text{otherwise,} 
					\end{array}  
					\right. \label{inverse_constant_curve}}
			\end{equation}
			with \Rev{$\mathcal{C}_{\rm reg}(K,F_i):=|K|/\big(|F_i|\sup_{\uu{x}^0_i\in K}\min_{\mathbf{x}\in F_i}(\m\cdot\mathbf{n})\big)$, and} $c_{as}>0$  {as in Def.~\ref{def_poly_assumption}}.
		\end{lemma}
		\begin{proof}
			If $K$ is not $p$-coverable, using \eqref{eq:inv_gen}, we simply have 
			\begin{equation}\label{eq:inv_in_nopcov}
				\|v\|_{F_i}^2\le\frac{(p+1)(p+d)}{\displaystyle \min_{\mathbf{x}\in F_i}(\m\cdot\mathbf{n})}  \|v\|_{K_{F_i}}^2 \le \frac{(p+1)(p+d)|F_i|}{|K|}\Rev{\mathcal{C}_{\rm reg}(K,F_i)}\|v\|_{K}^2.
			\end{equation}
			If, on the other hand, $K$ is $p$-coverable, then $K_{F_i}\subset K\subset \cup_{j=1,\dots,m_K^{}} \overline{K}_j$ and, thus,
			\[
			\|v\|_{F_i}^2 \le |F_i|\|v\|_{L_{\infty}(F_i)}^2\le |F_i|\|v\|_{L_{\infty}(K_{F_i}^{})}^2\le|F_i| \max_{j=1,\dots,m_K^{}}\|v\|_{L_{\infty}(\overline{K}_j)}^2.
			\]
			{Now, Lemma \ref{basic_Linfty_L2} (together with an elementary scaling argument),} along with \eqref{T_tilde_condition} and \eqref{subel_pcover}, imply
			\[
			\|v\|_{L_{\infty}(\overline{K}_j)}^2\le  \Rev{32^d p^{2d} } |\overline{K}_j|^{-1}\|v\|_{\overline{K}_j}^2\le  2 \Rev{(32 p^{d})^2 } c_{as}^{-1}|K|^{-1}\|v\|_{K}^2.
			\]
			Combining the last two estimates, taking the supremum over $\uu{x}^0_i\in K$, the inverse estimate constant is then given by the minimum of the two estimates.
		\end{proof}

		The above result generalizes both \cite[Lemma 11]{book}
		and \cite[Lemma 4.9]{CGS18} in a number of ways.  {The coverings are now allowed to consist of curved domains. Also,} elements with arbitrary number of (curved) faces are now admissible and an earlier hypothesis on uniform boundedness of $m_K^{}$ across the mesh has now been removed by a more careful analysis. Note that, when $K\in\mathcal{T}$ is a polytopic element with straight faces, Lemma \ref{curved_inv_est_el} collapses to \cite[Lemma 11]{book}  {with improved constants.}

		\begin{remark}\label{remark:optimisation}
			The sub-division $\{F_i\}_{i=1}^{n_K^{}}$ of the (curved) element boundary $\partial K$ is typically \emph{not} unique.  We can seek to minimize the coefficient \eqref{inverse_constant_curve} by considering different candidates for $\{F_i\}_{i=1}^{n_K}$. However, such optimization would be practically beneficial only for rather ``exotic'' element shapes as, in most cases, we can simply resort to \eqref{stronger_assumption}. Of course, extremely general curved ``exotic'' element shapes must be used only when deemed beneficial for the particular problem at hand. In such cases, a basic geometric study for improving the constant \eqref{inverse_constant_curve} (and, therefore, as we shall see below, the dG discontinuity-penalization  {function}, cf. Remark \ref{remark:sigma} below) may be in order. In any case, Lemma \ref{curved_inv_est_el} is sharp for each given subdivision $\{F_i\}_{i=1}^{n_K^{}}$ and directly generalizes the inverse estimates in \cite{book}.
		\end{remark}

		Next, we present an $H^1-L_2$-inverse inequality  {for polynomials on a general curved element $K\in\mathcal{T}$. }
		
		\begin{lemma}\label{curved_inv_est_h1l2}
			Let $ K\in\mathcal{T}$ satisfy Assumptions \ref{assumption_mesh}  {and \ref{ass:c1}}.  
			Then, for each $v\in\mathcal{P}_{p}(K)$, the inverse estimate
			\begin{equation}\label{H1_L2_est}
				\|\nabla v\|_{K}^2\le  {\mathcal{C}^B_{\rm INV}}(p,K) \frac{p^4}{\rho_{K}^2} \|v\|_{K}^2,
			\end{equation}
			holds, with $\rho_\omega$ denoting the radius of the largest inscribed circle of a domain  $\omega\subset\mathbb{R}^d$,  and 
			{
				\begin{equation}
					\mathcal{C}^B_{\rm INV}(p,K)
					: = \left\{
					\begin{array}{ll}
						\Revthree{4}	\min\big\{\rho_{\rm cov}(K),\rho_{\rm p-cov}(p,K) \big\},  & \text{if $K$ $p$-coverable} 
						\vspace{2mm}	\\ 
						\Revthree{4}	\rho_{\rm cov}(K), &  \text{otherwise,} 
					\end{array}  
					\right. \label{inverse_constant_H1}
				\end{equation}
			}
			\Rev{with  			\begin{equation}\label{rho_cov}
					\rho_{\rm cov}(K):= \sum_{j=1}^{m_K^{}} C_{\rm inv}^B(d,\hat{r}_j)\Big(\frac{ \hat{r}_j\rho_K}{\rho_{\overline{K}_j}}\Big)^{2},
				\end{equation}
				and
				\begin{equation}\label{rho_pcov}
					\rho_{\rm p-cov}(p,K):=  c_{as}^{-1} (32)^d	(p-1)^{2d} \max_{1\le \ell\le m_K}C_{\rm inv}^B(d,\hat{r}_\ell)	\Big(\frac{\hat{r}_\ell\rho_K}{\rho_{\overline{K}_\ell}}\Big)^{2},
				\end{equation} 
				for $\overline{K}_\ell$, $\ell=1,\dots,m_K^{}$, cover of $K$ consisting of affinely mapped generalised prisms.}
		\end{lemma}
		\begin{proof}  {\Rev{From Lemma \ref{lem:tight_covering}, there exists a cover of $K$,} consisting of affinely mapped generalized prisms $\overline{K}_j$, $j=1,\dots,m_K^{}$. Thus,} for $v\in\mathcal{P}_{p}(K)$,  {Lemma \ref{basic_H1_L2}, (with a standard affine scaling)} and \eqref{subel_pcover}  imply:
			\[
			\|\nabla v\|_{K}^2\le\sum_{j=1}^{m_K^{}} 	\|\nabla v\|_{\overline{K}_j}^2\le \sum_{j=1}^{m_K^{}}  {C_{\rm inv}^B(d,\hat{r}_j)}\frac{ {\Revthree{2}\hat{r}_j^2}p^{4}}{\rho_{\overline{K}_j}^2}\|v\|_{\overline{K}_j}^2 \le  \Revthree{4} \sum_{j=1}^{m_K^{}}  {C_{\rm inv}^B(d,\hat{r}_j)}\frac{ {\hat{r}_j^2}p^{4}}{\rho_{\overline{K}_j}^2}\|v\|_{K}^2,
			\]
			{with $\hat{r}_j$ denoting the $\hat{r}$ of $\overline{K}_j$ as per Definition \ref{def:gen_prism}.} 
			Thus, we have
			\begin{equation}\label{h1l2_one}
				\|\nabla v\|_{K}^2\le  {\rho_{\rm cov}(K)} \frac{p^4}{\rho_K^2}\|v\|_K^2.
			\end{equation}
			Note that $\rho_{\rm cov}(K)$ grows with $\rho_K/\min_{j=1,\dots,m_K^{}}\rho_{\overline{K}_j}$ growing.

			On the other hand,  {if $K$ is $p$-coverable,} there exists $\ell\in \{1,\dots,m_K^{}\}$ such that 
			$
			\|\nabla v\|_{K}^2\le|K|\|\nabla v\|_{L_{\infty}(K)}^2\le 
			|K|\|\nabla v\|_{L_{\infty}(\overline{K}_\ell)}^2$. 
			{Using now Lemmata \ref{basic_Linfty_L2} and \ref{basic_H1_L2} (with scaling), as well as \eqref{T_tilde_condition2},} we deduce, respectively,
			\begin{equation}\label{h1l2_two}
				\begin{aligned}
					\|\nabla v\|_{L_{\infty}(\overline{K}_\ell)}^2\le&\ \Rev{32^d	(p-1)^{2d}}|\overline{K}_{\ell}|^{-1}\|\nabla v\|_{\overline{K}_\ell}^2\\
					\le&\ 32^d	(p-1)^{2d}p^{4} C_{\rm inv}^B(d,\hat{r}_\ell)\Revthree{2}\hat{r}_\ell^2\rho_{\overline{K}_\ell}^{-2}	|\overline{K}_{\ell}|^{-1}
					\| v\|_{\overline{K}_\ell}^2.
				\end{aligned}
			\end{equation}
			
			The result already follows by combining \eqref{h1l2_one} and \eqref{h1l2_two}.
		\end{proof}\vspace{-.0cm}
		{
		}
		{In the special case of an element $K$ being star-shaped with respect to a contained ball, we can have a more precise statement in terms of the constants involved.}
		{
			\begin{corollary}
				Let $K\subset \mathbb{R}^d$ domain which is star-shaped with respect to a ball $B({\bf x},\rho_K)\subset K$, ${\bf x}\in K$. Then, for any $v\in\mathcal{P}_p(K)$, we have the inverse estimate
				\[
				\|\nabla v\|_{K}^2\le C(d) \Big(\frac{h_K}{\rho_{K}}\Big)^{d+5} \frac{ p^4}{\rho_{K}^2}\|v\|_{K}^2,
				\]
				for some universal constant $C(d)>0$ that can be estimated explicitly. Thus, if additionally, $K$ is shape-regular, i.e., $h_K\sim \rho_K$, we retrieve the classical inverse estimate $\|\nabla v\|_{K}^2\le C p^4/h_{K}^2\|v\|_{K}^2,$ with $C$ now also dependent on the shape-regularity constant.
			\end{corollary}
		}
		
		{
			\begin{proof}
				We have $B({\bf x},\rho_K)\subset K\subset B({\bf x}, h_K)$. A comparison of the area of the largest $(d-1)$-hypercube contained in $B({\bf x},\rho_K)$, given by $\rho_K^{d-1}/2^{(d-1)/2}$, with the surface of $B({\bf x}, h_K)$, shows that we can cover $K$ using $z:=\lfloor 2^{(d+1)/2}\pi^{d/2} h_K^{d-1}/(\Gamma(d/2)\rho_K^{d-1})\rfloor +1$ mapped right generalized prisms $K_j$, $j=1,\dots, z$, whose bases are given by rotations of the largest $(d-1)$-hypercube contained in $B({\bf x},\rho_K)$. So, we have
				\[
				\|\nabla v\|_{K}^2\le \sum_{j=1}^z\|\nabla v\|_{K_j}^2\le \sum_{j=1}^zC_{\rm inv}^B(d,\hat{r}_j)\frac{\Revthree{2}\hat{r}_j^2 p^4}{\rho_{K_j}^2}\|v\|_{K_j}^2;
				\] here we have used scaling via $\Phi_j:\hat{K}_j\to K_j$ affine mapping $[0,1]^{d}$ to a rotation of the largest $d$-hypercube contained in $B({\bf x},\rho_K)$. Since each $K_j$ is right, we have $\hat{r}_j\le \lfloor h_K/\rho_{K_j}\rfloor +1$, for all $j=1,\dots,z$. Also, from the star-shapedness with respect to $B({\bf x},\rho_K)$, we have $\rho_{K_j}\ge \rho_{K}/2$. Combining the above, we deduce
				\[
				\|\nabla v\|_{K}^2\le \Revthree{8}C_{\rm inv}^B(\lfloor h_K/\rho_{K_j}\rfloor +1,d)(\lfloor h_K/\rho_{K_j}\rfloor +1)^2\frac{ p^4}{\rho_{K}^2}\sum_{j=1}^z\|v\|_{K_j}^2.
				\]
				Using the (pessimistic) bound $\sum_{j=1}^z\|v\|_{K_j}^2\le z \|v\|_{K}^2$, and combining the numerical constants, the result follows.
			\end{proof}
		}
		
		{
			The last result holds under weaker domain assumptions compared to \cite[Theorem 1]{Kroo1} and, in contrast to the main result in \cite{Kroo2},  it offers explicit dependence on the domain size in the case of piecewise $C^1$ domains. We also note \cite[Theorem 3]{Kroo1}, which provides a similar bound for the special case of $K$ being a $d-$ellipsoid, in conjunction with John's Ellipsoid Theorem.  It is interesting to investigate the extension of the above inverse estimates with explicit constants to cuspoidal domains in the spirit of \cite{Kroo2}; this will be considered elsewhere.
		}
		
		\begin{example}\label{example_circ_2}
			We revisit  {Example \ref{inv_circle} for $d=2$, with $K=B(0,R)$ a} circular element with radius $R$.  {Let $K_1=B(0,R)\cap\{(x_1,x_2): x_1\in\mathbb{R},\ -\sqrt{2}R/2\le x_2\le \sqrt{2}R/2\}$ and $K_2=B(0,R)\cap\{(x_1,x_2): -\sqrt{2}R/2\le x_1\le \sqrt{2}R/2,\ x_2\in\mathbb{R}\}$, so that $K=K_1\cup K_2$. We further subdivide each $K_i$ in half to form prisms with flat base; for each of these, we can select $\hat{r}=2$. Thus, Lemma \ref{curved_inv_est_h1l2} implies 
				\[
				\|\nabla v\|_{K}^2
				\le \mathcal{C}_{\rm INV}^B\frac{p^4}{R^2}\|v\|_{K}^2,
				\] 
				with $\mathcal{C}_{\rm INV}^B\le \Revthree{2304}$. The constant \Revthree{in this special case} can be improved considerably upon taking advantage of the circle's symmetries.} 
		\end{example}

		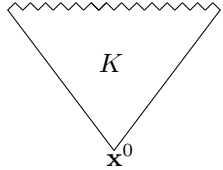
\begin{figure}[h]
			\centering
			\setlength{\unitlength}{2cm}
			\begin{picture}(-2,.9)
				\put(-1,0){\line(3,4){.7}}
				\put(-1,0){\line(-3,4){.7}}
				
				\put(-1.7,.93){\line(1,1){.05}}
				\put(-1.65,.98){\line(1,-1){.05}}
				\put(-1.6,.93){\line(1,1){.05}}
				\put(-1.55,.98){\line(1,-1){.05}}
				\put(-1.5,.93){\line(1,1){.05}}
				\put(-1.45,.98){\line(1,-1){.05}}
				\put(-1.4,.93){\line(1,1){.05}}
				
				\put(-1.3,.93){\line(-1,1){.05}}
				\put(-1.25,.98){\line(-1,-1){.05}}
				\put(-1.2,.93){\line(-1,1){.05}}
				\put(-1.15,.98){\line(-1,-1){.05}}
				\put(-1.1,.93){\line(-1,1){.05}}
				\put(-1.05,.98){\line(-1,-1){.05}}
				\put(-1.1,.93){\line(-1,1){.05}}
				\put(-1.05,.98){\line(-1,-1){.05}}
				\put(-1,.93){\line(-1,1){.05}}
				\put(-.95,.98){\line(-1,-1){.05}}
				\put(-.9,.93){\line(-1,1){.05}}
				\put(-.85,.98){\line(-1,-1){.05}}
				\put(-.8,.93){\line(-1,1){.05}}
				\put(-.75,.98){\line(-1,-1){.05}}
				\put(-.7,.93){\line(-1,1){.05}}
				\put(-.65,.98){\line(-1,-1){.05}}
				\put(-.6,.93){\line(-1,1){.05}}
				\put(-.55,.98){\line(-1,-1){.05}}
				\put(-.5,.93){\line(-1,1){.05}}
				\put(-.45,.98){\line(-1,-1){.05}}
				\put(-.4,.93){\line(-1,1){.05}}
				\put(-.35,.98){\line(-1,-1){.05}}
				\put(-.3,.93){\line(-1,1){.05}}
				
				\put(-1.1,.52){$K$}
				
				\put(-1.05,-.1){$\uu{x}^0$}
				
			\end{picture}
			\caption{Example \ref{example_rough}. $K\in\mathcal{T}$ with `multiscale' boundary behaviour.}
			\label{fig:rough_triangle}
		\end{figure}
		\begin{example}\label{example_rough}
			Let $d=2$, and consider the polygonal element $K\in\mathcal{T}$ with `multiscale' boundary behaviour depicted in Figure \ref{fig:rough_triangle}. Denoting by $r$ the length of each of the (equal length) $n$ small faces and with $h_K$ its diameter, we consider the case $r\ll h_K$. 
			If $r< h_K/p^2$, we can cover $K$ by one triangle {, namely,} the smallest simplex containing $K$.  {Then} $K$ is $p$-coverable and $ {\mathcal{C}_{\rm INV}^B}(p,K)$  {remains bounded, independently of $n$}. Hence, when the two geometric scales $h_K$ and $r$ are significantly different, $K$ is essentially a simplex in this context. 
			
			On the other hand, for $p$ large enough and fixed $r$ and $n$, we have $r> h_K/p^2$ and, hence, we cannot cover $K$ as before. Instead, we consider a family of $n$ non-overlapping simplices $\overline{K}_j\subset K$, each defined by one `small' face of length $r$ and the vertex $\uu{x}^0$. Then, we have $c_{as}=n^{-1}$ in Definition \ref{def_poly_assumption} and $\rho_{\overline{K}_j} \sim h_K/n$. Since also $\rho_K\sim h_K$ and $q_\ell=1$, we compute
			$
			{\mathcal{C}_{\rm INV}^B}(p,K)\sim n^{-1}. 
			$ This is reasonable, as sufficiently high polynomial degree $p$ basis functions can resolve the scale of the `sawtooth' face ensemble.
		\end{example}
		
		\subsection{ {Best approximation estimates}} We now turn to $hp$-version polynomial approximation bounds over general domains. The setting here remains essentially unchanged compared to the case of just polytopic elements presented in \cite{cangiani2013hp,book}. More specifically, under a mild set of covering assumptions and upon postulating the existence of so-called function space domain extension operators, we are able to apply $hp$-version best approximation results in various seminorms.

		\begin{definition}\label{ch3:definition_mesh_covering}
			Given a mesh $\mathcal{T}$, we define a \emph{covering} $\coveringmesh = \{ \mathcal{K} \}$ of $\mathcal{T}$ to be a   
			set of open shape-regular $d$--simplices  $\mathcal{K}$, such that for each $K\in\mathcal{T}$, 
			there exists a $\mathcal{K}\in\coveringmesh$ with $K\subset\mathcal{K}$. 
			For a given $\coveringmesh$, we define the \emph{covering domain $\bar{\Omega}_{\sharp}:=\cup_{\mathcal{K}\in\coveringmesh}\bar{\mathcal{K}}$.}
		\end{definition}

		\begin{figure}[h]
			\centering
			\setlength{\unitlength}{1.2cm}
			\begin{picture}(-1,1.6)
				\cbezier(0, 0)(-0.3, 0.65)(0.3, 0.4)(-.1, 1)	
				\cbezier(0, 0)(-0.4, 0.2)(-0.8, -0.2)(-1.1, 0)
				\cbezier(-1.1, 0)(-0.8, 0.1)(-0.7, 0.45)(-.7, .5)
				\cbezier(-.7, .5)(-0.8, 0.5)(-0.9, 0.6)(-1, .9)
				\cbezier(-1, .9)(-0.6, 1.5)(-0.3, 0.7)(-.1, 1)

				\put(-.4,.5){\makebox(0,0){$\color{black} K$}}
				
				\put(-.5,1.1){\makebox(0,0){$\color{black} \mathcal{K}$}}
				
				\put(-.033,-.01){\vx}
				\put(-.12,.97){\vx}
				\put(-1.12,-.019){\vx}
				\put(-1.02,.88){\vx}
				\put(-.72,.47){\vx}

				\put(.56,-.07){\line(-1,0){2.18}}
				\put(.56,-.07){\line(-3,5){1.085}}
				\put(-1.61,-.07){\line(3,5){1.085}}

			\end{picture}

			\caption{A simplex $\mathcal{K}\in\mathcal{T}^{\sharp}$ covering an element $K\in\mathcal{T}$.} 	\label{fig_covering_anis}
		\end{figure}
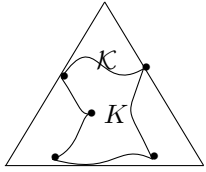

		For illustration, in Figure~\ref{fig_covering_anis} we show a single two-dimensional curved element 
		$K\in\mathcal{T}$, along with a covering simplex $\mathcal{K}\in\coveringmesh$ 
		with $K\subset\mathcal{K}$.
		
		\begin{assumption}\label{assumption_overlap}
			For a given mesh $\mathcal{T}$, we postulate the existence of a covering $\coveringmesh$, and of a (global) constant 
			$\mathcal{O}_{\Omega}\in\mathbb{N}$, independent of the mesh parameters, such that
			$$
			\max_{K \in \mathcal{T}} \mbox{card} 
			\Big\{ K' \in \mathcal{T} : 
			K' \cap \mathcal{K} \neq \emptyset, ~\mathcal{K}\in\coveringmesh ~\mbox{ such that } ~K\subset\mathcal{K} \Big\} \leq {\mathcal{O}}_{\Omega}.
			$$
			For such $\mathcal{T}^\sharp$, we further assume that
			$
			h_{\mathcal{K}}:=\diam(\mathcal{K})\le C_{\diam} h_{K},
			$ 
			for all pairs $K\in\mathcal{T}$, $\mathcal{K}\in\coveringmesh$, with $K\subset\mathcal{K}$, for a (global) constant $C_{\diam}>0$, uniformly with respect to the mesh size $h_K$.
		\end{assumption}

		\begin{remark}
			Assumption~\ref{assumption_overlap} ensures the shape--regularity 
			of the mesh covering $\mathcal{T}^\sharp$ only. Shape-regularity of the mesh $\mathcal{T}$ is \emph{not} assumed. We refer to Figure 3.6 in \cite{book} for an example on how these two concepts may differ considerably.
		\end{remark}
		
		The validity of Assumption \ref{assumption_overlap} allows for the \Rev{use} known $hp$--version approximation estimates on simplicial elements~\cite{Babuska-Suri:RAIRO:1987,Babuska-Suri:SINUM:1987,schwab}, on each $\mathcal{K}$ and, subsequently restrict the error over $K\subset \mathcal{K}$. However, it requires to extend the exact solution $u$ into $\Omega_{\sharp}$ in a stable fashion. To that end, we shall use the following classical result.
		
		\begin{theorem}[\cite{stein}] \label{thm-extension}
			Let $\Omega$ be a domain with a Lipschitz boundary. Then there exists a linear extension
			operator $\frak{E}:H^s(\Omega) \mapsto H^s({\mathbb R}^d)$, $s \in {\mathbb N}_0$, such that
			$\frak{E}v|_{\Omega}=v$ and
			$
			\| \frak{E} v \|_{H^s({\mathbb R}^d)} \leq C_{\frak{E}} \| v \|_{H^s(\Omega)},
			$
			where the constant $C_{\frak{E}}>0$ depends only on $s$, $\Omega$.
		\end{theorem}
		Subsequent refinements of the dependence of the constant $C_{\frak{E}}$ on the domain shape  in Theorem \ref{thm-extension}, have been presented for instance in \cite{sauter_warnke_1999, carstensen2004complicated}. 
		
		For the estimation of the best approximation error on the mesh skeleton $\Gamma_{\dint}\cup\partial\Omega$, we require the trace estimate on curved domains from Lemma \ref{lem:trace}. 
		
		We now have all the ingredients to assert the validity of the following $hp$-approximation error bounds.

		\begin{lemma}\label{lemma_approx_EASE}
			Let $K\in\mathcal{T}$ satisfy Assumptions \ref{assumption_mesh} and  \ref{assumption_overlap}, and let
			$\mathcal{K}\in\coveringmesh$ be the corresponding simplex 
			with $K\subset\mathcal{K}$ as per Definition~\ref{ch3:definition_mesh_covering}.
			Suppose that $v\in L_2(\Omega)$ is such that 
			$\frak{E}v|_{\mathcal{K}}\in H^{l_K}(\mathcal{K})$, for some $l_K\ge 0$, and that
			Assumption \ref{assumption_overlap} is satisfied. Then, there exists an operator $\pi_p: H^{l_K}(\mathcal{K}) \to \mathcal{P}_p(\mathcal{K})$, such that
			\begin{equation}\label{approxH_k}
				\norm{v - \pi_p v}{H^q(K)} 
				\le C_1 \frac{h_{K}^{s_{K}-q}}{p^{l_{K}-q}}\norm{\frak{E}v}{H^{l_{K}}(\mathcal{K})},\quad l_K\ge 0,
			\end{equation}
			for $0\le q\le l_K$, and 
			\begin{equation}\label{approx-inf_k}
				\norm{v - \pi_p v}{F_i}
				\le 	\mathcal{C}_{ap}^{1/2}(p,K,F_i) |F_i|^{1/2}\frac{h_{K}^{s_{K}-d/2}}{p^{l_{K}-1/2}} 
				\norm{\frak{E} v}{H^{l_K}(\mathcal{K})}, ~ l_{K}>d/2,~~~~~~~~
			\end{equation}
			with
			$$
			\mathcal{C}_{ap}(p,K,F_i) 
			:= C_2 \min \Big\{h_K^d
			\big(	|F_i|\sup_{\uu{x}^0_i\in K}
			\min_{\mathbf{x}\in F_i}(\m\cdot\mathbf{n})\big)^{-1},p^{d-1} \Big\},
			$$
			$s_{K}=\min\{p+1, l_{K}\}$, and $C_1,C_2>0$ constants
			depending only on the shape-regularity of ${\mathcal{K}}$, $q$, $l_K$, on $C_{\diam}$ (from Assumption \ref{assumption_overlap})  and on the domain $\Omega$. 
		\end{lemma}
		\begin{proof}
			Let $\Pi_p:H^{l}(\mathcal{K})\to\mathcal{P}_p(\mathcal{K})$ be a known optimal $hp$-version approximation operator on simplices, see, e.g., \cite{Babuska-Suri:RAIRO:1987,Babuska-Suri:SINUM:1987,schwab}. We define $\pi_p:H^{l}(\mathcal{K})\to\mathcal{P}_p(\mathcal{K})$ by $\pi_p v:= \Pi_p(\frak{E}v)$. 
			To prove \eqref{approxH_k}, we begin by observing that
			\[
			\norm{v - \pi_p v}{H^q(K)} = \norm{\frak{E}v - \Pi_{p}(\frak{E} v)}{H^q(K)}
			\le \norm{\frak{E}v - \Pi_{p}(\frak{E} v)}{H^q(\mathcal{K})}.
			\]
			Thus,
			Assumption \ref{assumption_overlap} and standard $hp$-approximation estimates on simplices (e.g. \cite{Babuska-Suri:RAIRO:1987,Babuska-Suri:SINUM:1987,schwab} yield 
			the desired bound; we refer to the proof of \cite[Lemma 3.7]{book} for a similar argument for polytopic elements.
			
			To prove \eqref{approx-inf_k}, we \Rev{use}
			the trace inequality \eqref{eq:trace} with $\zeta = p$ \Rev{to deduce} 
			\begin{equation}\label{face1}
				\norm{v - \pi_p v}{F_i}^2 
				\le \Rev{ C\frac{h_K}{\displaystyle\sup_{\uu{x}^0_i\in K} \min_{\uu{x}\in F_i}(\m\cdot\mathbf{n})}} \frac{h_{K}^{2s_{K}-1}}{p^{2l_{K}-1}}\norm{\frak{E}v}{H^{l_{K}}(\mathcal{K})}^2,
			\end{equation}
			noting that $ \max_{\uu{x}\in F_i}|\m|_2^2\le h_K^2$. 			
			On the other hand, we observe that
			$$
			\norm{v - \pi_p v}{F_i}^2 \le |F_i |\norm{v - \pi_p v}{L_{\infty}(K_{F_i})}^2\le  \norm{\frak{E}v - \Pi_{p}(\frak{E} v)}{L_{\infty}(\mathcal{K})}^2.
			$$ 
			Hence, employing a classical $hp$-approximation estimate for the maximum norm error from \cite{Babuska-Suri:RAIRO:1987,Babuska-Suri:SINUM:1987}, (cf. also \cite[Lemma 20]{book} 
			we arrive at
			\begin{equation}\label{face2}
				\norm{v - \pi_p v}{F_i}^2  \le C |F_i|
				\frac{h_{K}^{2s_{K}-d}}{p^{2l_{K}-d}}\norm{\frak{E}v}{H^{l_{K}}(\mathcal{K})}^2,
			\end{equation}
			for $l_{K}>d/2$.
			The result follows by taking  the minimum 
			between the bound in \eqref{face1} and the bound in  \eqref{face2}.
		\end{proof}
		
		\begin{remark}
			We note the correspondence between $\mathcal{C}_{ {\rm INV}}(p,K,F_i)$ from Lemma \ref{curved_inv_est_el} and  $\mathcal{C}_{ap}(p,K,F_i)$, in the typical case $h_K^d\sim |K|$. The key attribute of both expressions is that they remain bounded  for degenerating $|F_i|$, allowing for the estimates \eqref{eq:inv_gen_el} and  \eqref{approx-inf_k} to remain finite as $|F_i|\to 0$.
		\end{remark}
		
		\begin{remark}
			If the constant $\mathcal{C}_{ap}(p,K,F_i) $ in \eqref{approx-inf_k} is taken with the first term, then the approximation result \eqref{approx-inf_k} holds with  $l_K > 1/2$.
		\end{remark}
		
		\section{A priori error analysis}\label{sec:aprior}
		
		We are now ready to \Rev{briefly discuss} a priori error bounds for sufficiently smooth exact solutions, thereby generalizing the respective results presented in \cite{book} \Rev{to the case of curved} polytopic meshes. The line of argument is \Rev{similar to the case of straight} polytopic meshes presented in detail in \cite{book}.
		
		A crucial ingredient of the analysis for the proof of stability of the dG-EASE method is the precise definition of the discontinuity-penalization function $\sigma$ appearing in the method \eqref{adv_galerkin_dg}. It is important to define $\sigma$ sufficiently large for stability, while at the same time not substantially larger than what is required, as it could potentially cause loss of accuracy and/or conditioning issues.  {Additionally, following \cite{book}, we provide a stronger inf-sup stability result with respect to a `steamline-diffusion'-type augmented norm, when the wind ${\bf b}$ is non-zero. The size of the `steamline-diffusion' coefficient depends crucially on Lemma \ref{curved_inv_est_h1l2}, whose constant provides information on the stabilization capabilities of the method.}

		The dG norm for which we seek to prove a priori error bounds is given by
		$
		\ndg{v} :=  \big(\ndg{v}_{\rm ar}^2+\ndg{v}_{\rm d}^2\big)^{1/2},
		$	where 
		\[
		\ndg{v}_{\rm ar}^2 = \ltwo{c_0 v}{}^2 +\frac{1}{2}\sum_{K\in\mathcal{T}} \big(\| \sqrt{|\mathbf{b}\cdot\mathbf{n}|} \ujump{v}\|^2_{\partial_- K } 
		+  \|\sqrt{|\mathbf{b}\cdot\mathbf{n}|}  v\|^2_{{\partial_+ K}\cap \partial\Omega}  \big),
		\]
		with $c_0$ given in \eqref{assumption-cb}, and
		$
		\ndg{v}^2_{\rm d}=\ltwo{\sqrt{a} \nabla_{\mathcal{T}}^{} v}{} ^2 +\ltwo{\sqrt{\sigma} 
			\jump{v}}{\Gamma_{\dint}\cup\partial\Omega_{\rm D}}^2.
		$

		\begin{definition}\label{interfaces}
			For a mesh $\mathcal{T}$, we define the set $\mathcal{F}_{\dint}$ of \emph{interfaces} $F\subset \Gamma_{\dint}$ by
			\[
			\mathcal{F}_{\dint}:=\{ F\subset \Gamma_{\dint}: \text{ there exist } K,K'\in\mesh\text{ with } F=\partial K\cap\partial K'\};
			\]
			correspondingly, we set $\mathcal{F}_{\ddd}:=\{ F\subset \partial\Omega_{\ddd}: \text{ there exists } K\in\mesh\text{ with } F=\partial K\cap\partial \Omega_{\ddd}\}$. For notational compactness, we also define $\mathcal{F}_{\dint,\ddd}:=\mathcal{F}_{\dint}\cup\mathcal{F}_{\ddd}$.
			(Note that $F$ may comprise of one or more faces of $K,K'$.) Moreover, each interface $F$ may be contained in one or more $F_i$'s of the elements $K,K'$ as per Assumption \ref{assumption_mesh}. Thus, there exists a subset $\{F_i^K\}_{i\in I^K_F}$ with index set $I^K_F\subset \{1,\dots, n_K^{}\}$, such that $F\subset \cup_{i\in I_K^F} F_i^K$ and, correspondingly, a set $I^{K'}_F\subset \{1,\dots, n_{K'}^{}\}$ such that $F\subset \cup_{i\in I_{K'}^F} F_i^{K'}$.
		\end{definition}
		
		For technical reasons (cf. \cite{book} and the references therein), we shall make use of the following extensions $\tilde{B}_{\rm d}:(H^1(\Omega)+\fes)\times(H^1(\Omega)+\fes)\to\mathbb{R}$ and $\tilde{\ell}:(H^1(\Omega)+\fes)\to\mathbb{R}$ of the bilinear and linear forms \eqref{diffusion_bilinear} and \eqref{adv-dg-linear}, \Rev{which are given replacing $ \mean{a \nabla w }$ and  $\mean{a \nabla v }$ with $ \mean{a \Pi\nabla w }$ and  $\mean{a \Pi\nabla v }$  in $B_{\rm d}$ and $\ell$, respectively, where} $\Pi: [L_2(\Omega)]^d\to [\fes]^d$ denotes the orthogonal $L_2$-projection operator onto the (vectorial) finite element space. Observe that $\tilde{B}_{\rm d}(w,v)=B_{\rm d}(w,v)$ and $\tilde{\ell}(v)=\ell(v)$ when $w,v\in \fes$. Similarly, we define $\tilde{B}(w,v):=\tilde{B}_{\rm d}(w,v)+{B}_{\rm ar}(w,v) $. Next, we discuss the coercivity and continuity of $\tilde{B}_{\rm d}$.
		
		\begin{lemma}
			Let~\eqref{assumption_a} hold and consider a mesh $\mathcal{T}$ satisfying Assumption \ref{assumption_mesh}.  {With the notation introduced in Definition \ref{interfaces}, d}efine the discontinuity-penalization 
			function $\sigma:\Gamma_{\dint}\cup\partial\Omega_{\ddd} \rightarrow \mathbb{R}$  {for every} interface $F\in \mathcal{F}_{\dint,\ddd}$, $F=\partial K\cap \partial K'$,  {by}
			\begin{equation}\label{sigma}
				\sigma|_F:=
				2  
				\max_{\mathcal{K}\in \{K,K'\}}\Big\{  |I_F^\mathcal{K}|\max_{i\in I_F^\mathcal{K}}\big\{\Rev{\mathcal{C}_{\rm INV}
					({p}_\mathcal{K}^{},\mathcal{K},F_i^\mathcal{K})}|F_i^\mathcal{K}|\big\} \frac{ {\bar{a}_K}({p}_\mathcal{K}^{}+1)({p}_{\mathcal{K}}^{}+d)}{|\mathcal{K}|} \Big\};
			\end{equation}
			when $F\in\mathcal{F}_{\ddd}$ we set $K=K'$. Then,  {for all $w,v\in H^1(\Omega)+\fes$,} we have
			\begin{equation}\label{eq:coer_cont}
				\tilde{B}_{\rm d}(w,w)\ge\ \frac{1}{2}\ndg{w}_{\rm d}^2\quad\text{and}\quad
				\tilde{B}_{\rm d}(w,v)\le\ 2\ndg{w}_{\rm d}\ndg{v}_{\rm d}.
			\end{equation}
		\end{lemma}
		\begin{proof} The idea of proof is standard and makes use of the trace inverse estimate developed above. The novel attribute here is the choice of $\sigma$ which requires some care since the star-shapedness of each interface $F$ may correspond to different boundary segments $F_i$ in either side of the interface. To that end, for $w\in H^1(\Omega)+\fes$, we have
			$
			\tilde{B}_{\rm d}(w,w) \ge \ndg{w}_{\rm d}^2
			-  2\int_{\Gamma_\text{int}\cup \partial\Omega_\text{D}}  \mean{a \Pi  \nabla w }\cdot \jump{w}  \ud s.
			$

			Therefore, Lemma \ref{curved_inv_est_el} and the stability of the orthogonal $L_2$-projection give
			\[
			\norm{ \Pi \sqrt{a}\nabla w}{F\cap \partial K}^2
			\le  |I_F^K|\max_{i\in I_F^K}\big\{\Rev{\mathcal{C}_{\rm INV}({p}_K^{},K,F_i^K)}|F_i^K|\big\} \frac{({p}_K^{}+1)({p}_{K}^{}+d)}{|K|} \|\sqrt{a}  \nabla w\|_{K}^2,
			\]
			Coercivity already follows by  {a Young's inequality}. The proof of continuity is standard and, therefore, omitted for brevity.
		\end{proof}
		
		\begin{remark}\label{remark:sigma}
			The stability of the dG-EASE method is guaranteed under extremely general mesh assumptions thanks to the judicious choice of the penalization  {function} \eqref{sigma}.
			As discussed also in Remark \ref{remark:optimisation}, the latter ultimately depends on the choice of subdivisions  $\{F_i\}_{i=1}^{n_K^{}}$ of  $\partial K$ 
			appearing in Assumption \ref{assumption_mesh}.  
			Of course, whenever possible, by simply following the recipe in Remark \ref{remark:comments}({\rm ii}), we can easily arrive at a practical value of the penalization  {function} for general curved elements.
		\end{remark}

		We shall additionally assume for simplicity  {of the presentation} that \begin{equation}\label{bdotxi}
			\mathbf{b}\cdot\nabla \xi \in \fes,\qquad\text{ for all }\ \xi\in\fes,
		\end{equation} as is a standard in this context, cf.~\cite{newpaper} and also \cite[Chapter 5]{book}. Assumption \eqref{bdotxi} can be further relaxed at the expense of an additional mild suboptimality with respect to the polynomial degree $\uu{p}$; see  {\cite[Remark 3.13]{newpaper}} and Remark \ref{rem:gen_b} below. 

		\begin{theorem} \label{infsup_theorem}
			Let $\mesh=\{K\}$ a subdivision of $\Omega\subset\mathbb{R}^d$, consisting of, possibly curved, elements satisfying Assumptions~\ref{assumption_mesh}, {~\ref{ass:c1}}  and \ref{assumption_overlap}.  Then, assuming that  {\eqref{bdotxi} holds} and that the discontinuity-penalization function $\sigma$ is given by
			\eqref{sigma},
			there exists a constant $\Lambda_s>0$, independent of $h$ and of $\vecp$, such that:
			\begin{equation}\label{Inf-sup1}
				\inf_{{w\in \fes\backslash \{0\}}} \sup_{{v\in \fes  
						\backslash \{0\}}} \frac{ {B}(w,v)}{\nsdg{w} \nsdg{v}}\geq \Lambda_s,
			\end{equation}
			with
			$
			\nsdg{v}:=\Big(\ndg{v}^2+\sum_{K\in\mesh} \lambda_K \ltwo{ \bold{b} \cdot \nabla{v}}{K}^2\Big)^{1/2},
			$
			whereby
			\begin{align}\label{SP}
				\lambda_K:=& \frac{\min \bigg\{ \displaystyle\frac{\rho_K }{ {\sqrt{\mathcal{C}^B_{\rm INV} (p_K,K)}}}
					,
					\Big( \sum_{F\subset\partial K}  
					\sum_{i\in I_F^K}\Rev{\mathcal{C}_{\rm INV}({p}_K^{},K,F_i^K)} \frac{|F_i^K|}{|K|}\Big)^{-1}  \bigg\}}{\max\{    \linf{\bold{b}}{K}  ,   
					\sigma_K   \} (p_K+1)(p_K+d)}
				\nonumber,
			\end{align}
			for $K \in \mesh$,  $p_K \geq 1$,  $\sigma_K := \max \{\sigma_{K}^a,\sigma_{K}^b\}$, with
			$
			\sigma_{K}^a  := \max_{F \subset \partial K} \sigma|_F,
			$		
			and
			\begin{eqnarray}\label{sigma_K2}
				\sigma_{K}^b  &:=& 2\max_{F \subset \partial K} \Big\{ {\displaystyle  \max_{\mathcal{K}\in \{K,K'\}}
					\Big\{ 
					{ {\sqrt{\mathcal{C}^B_{\rm INV} ({p}_\mathcal{K}^{},\mathcal{K})}}} \frac{\bar{a}_\mathcal{K}({p}_\mathcal{K}^{}+1)({p}_{\mathcal{K}}^{}+d)}{\rho_\mathcal{K}} } \Big\} \Big\}.
			\end{eqnarray}	
			
		\end{theorem}
		\begin{proof} \Rev{The proof follows in a completely analogous fashion to the proof of \cite[Theorem 5.2]{book} and is, therefore, largely omitted for brevity: the key idea is to set $v\equiv v(w):=w+\alpha w_s$, for $w\in\fes$ with $w_s|_K:=\lambda_K \mathbf{b}\cdot\nabla w$, $K\in \mesh$, with $\mathbb{R}\ni \alpha>0$. Then, it is sufficient to prove that 
				$
				\nsdg{v} \leq C^* \nsdg{w},
				$
				and
				$
				B(w,v)\geq C_*\nsdg{w}^2,
				$
				and then to set $\Lambda_s=C_*/{C^*}$, for some $C_*,C^*>0$ constants independent of the discretization parameters. The last two conditions are proven by using the inverse estimates above along with a judicious use of $\alpha$.}		
		\end{proof}
		
		\begin{remark}
			Theorem \ref{infsup_theorem} extends respective results on polytopic meshes from
			\cite{cangiani2015hp,book},
			to  general meshes consisting of polytopic and/or curved  elements with arbitrary number of faces.  {Moreover, this choice} removes a dependence of the inf-sup constant $\Lambda_s$ on the inverse inequality constants ${C}_{\rm INV}$ and $ {\mathcal{C}^B_{\rm INV}}$; cf. \cite{cangiani2015hp, book}.
		\end{remark}
		\begin{theorem} \label{thm:apriori_all}
			Let $\mesh=\{K\}$ be a subdivision  of $\Omega\subset\mathbb{R}^d$, consisting of general curved elements satisfying Assumptions~\ref{assumption_mesh},~ {\ref{ass:c1}}  and \ref{assumption_overlap}. Let also $\coveringmesh=\{\mathcal{K}\}$ an associated covering of $\mesh$ consisting of shape-regular simplices as per Definition~\ref{ch3:definition_mesh_covering}.  {Assume that \eqref{bdotxi} holds.}
			Assume that $u\in H^1(\Omega)$ the exact solution to \eqref{pde},\eqref{pde_bcs},  {is} such that $u|_K \in H^{l_{K}}(K)$, $l_K>1+d/2$, for each $K \in \mesh$.
			Let $u_h\in \fes$, with $p_K \geq1$, $K\in\mesh$, be the solution of \eqref{adv_galerkin_dg}, with $\sigma$ as in \eqref{sigma}.
			Then, we have
			\begin{equation*}
				\nsdg{u-u_{h}}^2 \le  C \hspace{-0.1cm} \sum_{K\in\mesh} \frac{h_{{K}}^{2s_{K}}}{p_{{K}}^{2l_{K}}} 
				\left( \mathcal{D}_{K}(F,\mathcal{C}_{ap},\lambda_K,p_{K})
				+  \mathcal{G}_{K}(F,{\mathcal{C}}_{\rm INV},	\mathcal{C}_{ap},p_{K})
				\right)
				\|{\frak E} u\|_{H^{l_{{K}}}(\mathcal{K})}^2,
			\end{equation*}
			with $s_{{K}}=\min\{p_{{K}}+1,l_{K}\}$, 
			\begin{align*}
				\mathcal{D}_K (F,	\mathcal{C}_{ap},\lambda_K,p_K) 
				&=  
				\norm{c_0}{L_\infty(K)}^2+\zeta_K^2+\lambda_K^{-1}+\lambda_K\beta_K^2 {p_K^{2}}{h_K^{-2}}  +\bar{a}_K {p_K^{2}}{h_K^{-2}}
				\\ 
				&+
				\beta_K h_{K}^{-d} p_{{K}}^{} \!\! 
				\sum_{F\subset\partial{K}}
				|I_F^K|\max_{i\in I_F^K}\big\{\mathcal{C}_{ap}({p}_K^{},K,F_i^K)|F_i^K|\big\}	, 
			\end{align*}	
			and
			\begin{align*}
				\mathcal{G}_K (F,{\mathcal{C}}_{\rm INV},	\mathcal{C}_{ap},p_K) 
				& =
				\bar{a}_K^2 p_K^3 h_K^{-d-2} 
				\kern -0.6cm \sum_{F\subset\partial{K}\cap (\Gamma_{\dint}\cup\partial\Omega_{\ddd})}\kern-.8cm
				\sigma^{-1} |I_F^K|\max_{i\in I_F^K}\big\{\mathcal{C}_{ap}({p}_K^{},K,F_i^K)|F_i^K|\big\}	
				\\
				&\hspace{-1cm} 
				+ \bar{a}_{K}^2 p_{{K}}^{4}h_K^{-2} |{K}|^{-1} 
				\kern -0.6cm \sum_{F\subset\partial{K}\cap (\Gamma_{\dint}\cup\partial\Omega_{\ddd})}\kern-.8cm
				\sigma^{-1}
				|I_F^K|\max_{i\in I_F^K}\big\{\Rev{\mathcal{C}_{\rm INV}({p}_K^{},K,F_i^K)}|F_i^K|\big\}	
				\\ 
				&\hspace{-0cm} +
				h_{K}^{-d} p_{{K}}^{} \!\! \kern-.5cm
				\sum_{F\subset\partial{K}\cap (\Gamma_{\dint}\cup\partial\Omega_{\ddd})}\kern-.5cm
				\sigma |I_F^K|\max_{i\in I_F^K}\big\{\mathcal{C}_{ap}({p}_K^{},K,F_i^K)|F_i^K|\big\}	, 
			\end{align*}	
			$s_{K}=\min\{p_{K}+1,l_{K}\}$,
			$\zeta_K := \norm{c/c_0}{L_\infty(K)}$, $c_0$ is in \eqref{assumption-cb}, 
			$\beta_K :=\norm{\bold{b}}{L_\infty(K)} $, and $C$ is a positive constant, which depends on the shape-regularity
			of $\coveringmesh$, but is independent of the 
			discretization parameters.
			
			{In the special case in which the coefficient $a$ is strictly positive definite a.e. in $\Omega$ while  $\mathbf{b}=\mathbf{0}$, Assumption \ref{ass:c1} can be removed from the hypotheses.}
		\end{theorem}
		\begin{proof}
			The proof follows on very similar lines to the respective one for polytopic meshes and can be found in \cite[Section 5.2]{book}.
		\end{proof}
		
		The above $hp$--version {\em a priori} error bounds hold
		without any assumptions on the relative size of the 
		faces $F$, $F\subset \partial K$, of a given curved element $K \in \mesh$. 
		To aid the understanding of the rates of convergence resulting from the above results, we set $p_K =p \geq1$,  
		$h=\max_{K\in\mesh} h_K$, $s_K=s$, $s=\min\{p+1,l\}$, and $l>1+d/2$,
		and assume that $\mbox{diam}(F) ~\sim h_K$, 
		for all faces $F\subset \partial K$, $K\in \mesh$, so that $|F|\sim h_K^{(d-1)}$. Then, Theorem \ref{thm:apriori_all} 
		reduces to
		\begin{equation*}
			\ndg{u-u_{h}}_{\rm d}  \leq C \frac{h^{s-1}}{p^{l-\frac{3}{2}}} \| u\|_{H^{l}(\Omega)},
		\end{equation*}
		i.e., it proves optimal convergence in $h$ and  suboptimal in $p$ by $p^{1/2}$.
		
		At the other end of the spectrum, consider the case of transport equation, i.e., when $a\equiv {\bf 0}$. In this case, the dG norm $\ndg{\cdot}$ degenerates to $\ndg{\cdot}_{\rm ar}$; note that, then we have  $\lambda_K=\mathcal{O}(h_K/p_K^2)$, and the {\em a priori}  error bound in Theorem \ref{thm:apriori_all}  reduces to
		\begin{equation*}
			\ndg{u-u_{h}}_{\rm ar} \le \frac{h^{s-\frac{1}{2}}}{p^{l-1}} \| u\|_{H^{l}(\Omega)}.
		\end{equation*}
		This bound is, again, optimal in $h$ and  suboptimal in $p$ by $p^{1/2}$ and completely generalizes the error estimate derived in our previous work \cite{cangiani2015hp} to essentially arbitrarily-shaped meshes under the same assumption  {\eqref{bdotxi}}.
		
		{
			\begin{remark}
				We remark on typical cases which result to simplified formulas for $\lambda_K$. Assuming $|K| ~\sim h_K^d$, $\rho_K\sim h_K$, and $|F|\sim h_K^{(d-1)}$ for an element $K\in\mathcal{T}$ and for  its immediate neighbours, both constants $ {\mathcal{C}^B_{\rm INV}}$ and   ${\mathcal{C}}_{\rm INV}$ will be defined by the first term in the maxima in \eqref{inverse_constant_curve} and \eqref{inverse_constant_H1}, respectively. Then, we deduce $\lambda_K\sim h_K/p_K^2$ for the important case of advection-dominated problems. 
			\end{remark}
		}
		\begin{remark}\label{rem:gen_b}
			For general  advection fields $\bold{b}$, the proof of the inf-sup condition needs to be modified by using a slightly different norm involving $\Pi(\bold{b} \cdot \nabla_{\mathcal{T}})$ instead of $(\bold{b} \cdot \nabla_{\mathcal{T}})$ in the $s$-norm, yielding an error  bound  which is optimal in $h$ but  suboptimal in $p$ by $p^{3/2}$ for the purely hyperbolic problem. Of course, if we modify the method by including a streamline-diffusion stabilization term as done in  \cite{hss}, then an $hp$-optimal bound can be derived without  {enforcing \eqref{bdotxi}}. 
		\end{remark}

		\section{Numerical examples} \label{numerical example}
		We test  {the dG-EASE approach 
			through a series of numerical experiments using curved elements, ranging from basic domain approximation to highly complex element shapes arising from random element agglomeration} 
		of a fine background triangulation.
		
		{In the case of the agglomeration-constructed elements,} the background  {(curved)} triangulation is also used for the assembly step. In particular, the discontinuity-penalisation function $\sigma|_F$ is fixed following the recipe in \eqref{sigma} with the subdivisions $\{F_i\}_{i=1}^{n_{K}^{}}$ of  $\partial K$, $K\in\mathcal{T}$, appearing in Assumption \ref{assumption_mesh}, given by unions of faces of the background triangulation. Moreover,  {for simplicity}
		the background triangulation is also used for integration, exploiting parallellization of the quadrature process~\cite{Kappas_2020}, see also \cite{book} for a more detailed discussion of implementation of such methods.  {Nonetheless, very often it is possible to use substantially coarser subdivisions than the background triangulation the elements have been constructed from, e.g., a subdivision with one simplex per straight face.}  
		
		{For curved elements, the current implementation performs quadrature by a sufficiently fine sub-triangulation approximating the curved element, exactly as in the agglomerated-element case. We stress, however, that in this case the sub-triangulation is only used to generate the quadrature rules.} These calculations are fully parallelizable {:} in~\cite{Kappas_2020} it is shown that quadrature cost becomes irrelevant if modern GPU architectures are used in the implementation of assembly.  {Of course, this is \emph{not} the only possibility. For instance, domain-exact quadrature algorithms for many curved domains exist, see, e.g., \cite{artioli2020algebraic} and the references therein for such algorithms.}

		\subsection{ {Example 1: curved elements}}  {We begin by testing the method on triangular elements with (non-parametric) curved faces.  Specifically, we consider a two-dimensional diffusion problem with $a= I_{2\times 2}$, $ I_{2\times 2}$ denoting the $2\times 2$-identity matrix,  $\mathbf{b} = (0,0)^\top$, $c=0$ and $f$ so that $u(x_1,x_2)= \sin(\pi x_1) \sin(\pi x_2)$ in \eqref{pde}. We solve this problem on an irregular annular domain constructed as the unit disc centred at origin, with a circular hole centred at $(0.25,0.25)$  and radius $0.4$ removed; cf. Figure \ref{fig:ex3_three_p} for an illustration.
		}

		\begin{figure}
			\includegraphics[height=5.2cm,width=5.8cm]{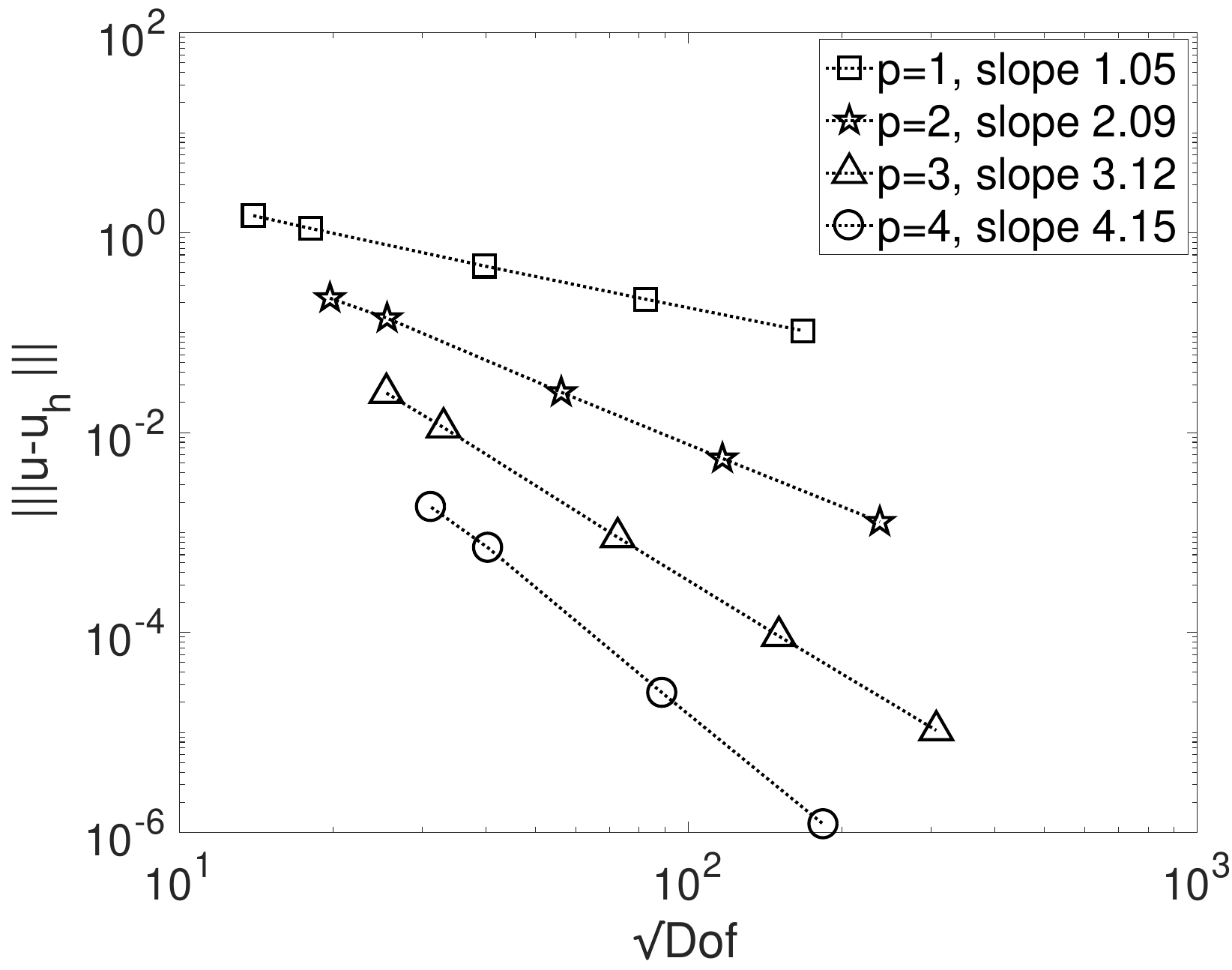}
			\includegraphics[height=5.2cm,width=5.8cm]{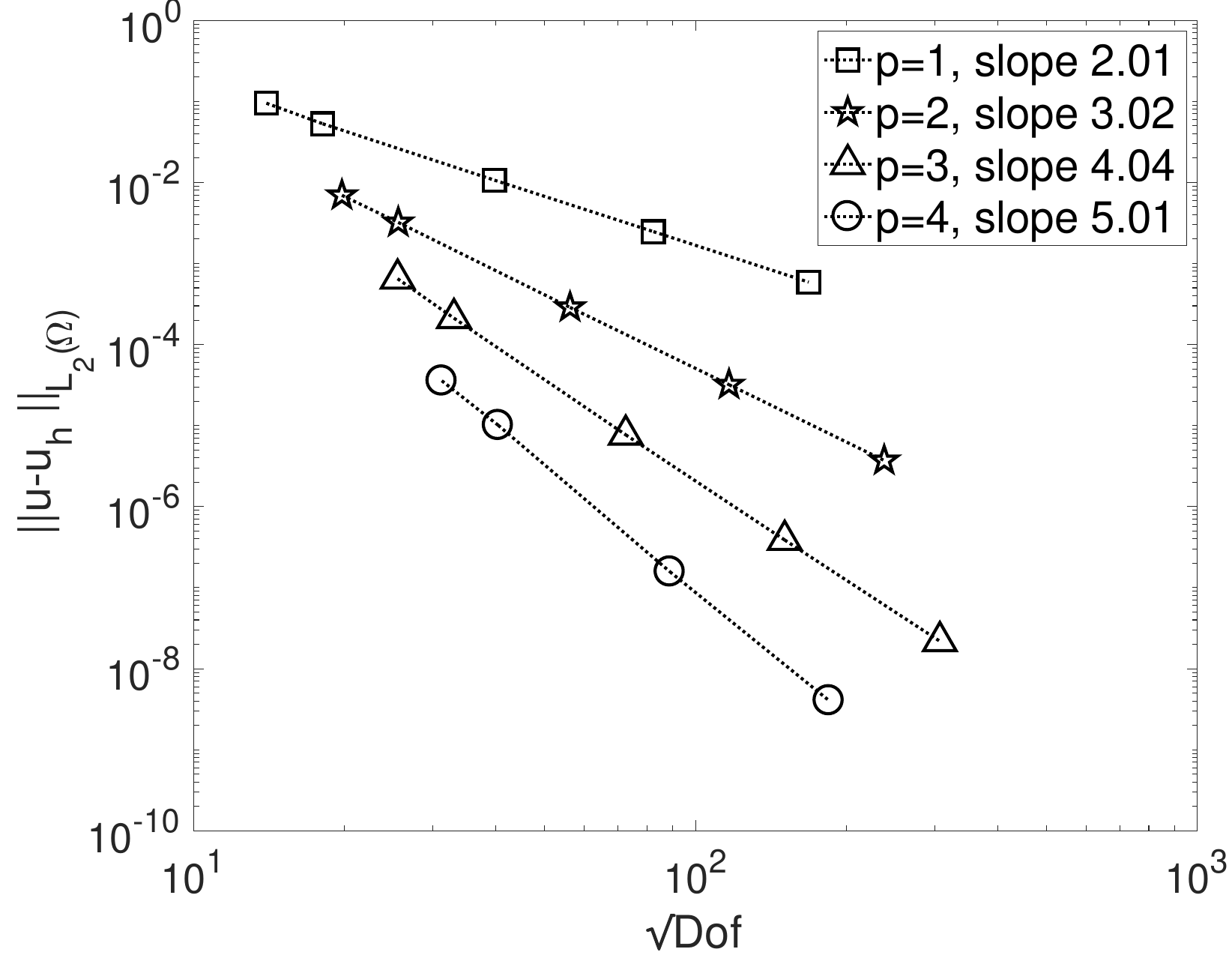}\\
			\includegraphics[height=5.2cm,width=5.8cm]{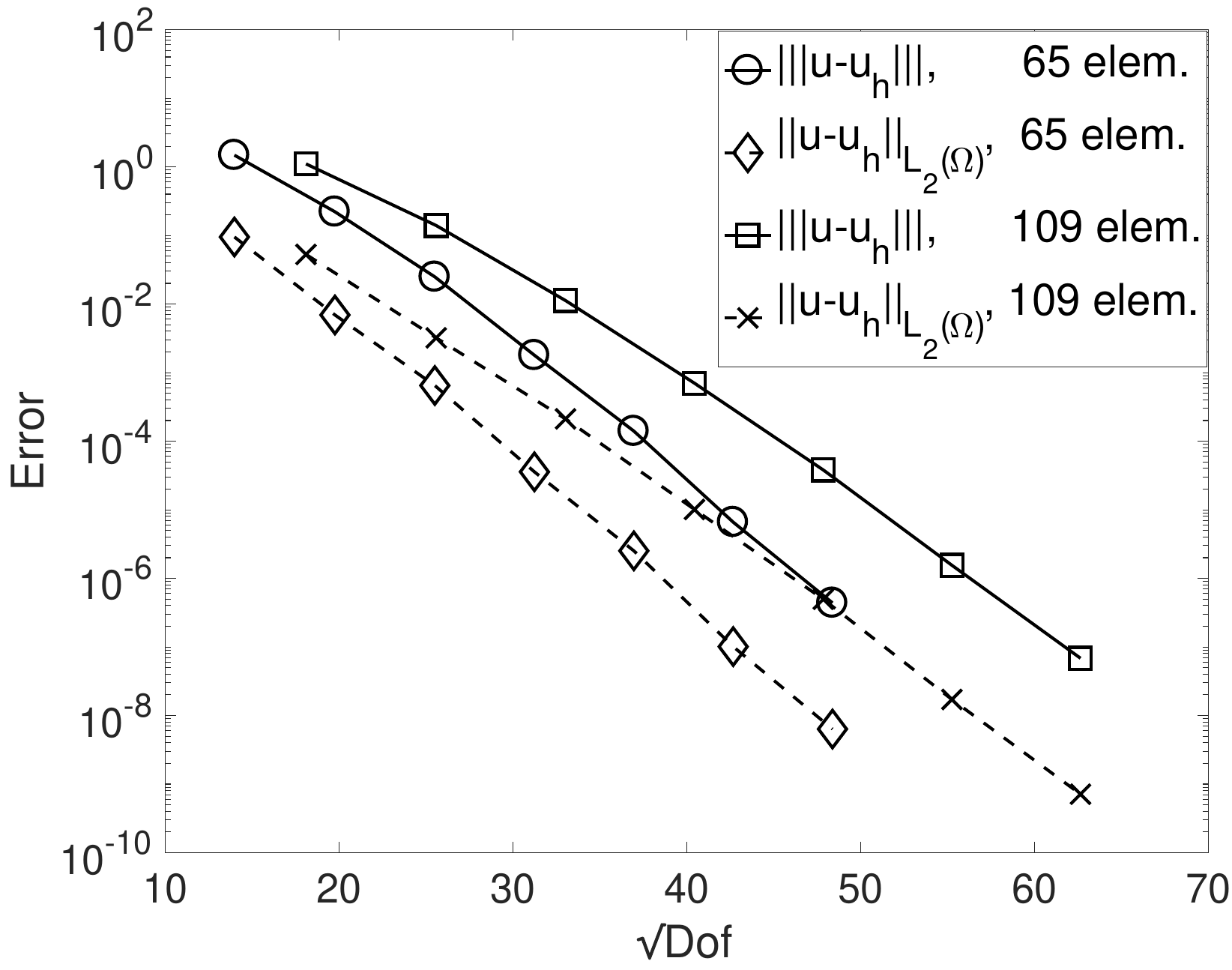}\hspace{.6cm}
			\includegraphics[height=5cm,width=5.2cm]{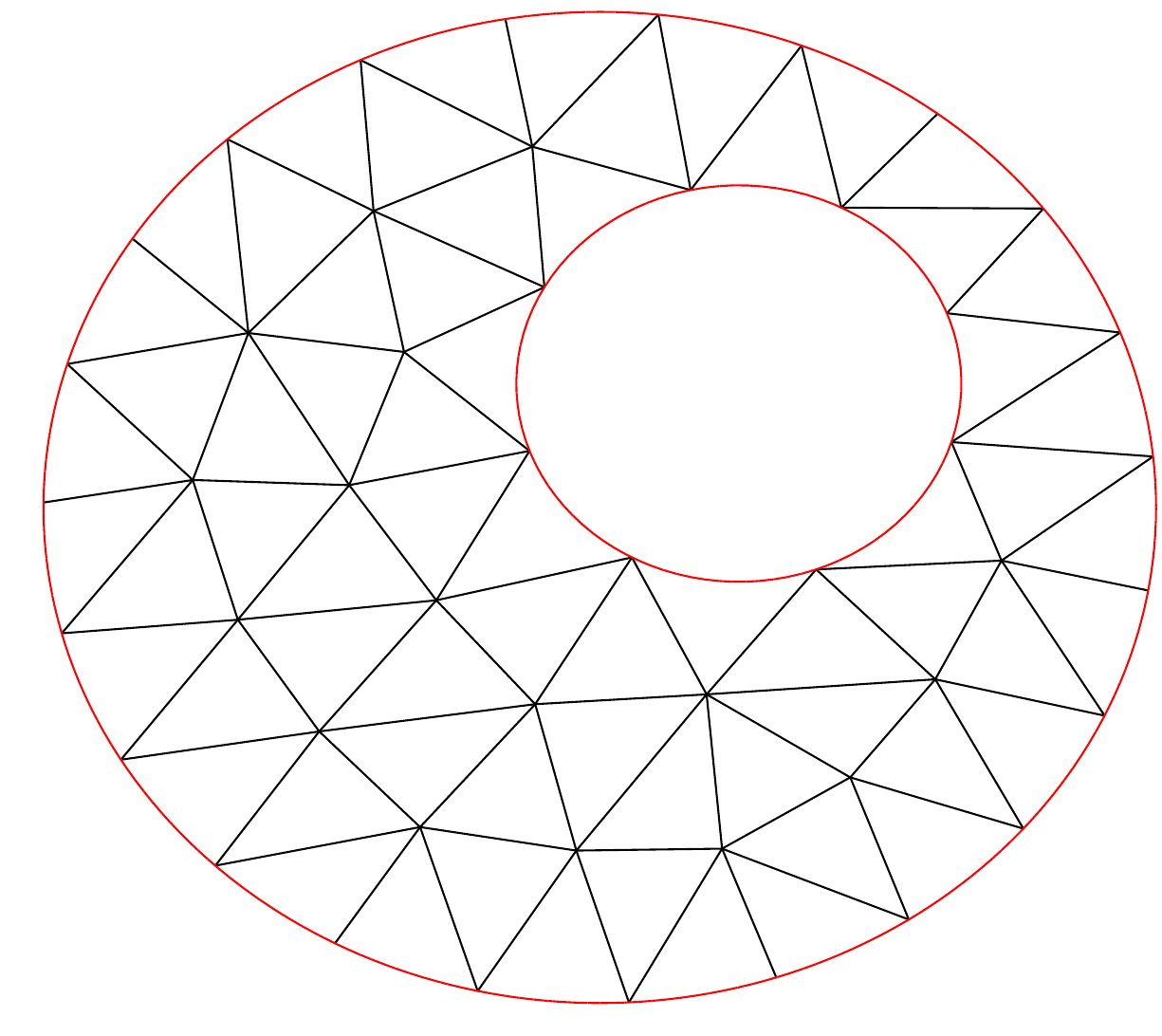}
			\vspace{-2mm}
			\caption{ {Example 1. Top: convergence history for $p=1,2,3,4$ in the $\ndg{\cdot}$ and $\norm{\cdot}{\Omega}$ norms against $\sqrt{{\rm Dof}}$ for the curved triangular mesh exemplified in the bottom (right) plot with $65, 109, 527, 2266$, and $9411$ elements, respectively. Bottom (left):  Convergence history for $p=1,2,\dots,7$ in the $\ndg{\cdot}$ and $\norm{\cdot}{\Omega}$ against $\sqrt{{\rm Dof}}$ for the meshes with $65$ and $109$ elements.
			}}\label{fig:ex3_three_p}
		\end{figure}
		{
			We construct a sequence of domain-fitted curvilinear meshes as follows.
			First, using the mesh generator from~\cite{persson2004simple}, we construct a sequence of meshes approximating the domain $\Omega$ comprising of $65, 109, 527, 2266$, and $9411$ quasi-uniform triangular elements, respectively. The $65$-element mesh is shown in Figure \ref{fig:ex3_three_p}. 
			Then, exploiting the knowledge of the level-set function of $\partial \Omega$, elements containing straight faces approximating the curved boundary are marked. Finally, all marked elements are treated as curved triangular elements with two straight faces and one curved face described by the domain level-set function, thus \emph{capturing the domain exactly.} }
		
		{In Figure \ref{fig:ex3_three_p} (top row), we present the convergence history of $\ndg{u-u_h}$ and $\norm{u-u_h}{\Omega}$ against $\sqrt{{\rm Dof}}$, with ${\rm Dof}$ the number of degrees of freedom on the aforementioned curvilinear meshes with $65, 109, 527, 2266$, and $9411$ elements, for $p=1,2,3,4$, respectively. We clearly observe that, for each fixed $p$,  all errors converge to zero at the optimal rates $\mathcal{O}(h^{p})$ and $\mathcal{O}(h^{p+1})$, respectively, as the mesh size $h$ tends to zero. In the two bottom plots in Figure \ref{fig:ex3_three_p}, we also investigate the convergence history of the dG-EASE solution under $p$-refinement, using the two meshes with $65$ and $109$ curved  elements, respectively, in linear-log scale. Here, we observe exponential convergence of all errors against $\sqrt{{\rm Dof}}$.}


		\subsection{Example 2: convergence study}\label{ex2}  {We now investigate the convergence of dG-EASE on a highly complex mesh comprising of elements arising from agglomeration of a very fine background mesh, which also contains curved boundary-touching elements. Set $a=\epsilon I_{2\times 2}$ and $\epsilon=0.01$, $\mathbf{b} = (1-x_2,1-x_1)^\top$, $c=2$ and $f$ so that $u(x_1,x_2)=\sin(\pi x_1)\sin(\pi x_2)$} in \eqref{pde} for $d=2$, on a domain $\Omega\approx (0,1)^2$ enclosed by a piecewise curved sinusoidal boundary; we refer to Figure \ref{fig:ex1_one} for an illustration. We impose non-homogeneous Dirichlet boundary conditions on $\partial \Omega$.
		
		\begin{figure}\hspace{-.2cm}
			\includegraphics[height=5cm,width=5.4cm]{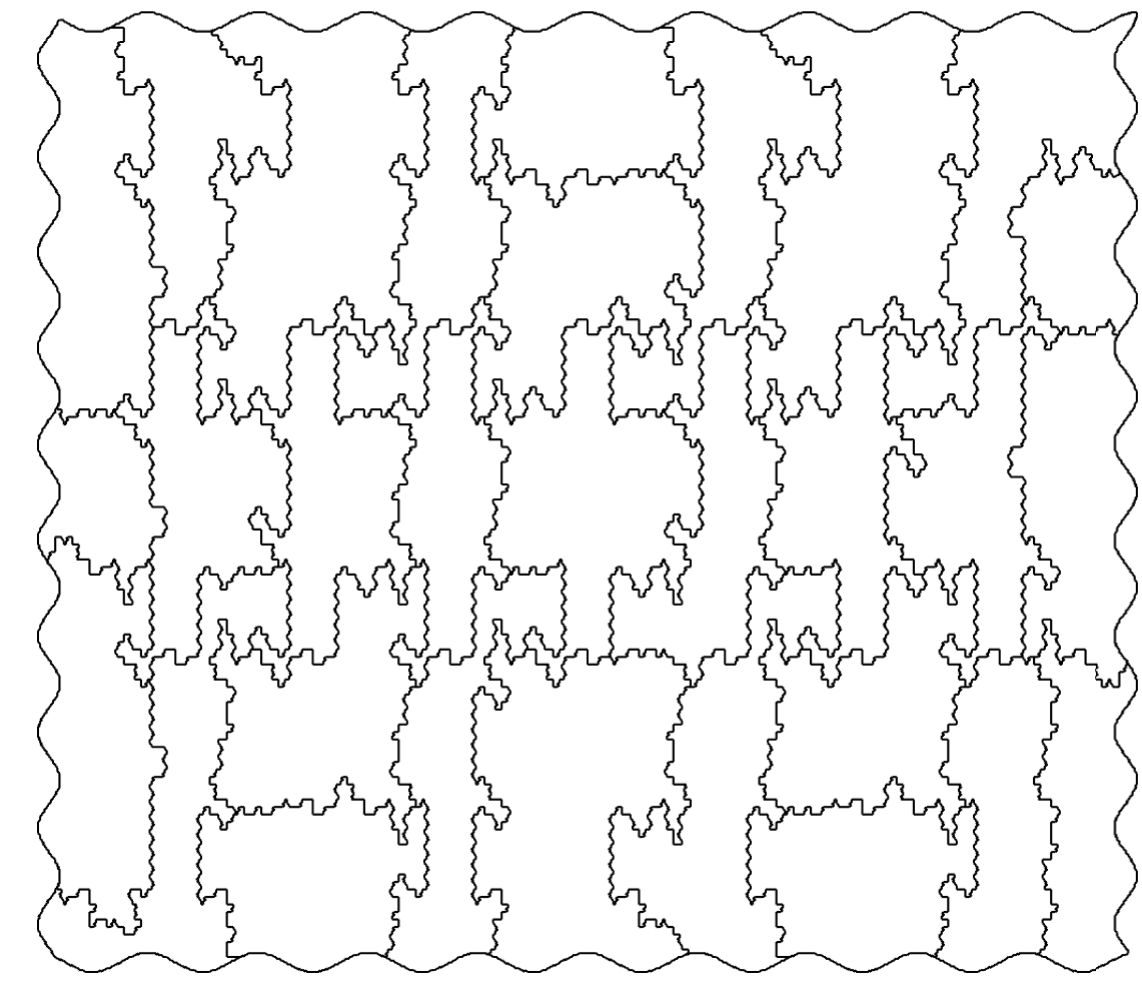}
			\includegraphics[height=5cm,width=5.4cm]{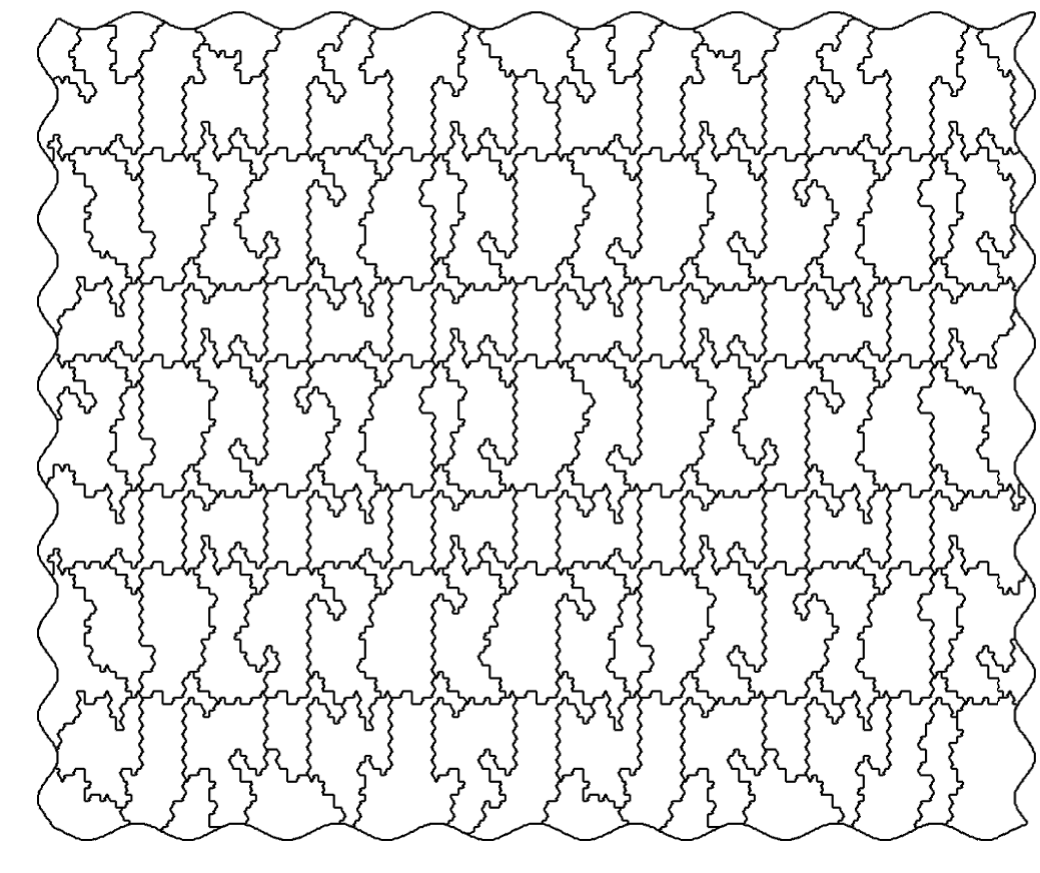}
			\vspace{-2mm}
			\caption{Example  {2}. The computational domain $\Omega$ and two meshes with $30$ and $132$ elements, respectively. }\label{fig:ex1_one}
		\end{figure}
		
		The mesh is constructed as follows. An initial  {curved mesh, fitted to the sinusoidal boundary via the level set approach described above,} is subdivided into a very fine background subdivision consisting of approximately $500$K sub-elements. The latter is, in turn, agglomerated into $30,132,555,2151,8337$ curved/polygonal elements using a  standard mesh partitioning software. The parameters chosen in the partitioning software have been selected to yield a high-frequency  {`sawtooth'} vertical boundary for many of the agglomerated elements. We refer to Figure \ref{fig:ex1_one} for an illustration of the resulting meshes with $30$ and $132$ agglomerated elements.
		
		In Figure \ref{fig:ex1_two}, the convergence history for $p=1,2,3,4$ for the  {errors $\ndg{u-u_h}$ and $\norm{u-u_h}{\Omega}$} against $\sqrt{{\rm Dof}}$ is presented for the aforementioned agglomerated meshes with $30,132,555,2151,8337$ elements. Here, we clearly observe that,  for each fixed $p$,  the dG- and $L_2(\Omega)$-norm  errors converge to zero at the optimal rates  {$\mathcal{O}(h^{p})$ and $\mathcal{O}(h^{p+1})$}, respectively, as the mesh size $h$ tends to zero.  {Further, we report also the error in the stronger `streamline-diffusion' norm $\nsdg{u-u_h}$ in Figure \ref{fig:ex1_two}; here we have chosen $\lambda_K=\mathcal{O}(\rho_K/p_K^2)$.  For each fixed $p$,  the errors converge to zero at the optimal rates $\mathcal{O}(h^{p})$, as the mesh size $h$ tends to zero.}

		Finally, in Figure \ref{fig:ex1_two} (bottom-right),  we also investigate the convergence history of the dG-EASE solution under $p$-refinement, using  {the mesh with $132$} elements shown in Figure~\ref{fig:ex1_one}(right). Here, we observe exponential convergence of  {the three} norm errors against  {$\sqrt{{\rm Dof}}$}.  {Interestingly, we observe that the difference between the errors $\ndg{u-u_h}$ and $\nsdg{u-u_h}$ is insignificant.}

		\begin{figure}
			\includegraphics[height=5cm,width=5.5cm]{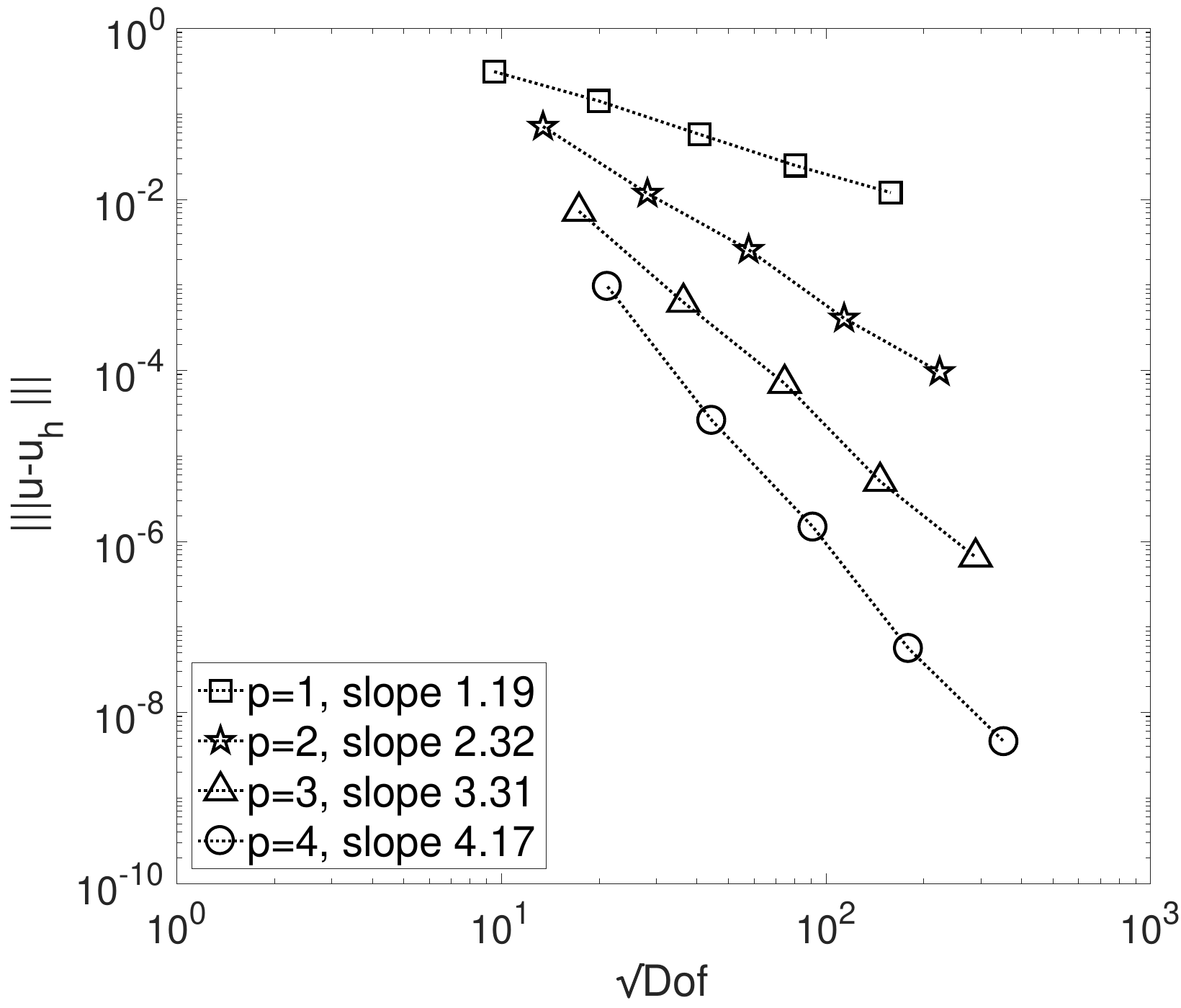}
			\includegraphics[height=5cm,width=5.5cm]{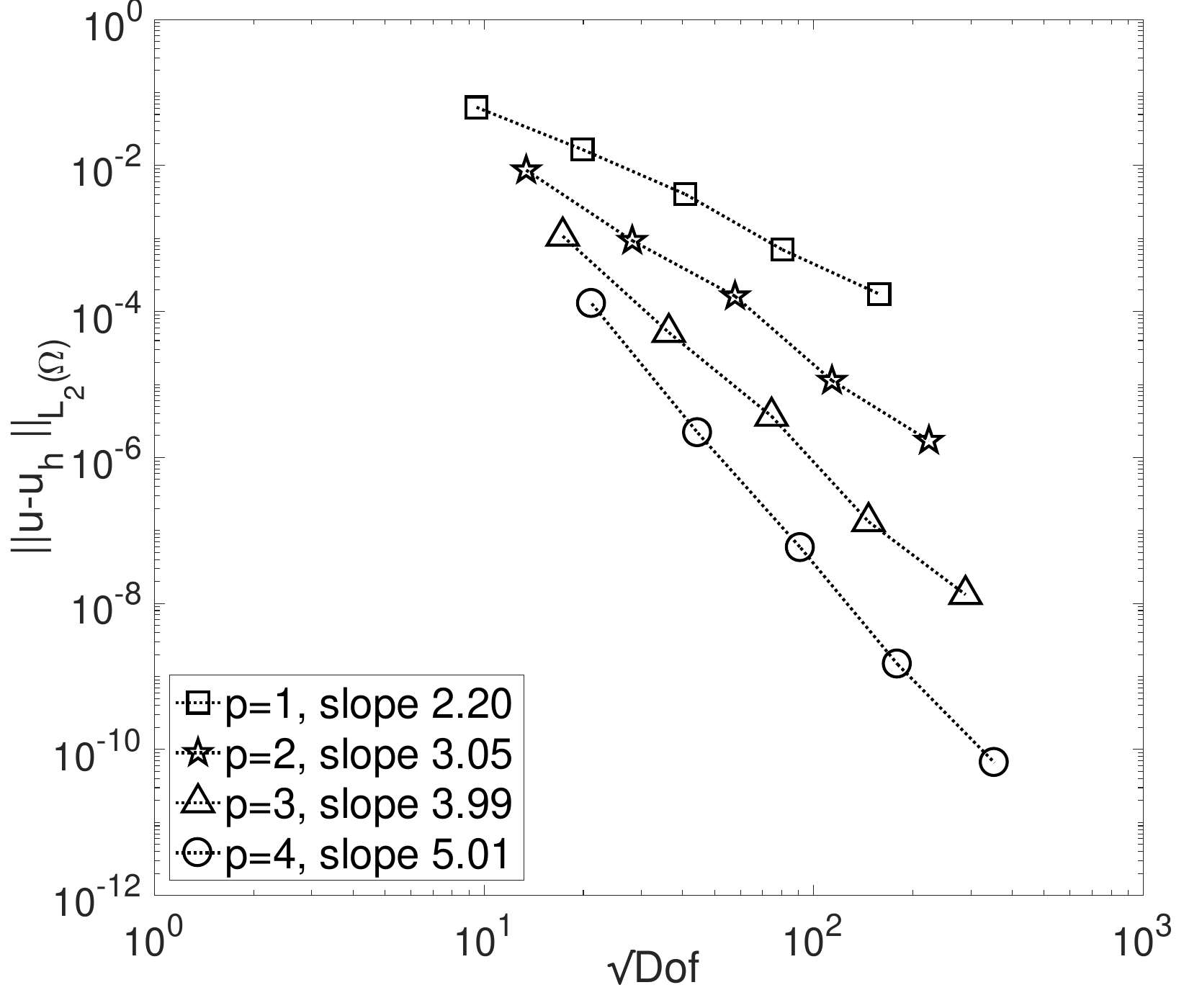}\hspace{5mm}\\
			\includegraphics[height=5cm,width=5.5cm]{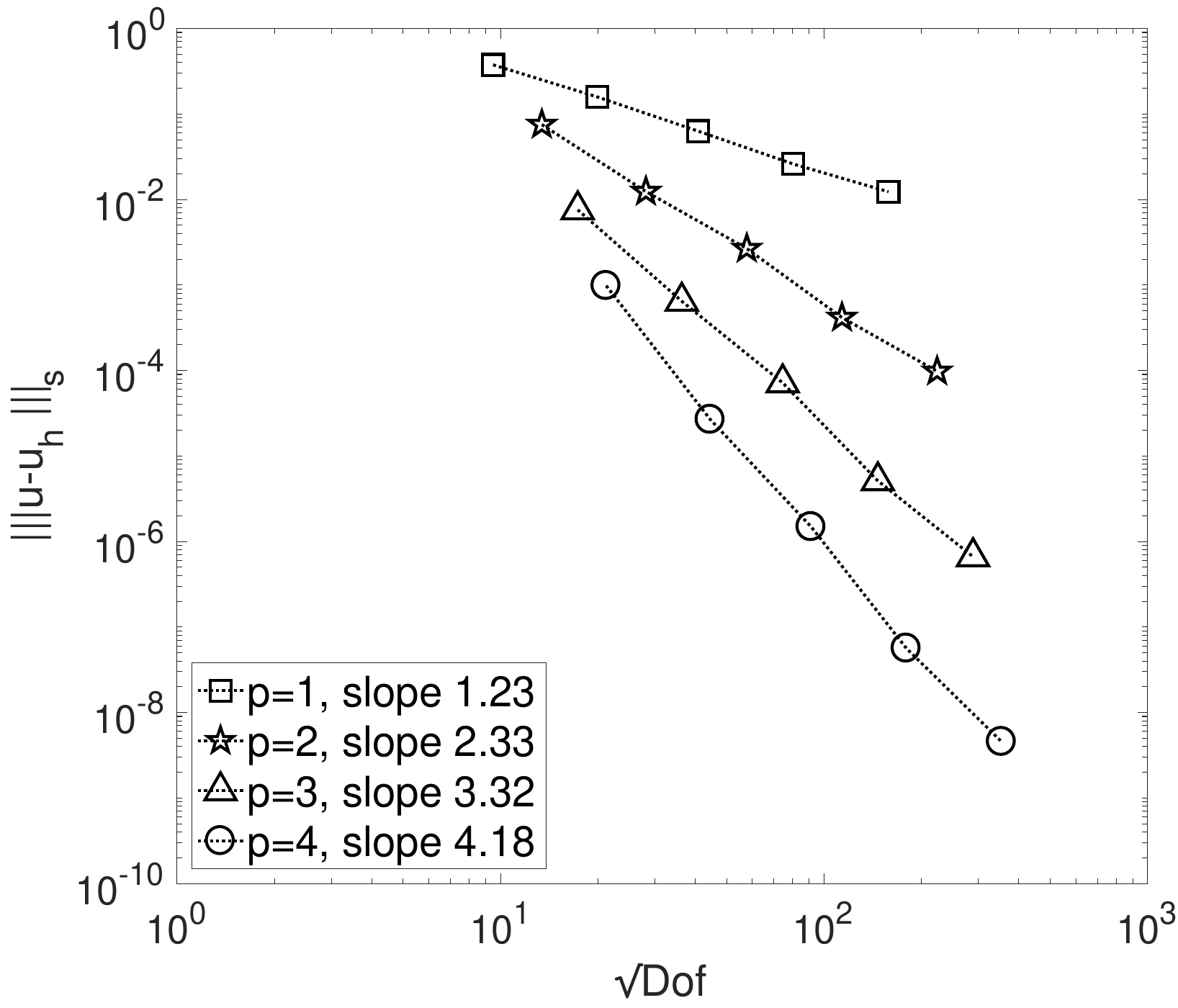}
			\includegraphics[height=5cm,width=5.5cm]{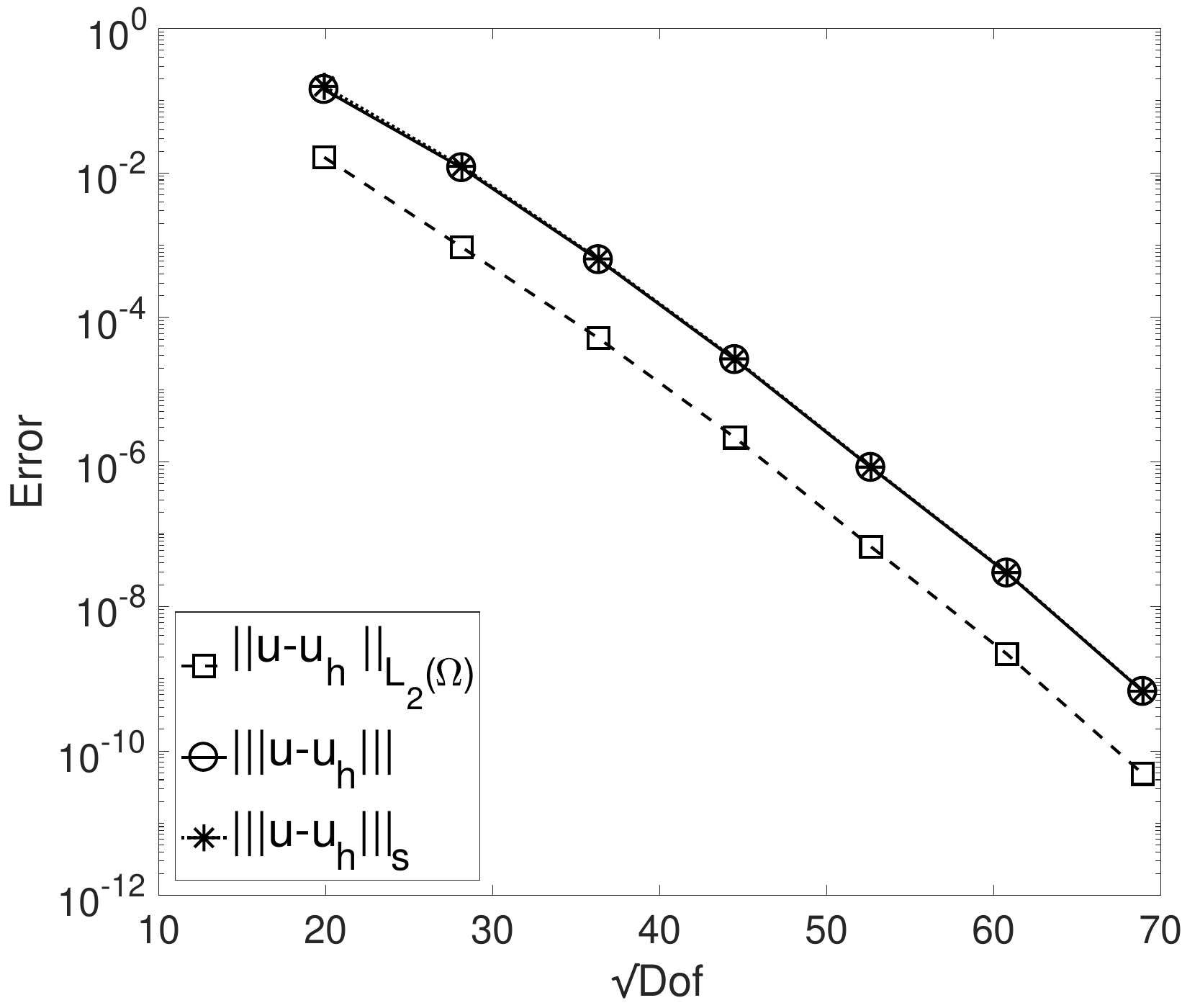}\hspace{5mm}
			\vspace{-2mm}
			\caption{ {Example 2. Convergence history for $p=1,2,3,4$, in $\ndg{\cdot}$ (top-left), $\norm{\cdot}{\Omega}$(top-right) and $\nsdg{\cdot}$(bottom-left) against $\sqrt{{\rm Dof}}$ for the meshes exemplified in Figure \ref{fig:ex1_one} with $30,132,555,2151,8337$ elements, respectively. Bottom right: convergence history for $p=1,2,\dots,7$ in the three norms against $\sqrt{{\rm Dof}}$ for the mesh with $132$ elements shown in Figure~\ref{fig:ex1_one} (right).}
			}\label{fig:ex1_two}
		\end{figure}

		\subsection{Example 3: stability study}  We  {continue by assessing} the stability of the  {dG-EASE method}
		for convection-diffusion problems in the presence of
		unresolved lower-dimensional sharp solution layers. To this end, for $d=2$, we set $a=\epsilon I_{2\times 2}$ and $\epsilon=10^{-4}$, $\mathbf{b} = (1,1)^\top$, $c=0$ and $f=1$ in \eqref{pde}. We solve this problem on a variant of the domain $\Omega$ from Example  {2} above, in which circular internal pieces of the domain of various radii have been removed; we refer to Figure \ref{fig:ex2_one}(left) for an illustration of the domain and sample mesh of essentially arbitrarily-shaped elements obtained using a completely analogous construction to that used in Example  {2}. We close the problem by prescribing homogeneous Dirichlet boundary conditions on $\partial\Omega$ (i.e., including the internal boundaries at the holes).  We expect strong exponential boundary layers on the top and right portions of the curved boundary, as well as variable intensity layers at the outflow portions of the internal hole boundaries.
		
		\begin{figure}
			\includegraphics[height=5cm,width=5.5cm]{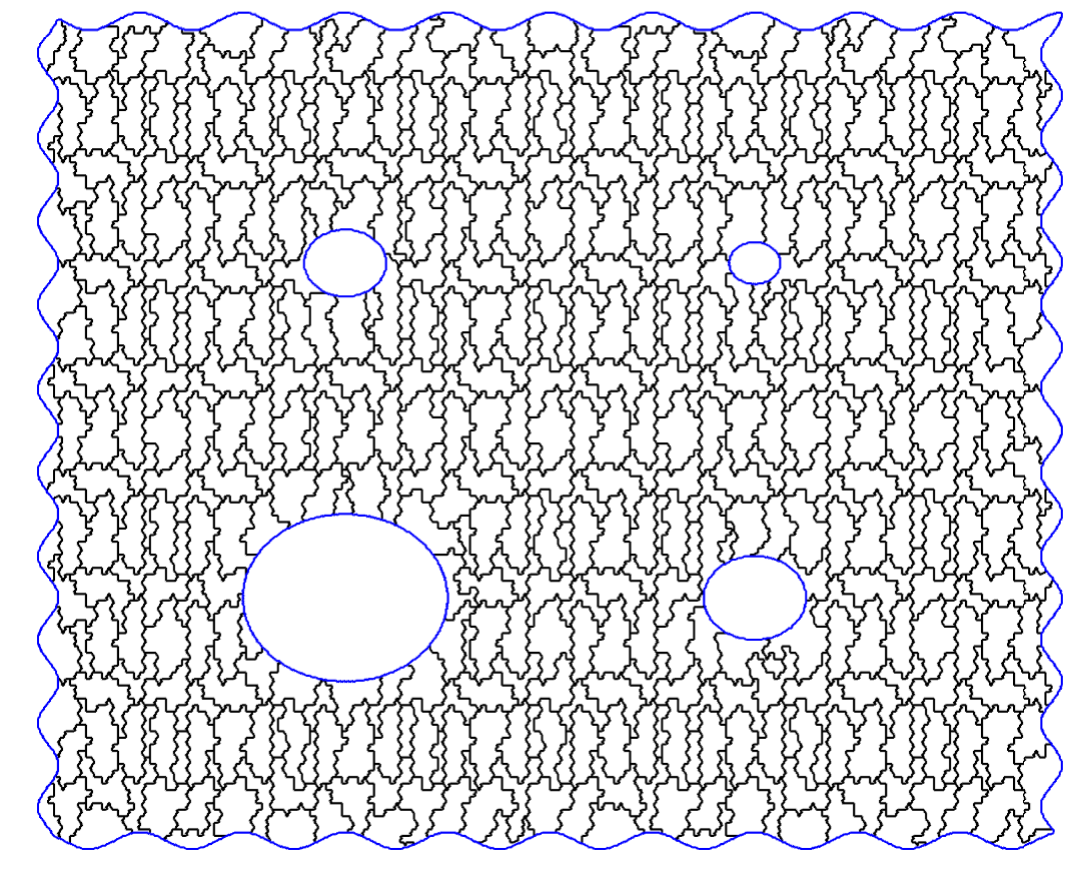}
			\includegraphics[height=5cm,width=5.8cm]{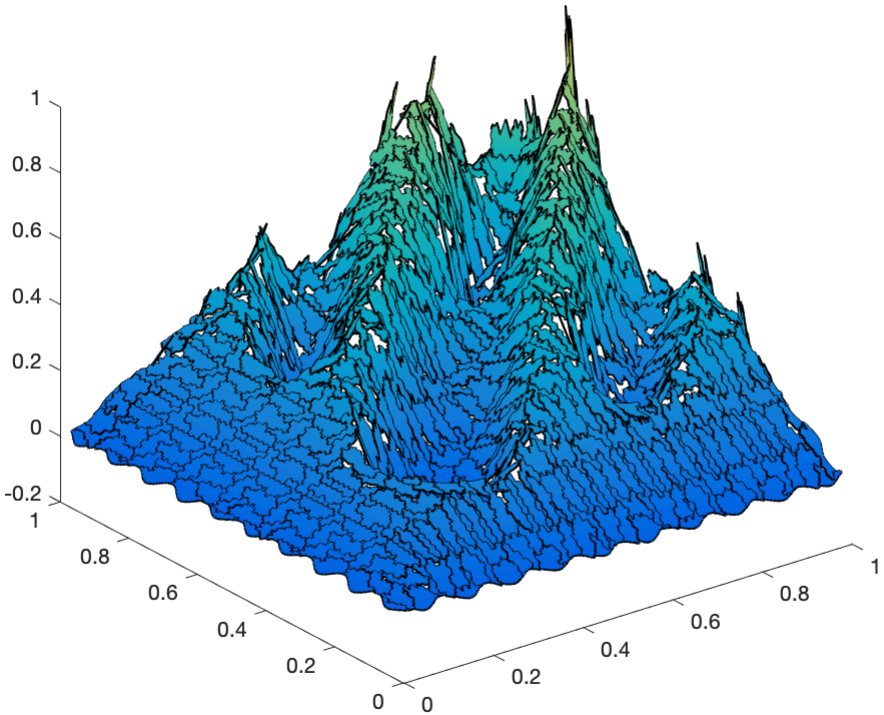}
			\vspace{-2mm}
			\caption{Example  {3}. Domain with holes and $531$-element mesh (left). Corresponding solution obtained with \Rev{$p=1$} (right).}\label{fig:ex2_one}
		\end{figure}
		In Figure \ref{fig:ex2_one} (right), we  {provide} the dG-EASE solution using  $p=1$ and the mesh of $531$ elements shown on the left plot. This mesh is not fine enough to resolve the singularly perturbed behaviour in the vicinity of the outflow portions of the boundary. Nevertheless, the dG-EASE method provides a stable discretization with very localized, expected, oscillatory behaviour at the vicinity of the outflow boundary. The stable behaviour of dG-EASE with respect to the size of the P\'eclet number $Pe:=\|\mathbf{b}\|/\epsilon$ is expected due to the upwind flux used in $B_{\rm ar}(\cdot,\cdot)$; nonetheless, to the best of our knowledge, its performance in the context of elements with such geometrical shape generality has not been tested before in the literature. To highlight the behaviour of the method  {on different meshes}, we  {report} the dG-EASE solution, obtained with meshes  composed of $129$ and $2048$ linear elements in Figure \ref{fig:ex2_two} (top). In both cases the mesh is not sufficiently fine to resolve the exponential boundary layer behaviour, while the finer mesh with $2048$ linear elements sufficiently resolves the parabolic layers initiated at the holes.  
		
		
		{Finally, we test the hyperbolic limit case by setting $\epsilon=0$. The DG-EASE solution, shown in Figure \ref{fig:ex2_two} (bottom), remains stable and there is no oscillation around the outflow boundaries, as expected by a stabilised method.}
		
		
		\begin{figure}
			\includegraphics[height=5.2cm,width=6cm]{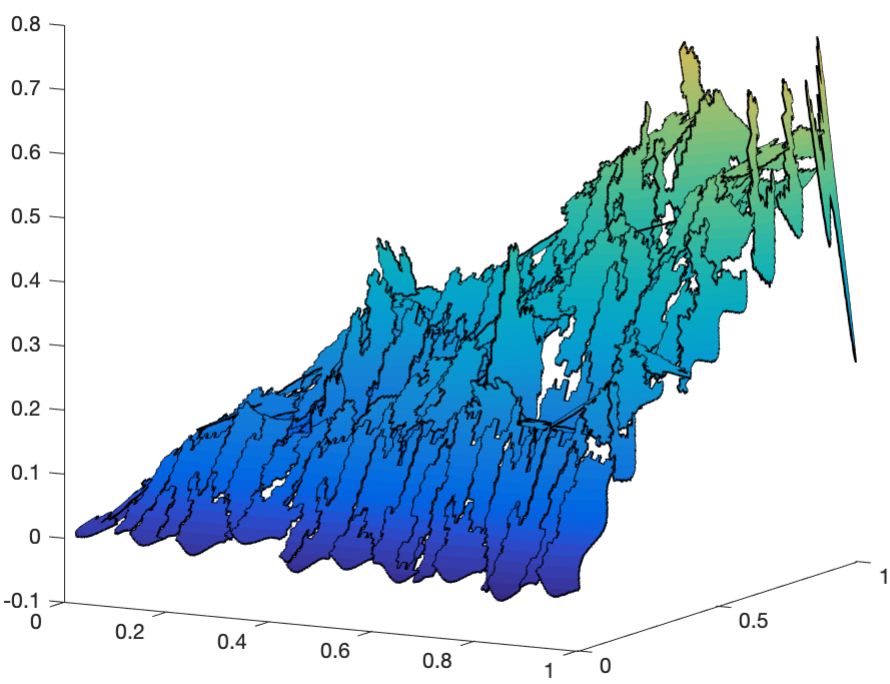}\vspace{0cm}
			\includegraphics[height=5.2cm,width=6cm]{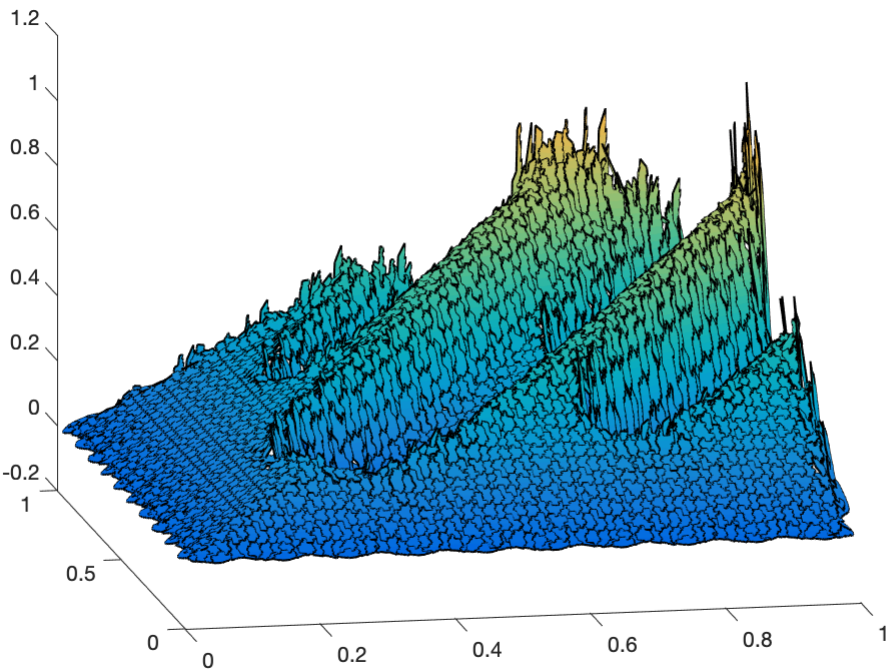}
			\includegraphics[height=5cm,width=6cm]{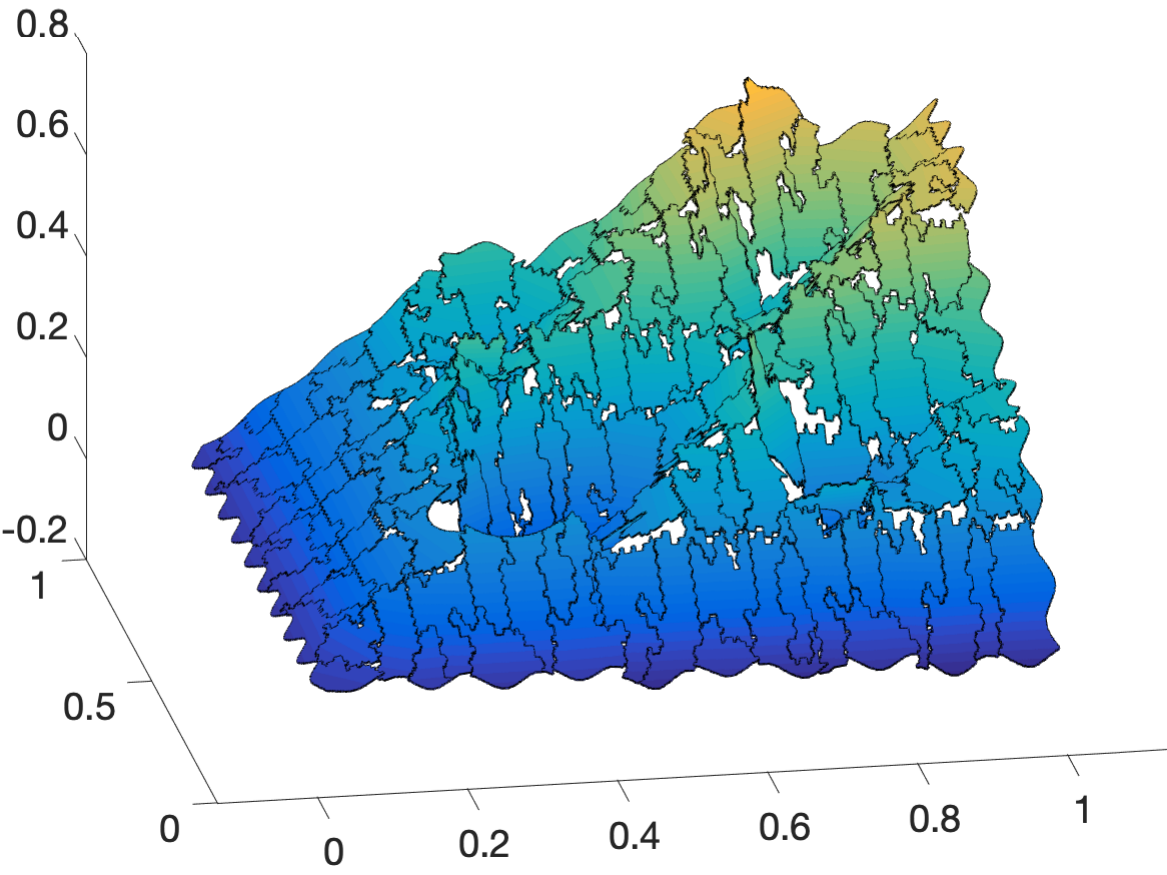}
			\includegraphics[height=5cm,width=6.2cm]{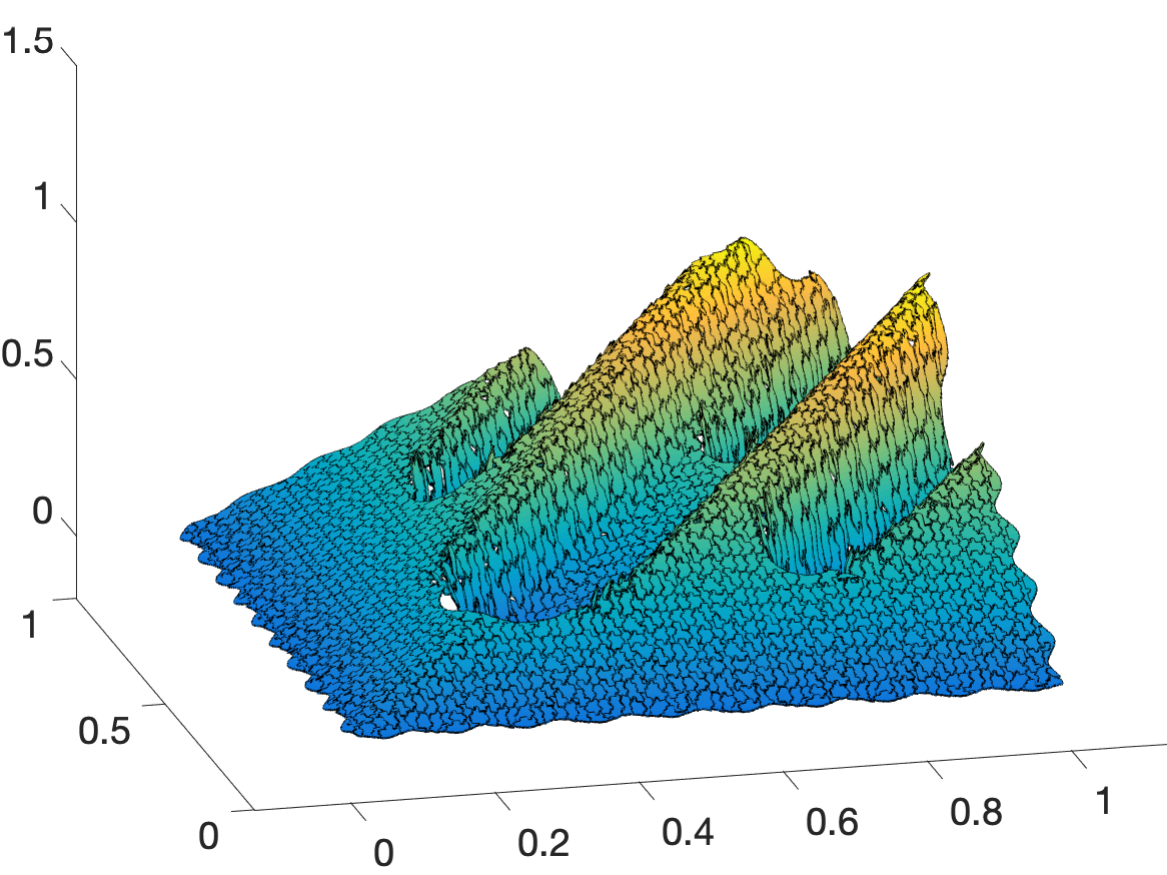}
			\vspace{-2mm}
			\caption{\Rev{Example  {3. Solutions computed for $\epsilon=10^{-4}$ (top) and  $\epsilon=0$ (bottom) using linear elements. The meshes are composed of $129$ (left) and  $2048$ (right) elements with compex shapes.}}}
			\label{fig:ex2_two}
		\end{figure}

		
		\subsection{ {Example 4: changing type PDE across a curved interface}}
		
		{To highlight a number of attractive features of the dG-EASE approach, we consider a coupled parabolic-hyperbolic partial differential equation, whose type changes across a sinusoidal interface $\Gamma$. Let $\Omega =\Omega_1\cup\Omega_2\cup\Gamma$ with
			\[
			\begin{array}{l}
				\Omega_1=\{(x_1,x_2):-1 \leq x_1 \leq 1,\  A \sin(\omega\pi x_1)\le x_2\le1 \},\\
				\Omega_2=\{(x_1,x_2):-1 \leq x_1 \leq 1,\  -1\le x_2\le A \sin(\omega\pi x_1)\},
			\end{array}
			\]
			for $A,\omega>0$ whose precise values will be given below; we refer to Figure \ref{fig:ex4_one} for an illustration. On this geometrical setting, we consider the problem:
			\begin{equation*}\label{eq:change}
				\left\{
				\begin{array}{ll}
					-x_1^2u_{x_2^{}x_2^{}} + u_{x_1^{}}+A \omega \pi \cos (\omega \pi x_1^{})u_{x_2^{}}+u = 0, &\quad \text{in } \Omega_1, \vspace{2mm}
					\\
					u_{x_1^{}}+A\omega\pi\cos(\omega\pi x_1^{})u_{x_2^{}}+u = 0, &\quad \text{in } 
					\Omega_2,
				\end{array}
				\right.
			\end{equation*}
			coupled with inflow and Dirichlet boundary conditions, so that the analytical solution is given by
			\begin{equation*}\label{eq:change-sol}
				u(x_1,x_2) =
				\begin{cases}
					\sin (\frac{\pi}{2}(1+x_2-A\sin(\omega \pi x_1))) \exp(-(x_1+\frac{\pi^2x_1^3 }{12})),& \text{in }
					\Omega_1,
					\vspace{2mm}
					\\
					\sin (\frac{\pi}{2} (1+x_2-A\sin(\omega \pi x_1))) \exp(-x_1). &\text{in } 
					\Omega_2.
				\end{cases}
			\end{equation*}
			This problem is hyperbolic when $x_2^{}\leq  A\sin(\omega \pi x_1^{})$, $x_1^{}\in(-1,1)$, and parabolic otherwise. 
			The normal flux of the exact solution is continuous  across the interface $\Gamma$ with equation $x_2^{}=A\sin(\omega\pi x_1^{})$, while the solution itself has a discontinuity across the interface.  This problem is a variant of an example from \cite{thesis,cangiani2015hp}.  As such, there is no discontinuity penalisation imposed at the interface $\Gamma$. \Rev{Moreover, we point out that  ${\bf b}\cdot {\bf n}=0$ at the interface in this example. 
			}
		}
		
		{Our aim is to highlight the performance of dG-EASE of arbitrary order, when the mesh is fitted with respect to a discontinuity of the exact solution. To that end, we focus on $p$-version convergence, using $64$ rectangular elements with curved faces exactly fitting the interface; we refer to Figure \ref{fig:ex4_one} for an illustration with $A=0.025$, and $\omega = 8$ and $\omega = 16$, respectively. }

		{Interestingly, the mesh is \emph{not} aligned with the inflow and outflow parts of the boundary $\partial\Omega$. This is due to the oscillating coefficient of the first order term. In Figure \ref{fig:ex4_one}, the inflow parts of the boundary are marked in red; on these parts, inflow boundary conditions are imposed. Correspondingly, this pattern continues in the internal element faces in which the inflow parts of $\partial_- K$ are also not aligned with the faces. As such, face integral terms in the dG method may be computed only on parts of a face of  a rectangular element. Nonetheless, the method is able to cope unaltered with this complication. The quadrature is implemented in the composite fashion described in Example 1 above.}

		\begin{figure}\hspace{-.2cm}
			\includegraphics[height=4.2cm,width=4.8cm]{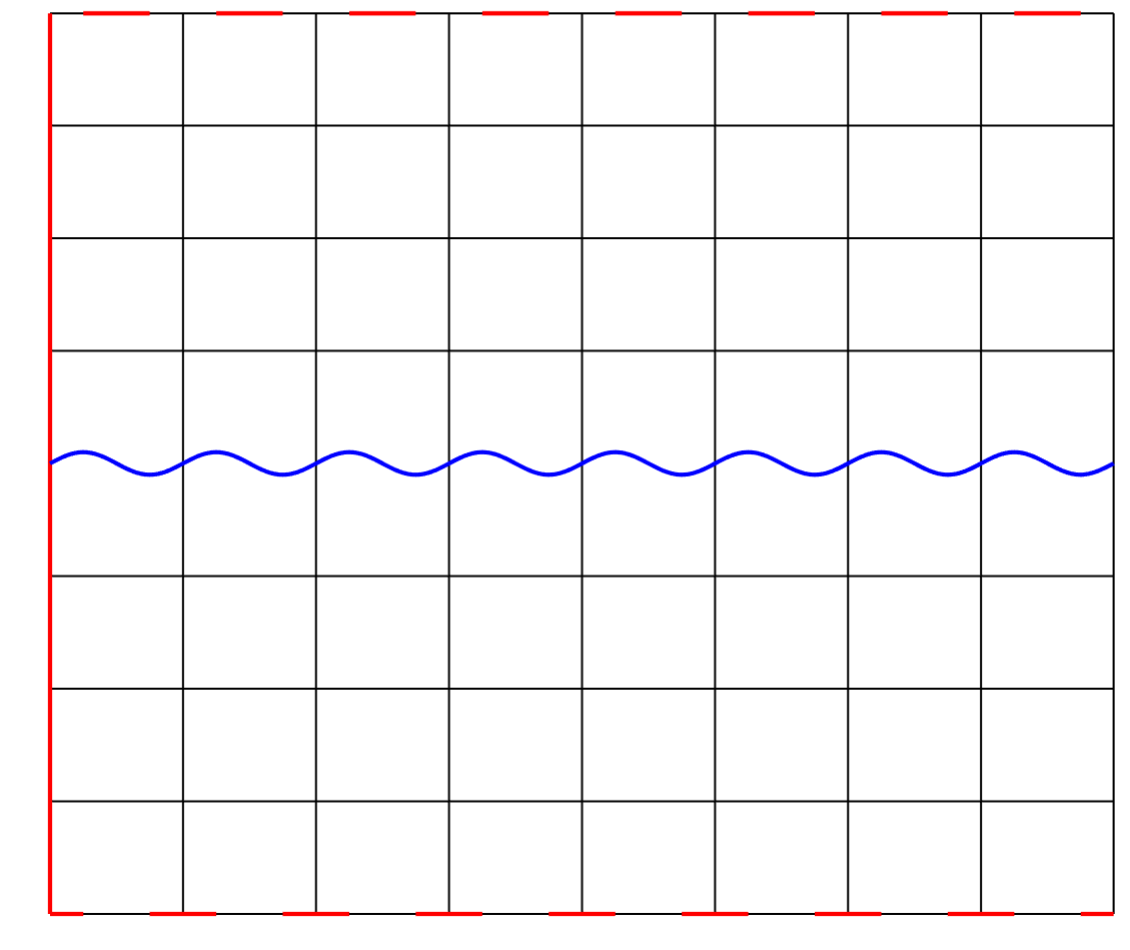}\hspace{.4cm}
			\includegraphics[height=4.2cm,width=4.8cm]{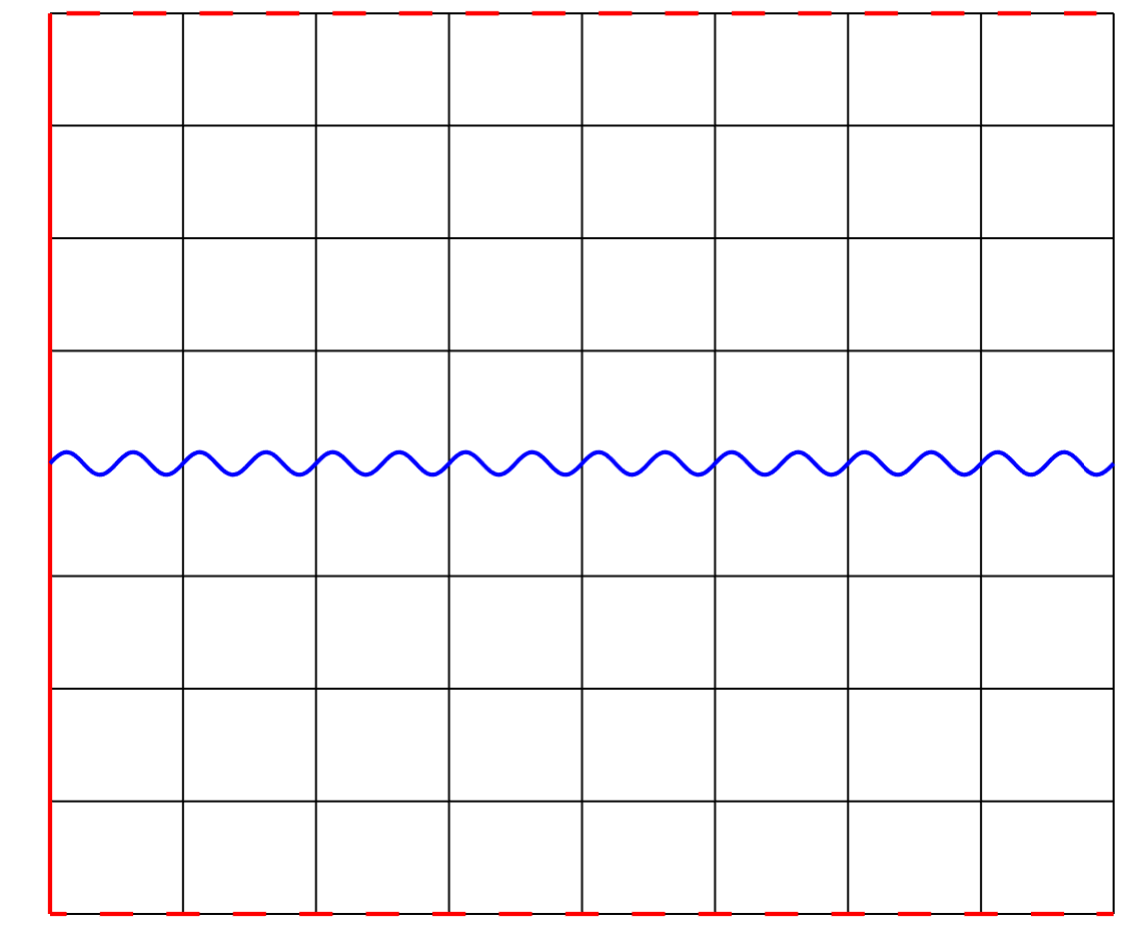}
			\vspace{-2mm}
			\caption{ {Example 4. The domain $\Omega$ with the $64$ element mesh fitted to the interface (blue) for $A=0.025$, and $\omega = 8$ (left) and $\omega = 16$ (right). The inflow parts of the boundary are drawn in red.  }}\label{fig:ex4_one}
		\end{figure}
		
		{We begin by setting $A=0.025$ and $\omega =8$. In Figure~\ref{fig:ex4_two}, we record the convergence history of the dG-EASE solution under $p$-refinement, using the mesh shown in Figure~\ref{fig:ex4_one} (left) and $p=1,\dots,13$. Although the elements are perfectly aligned with the interface $\Gamma$, the mesh is still coarse: each element includes roughly one full oscillation of the wind ${\bf b}$.  Still we observe exponential convergence of $\ndg{u-u_h}$ and $\|u-u_h\|_{\Omega}$ errors against $\sqrt{{\rm Dof}}$ under $p$-refinement. This result reinforces the claim that dG-EASE on perfectly aligned meshes with appropriate quadrature rules can lead to spectral accuracy. In contrast, if the mesh is not aligned exactly with the solution's discontinuity, the error is only expected to decay at an algebraic rate, according to standard best approximation results.}
		
		{Next, we set $\omega =16$ and we record the convergence history under $p$-refinement, for $p=1,\dots,17$, for the fixed mesh from Figure~\ref{fig:ex4_one}(right). Here $64$ elements constitute a very coarse mesh as, at the interface,  there are now two full oscillations of the wind ${\bf b}$ per element. Again, we observe exponential convergence of $\ndg{u-u_h}$ and $\|u-u_h\|_{\Omega}$  against $\sqrt{{\rm Dof}}$ under $p$-refinement, after an initial plateau for $p\le 5$. This is expected as the dG-EASE approach is not designed as a multiscale framework. Nevertheless, for $p\ge 6$, exponential convergence is observed.}
		
		\begin{figure}[t]
			\includegraphics[height=4.8cm,width=5.8cm]{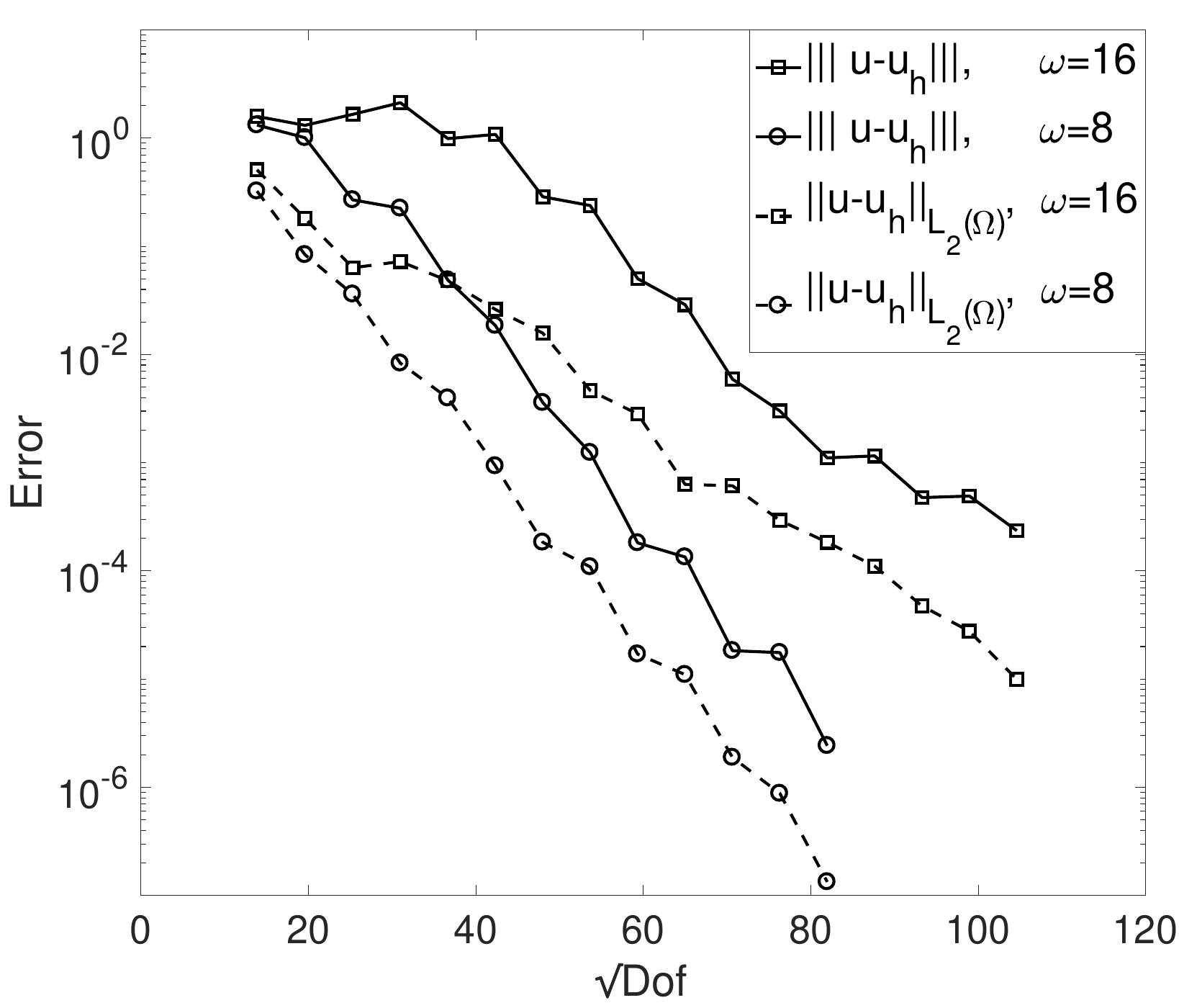} \hspace{1cm}
			\vspace{-2mm}
			\caption{ {Example 4. Convergence history for $p=1,2,\dots$ in the $\ndg{\cdot}$ and $\norm{\cdot}{\Omega}$ against $\sqrt{{\rm Dof}}$ for the curved triangular mesh with $64$ elements shown in Figure~\ref{fig:ex4_one}, for $A =0.025$, $\omega=8,16$.}} \label{fig:ex4_two}
		\end{figure}

		\section{Acknowledgements} 
		
		{
			We are grateful to the anonymous referees and to the editor for their constructive comments which helped to improve this work substantially. AC gratefully acknowledges support from the MRC (MR/T017988/1), ZD from IACM-FORTH, Greece, and  EHG from The Leverhulme Trust (RPG-2015- 306).  This work was supported by the Hellenic Foundation for Research and Innovation (H.F.R.I.) under the ``First Call for H.F.R.I. Research Projects to support Faculty members and Researchers and the procurement of high-cost research equipment grant'' (Proj. no. 3270).}

		\bibliographystyle{siam} 

		\bibliography{DG-EASE_Revision} 
	\end{document}